\documentclass[12pt,reqno]{amsart}
\usepackage{amsthm}
\usepackage{cases}
\usepackage{amscd}
\usepackage{amsfonts}

\usepackage{amssymb}
\usepackage{amsmath}
\usepackage[dvipdfmx]{graphicx}
\usepackage{booktabs}
\usepackage{placeins} 

\usepackage{subfigure}
\usepackage{algorithm}
\usepackage{algpseudocode} 
\renewcommand{\algorithmicrequire}

\usepackage{color}
\usepackage{datetime}

\allowdisplaybreaks

\makeatletter
\newcommand\figcaption{\def\@captype{figure}\caption}
\newcommand\tabcaption{\def\@captype{table}\caption}
\makeatother

\usepackage{geometry}
\usepackage{algorithm}
\usepackage{multirow}
\usepackage{mathrsfs}

\numberwithin{equation}{section}
\newtheorem{sch}{Scheme}[section]
\newtheorem{example}{Example}[section]
\newtheorem{remark}{Remark}[section]
\newtheorem{theorem}{Theorem}[section]
\newcommand{\beq}{\begin{equation}}
\newcommand{\eeq}{\end{equation}}
\newcommand{\bea}{\begin{eqnarray}}
\newcommand{\eea}{\end{eqnarray}}
\newcommand{\beas}{\begin{eqnarray*}}
\newcommand{\eeas}{\end{eqnarray*}}

\newcommand{\uu}{\mathrm{\textbf{u}}}
\newcommand{\dd}{\mathrm{\textbf{d}}}
\newcommand{\ww}{\mathrm{\textbf{w}}}

\newcommand{\vv}{\mathrm{\textbf{v}}}
\newcommand{\bx}{\mathrm{\textbf{x}}}
\newcommand{\nn}{\textbf{n}}

\usepackage[hidelinks]{hyperref}
\usepackage{cleveref}

\begin{document}
\title[PCSAV method for the Ericksen-Leslie model]{A linear, unconditionally stable, second order decoupled method for the Ericksen-Leslie model with SAV approach}

\author[R. Cao and N. Yi]{Ruonan Cao$^\dagger$ and Nianyu Yi$^{\ddagger,\S}$}
\address{$^\dagger$ School of Mathematics and Computational Science, Xiangtan University, Xiangtan 411105, P.R.China }
\email{ruonan9cao@163.com}
\address{$^\ddagger$ Hunan Key Laboratory for Computation and Simulation in Science and Engineering; School of Mathematics and Computational Science, Xiangtan University, Xiangtan 411105, P.R.China}
\email{yinianyu@xtu.edu.cn}



\begin{abstract}
In this paper, we present a second order, linear, fully decoupled, and unconditionally energy stable scheme for solving the Erickson-Leslie model.
This approach integrates the pressure correction method with a scalar auxiliary variable technique. We rigorously demonstrate the unconditional energy stability of the proposed scheme.
Furthermore, we present several numerical experiments to validate its convergence order, stability, and computational efficiency. 
\end{abstract}

\maketitle


\section{Introduction}
Liquid crystals represent the fourth state of matter, distinct from gases, liquids, and solids. 
There are four types of liquid crystals:nematic, smectic, cholesteric, and discotic, differentiated by molecular alignments.
Furthermore, they also can be categorized into thermotropic or lyotropic depending on their production conditions. Nematic liquid crystals are a common type of the thermotropic class.

Ericksen \cite{18Ericksen,19Ericksen} and Leslie \cite{36Leslie,4Leslie} developed the hydrodynamic theory of the nematic liquid crystals. They proposed the Ericksen-Leslie model, which consists of an convective harmonic map heat flow equation for the evolution of the director field coupled with an incompressible Navier-Stokes equation for the velocity and the pressure with a certain additional stress tensor. 
The Ericksen-Leslie model characterizes the macroscopic orientation of nematic liquid crystals using a unit vector, known as the director, and describes the elastic distortion of the nematic phase through the Oseen-Frank elastic energy \cite{36Leslie,4Leslie,2003ad}. As an alternative, the Landau-de Gennes theory \cite{95pg,15cs,12pai,98qian,02wang,17zhao} offers a more general tensorial description that captures both the orientation and degree of order of the liquid crystal phase.

Due to the complexity of the original Ericksen-Leslie model with some reaction-coupling terms, Lin \cite{12Lin}  proposed the following simplified Ericksen-Leslie model:
\begin{numcases}{}
\frac{\partial \dd}{\partial t}+\uu\cdot\nabla \dd -\gamma\Delta\dd-\gamma\lvert\nabla\dd\rvert^2\dd=0,\label{eq:SEL-1}\\
\frac{\partial \uu}{\partial t}+\uu\cdot\nabla \uu-\nu\Delta\uu+\nabla p+\lambda\nabla\cdot\left(\left(\nabla\dd\right)^t\nabla\dd\right)=0,\label{eq:SEL-2}\\
\nabla\cdot\uu=0,\label{eq:SEL-3}\\
|\dd|=1,\label{eq:SEL-4}
\end{numcases}
in $\Omega_T=\Omega\times (0,T]$, where $\Omega\subset R^d \ (d=2, 3)$ is a bounded domain with a Lipschitz continuous boundary $\partial \Omega$. The system is subject to the boundary conditions
\begin{equation}\label{bnddata}
\uu|_{\partial \Omega}=0, \quad \partial_{\nn}\dd |_{\partial\Omega}=0,
\end{equation}
where $\nn$ denotes the unit outward normal vector, and with the initial conditions
\begin{equation}\label{inidata}
\dd\left(\bx,0\right)=\dd^0\left(\bx\right), \quad \uu\left(\bx,0\right)=\uu^0\left(\bx\right), \quad
p\left(\bx,0\right)=p^0\left(\bx\right), \quad
\text{for}\;\bx\in\Omega,
\end{equation}
where $\uu^0:\Omega_{T}\to\mathbb{R}^d$, $\dd^0:\Omega_{T}\to\mathbb{R}^d$, and $p^0:\Omega_{T}\to\mathbb{R}$ are given functions.
The unknowns are the velocity field $\uu$, pressure $p$, and director field $\dd$. 
In the Ericksen-Leslie model \eqref{eq:SEL-1}-\eqref{eq:SEL-4}, the vector $\mathbf{d}$ describes the local average orientation of nematic liquid crystal molecules. 
Since these molecules are rigid and inextensible, their orientation is naturally modeled by a unit vector field. The unit-length constraint $|\mathbf{d}| = 1$ is a fundamental physical assumption in modeling nematic liquid crystals. 
The positive physical constant parameters are the fluid viscosity $\nu$, elasticity constant $\lambda$ and relaxation time constant $\gamma$.

Let $\Vert\cdot\Vert$ denotes the $L^2$ norm of scalars, vectors or tensors, and $(\cdot,\cdot)$ means the $L^2$ inner product. Taking the inner products of \eqref{eq:SEL-1} with $\lambda\left(-\Delta\dd-\lvert\nabla\dd\rvert^2\dd\right)$, \eqref{eq:SEL-2} with $\uu$, and summing up the two equations with some algebraic manipulation, we obtain the following energy dissipation law 
$$\frac{d}{dt}W_s(\uu,\dd)
=-\nu\lVert \nabla \uu\rVert^2-\lambda\gamma\|\Delta\dd+|\nabla\dd|^2\dd\|^2\leq 0.$$
where
$$W_s(\uu,\dd)=\frac{1}{2}\lVert \uu\rVert^2
+\frac{\lambda}{2}\lVert\nabla\dd\rVert^2.$$


Lin and Liu \cite{20Lin} established its well-posedness of the simplified Ericksen-Leslie model \eqref{eq:SEL-1}-\eqref{eq:SEL-4}. Indeed, they prove global existence of weak solutions as well as local existence of strong solutions to system \eqref{eq:SEL-1}-\eqref{eq:SEL-4}.
For numerical approximation we refer to \cite{24Becker, 22Du, 23Lin}. 
Du, Guo and Shen \cite{22Du} proposed a Fourier-spectral method and derived an error estimation, which demonstrated the spectral accuracy of the proposed method. 
In \cite{23Lin}, Lin and Liu presented the $C^0$ finite element method for $2D$ hydrodynamic liquid crystal model. 

The Ginzburg-Landau penalty function \cite{21Liu} is commonly used to relax the constraint $|\dd|=1$, then obtain a modified Ericksen-Leslie model:
\begin{numcases}{}
\frac{\partial \dd}{\partial t}+\uu\cdot\nabla \dd +\gamma\left(-\Delta\dd+f_\varepsilon\left(\dd\right)\right)=0,\;|\dd|\le1,\label{eq:PEL-1}\\
\frac{\partial \uu}{\partial t}+\uu\cdot\nabla \uu-\nu\Delta\uu+\nabla p+\lambda\nabla\cdot\left(\left(\nabla\dd\right)^t\nabla\dd\right)=0,\label{eq:PEL-2}\\
\nabla\cdot\uu=0,\label{eq:PEL-3}
\end{numcases}
where $0<\varepsilon\ll1$ is the penalty parameter.
Here, the Ginzburg-Landau penalty function $f_\varepsilon(\dd)$ is defined as $f_\varepsilon(\dd):=\frac{1}{\varepsilon^2}\left(|\dd|^2-1\right)\dd$, which is the gradient of the Ginzburg-Landau potential function $F_\varepsilon(\dd)=\frac{1}{4\varepsilon^2}\left(|\dd|^2-1\right)^2$, i.e. $f_\varepsilon(\dd)=\nabla_{\dd} F_\varepsilon(\dd)$.
As $\varepsilon \to 0$, the penalized system \eqref{eq:PEL-1}-\eqref{eq:PEL-3} equivalents to the simplified model \eqref{eq:SEL-1}-\eqref{eq:SEL-4} \cite{32Badia}.

Multiplying \eqref{eq:PEL-1} by $\lambda\left(-\Delta\dd+f_\varepsilon(\dd)\right)$ and \eqref{eq:PEL-2} by $\uu$, after the integration in $\Omega$, we obtain the following energy law:
\begin{equation}\label{edl4mEL}
\frac{d}{dt}W(\uu,\dd)
=-\nu\lVert \nabla \uu\rVert^2-\lambda\gamma\|-\Delta\dd+f_\varepsilon(\dd)\|^2,
\end{equation}
where 
$$W=\frac{1}{2}\lVert \uu\rVert^2
+\frac{\lambda}{2}\lVert\nabla\dd\rVert^2+\lambda\int_\Omega F_\varepsilon(\dd)\;d\bx$$
represents the total energy of the system \eqref{eq:PEL-1}-\eqref{eq:PEL-3}, consisting of the kinetic energy $W_{kin}=\frac{1}{2}\lVert u\rVert^2$, the elastic energy $W_{ela}=\frac{\lambda}{2}\lVert\nabla\dd\rVert^2$, and the penalty energy $W_{pen}=\lambda\int_\Omega F_\varepsilon(\dd)\;d\bx$. 

Liu and Walkington \cite{21Liu} developed a finite element scheme for \eqref{eq:PEL-1}-\eqref{eq:PEL-3} and simulated the dynamical behaviors of defects in liquid crystals. 
Becker and Feng et al.\cite{24Becker} proposed two fully discretized schemes: one for the system \eqref{eq:PEL-1}-\eqref{eq:PEL-3}, which is unconditionally stable and satisfies the discrete energy law, and the other for direct discretization of system \eqref{eq:SEL-1}-\eqref{eq:SEL-4}. 
Girault and Guill\'en-Gonz\'alez \cite{14Girault} investigated a fully discrete scheme based on $C^0$ finite elements in space and a semi-implicit Euler scheme in time. They established its unconditional stability and convergence theories.
In \cite{25Guillen}, Guill\'en-Gonz\'alez and Guti\'errez-Santacreu introduced a linear, unconditionally stable, semi-implicit scheme that satisfies the energy law for the modified system \eqref{eq:PEL-1}-\eqref{eq:PEL-3}.
Zheng et al. \cite{23zheng} proposed an IMEX-SAV-DG approach to construct a linear and fully decoupled numerical scheme. They established the unconditional energy stability of the method and provided a rigorous error analysis.
Zou et al. \cite{23zou} combined the extrapolated Crank-Nicolson time-stepping scheme with a convex splitting method to develop a fully discrete virtual element scheme. They also established the stability and convergence of the proposed scheme.
Chen and Yang \cite{23chen} considered this model coupled with the Cahn-Hilliard equation to describe the behavior of liquid crystal phase immersed in free flow.
They introduced two auxiliary variables to construct a fully decoupled numerical scheme.
Additional related contributions can be found in \cite{22chen,26Liu,27Zhao,21wang,23zhang,28Zhao,29Zhao}.

In view of the energy decay property of the modified Ericksen-Leslie model \eqref{eq:PEL-1}-\eqref{eq:PEL-3}, it is desirable to design numerical schemes that preserve the discrete version of energy dissipation law \eqref{edl4mEL}.
However, developing such schemes remains challenging due to the strong nonlinearities and coupling among the velocity, pressure, and director fields.
This paper aims to construct a linear, unconditionally stable and fully decoupled numerical scheme for the modified Ericksen-Leslie model.
The main difficulty in designing such efficient numerical scheme comes from the strongly nonlinear and coupled terms.  
To overcome these challenges, we employ the Lagrange multipliers \cite{25Guillen} to linearize the nonlinear term and the rotational incremental pressure correction (PC) method \cite{11Guermond} to decouple the velocity and pressure.
Furthermore, we implement the scalar auxiliary variable (SAV) \cite{16Jiang} decoupling strategy to handle the convection terms in the system, enhancing the overall efficiency and stability of the numerical scheme.
We then apply the second-order backward differentiation formula (BDF2) to the modified Ericksen-Leslie model \eqref{eq:PEL-1}-\eqref{eq:PEL-3} and establish the unconditional energy stability of the proposed scheme.  
A notable feature of this scheme is that each variable can be updated by solving a single linear system at each time step without any iterative coupling. The SAV strategy is employed to handle the nonlinear convection terms while preserving the modified energy law. Combined with the pressure correction strategy and the Lagrange multiplier method, the scheme achieves second-order accuracy in time, ensures unconditional energy stability, and significantly improves computational efficiency.

This paper is organized as follows. In Section \ref{sec:pcsav}, we develop a semi-discrete scheme based on the pressure correction method and the scalar auxiliary variable approach (PCSAV), prove it unconditional energy stability, and describe an efficient implementation. Section \ref{sec:PCSAV-ECT} introduces a PCSAV method with explicit time-stepping scheme for convection term (PCSAV-ECT). In Section \ref{sec:result}, we present several numerical experiments to verify the accuracy, stability and computational efficiency of two schemes.

\section{PCSAV method for the modified Ericksen-Leslie model}\label{sec:pcsav}
In this section, we propose a linear and decoupled semi-discrete scheme for the modified Ericksen-Leslie model. 
A scalar auxiliary variable is introduced to reformulate the equations \eqref{eq:PEL-1}-\eqref{eq:PEL-3} into an equivalent system, then we apply BDF2 scheme with the pressure correction method to the equivalent system to obtain the semi-discrete scheme.
We prove that the proposed scheme satisfies the energy dissipation law at the discrete level.
In the end, we present a detailed procedure to efficiently implement the scheme. 


Notice that the elastic tensor $\lambda\nabla\cdot\left(\left(\nabla\dd\right)^t\nabla\dd\right)$ of \eqref{eq:PEL-2} can be written \cite{20Lin,32Badia} as
\begin{equation*}
\lambda\nabla\cdot\left(\left(\nabla\dd\right)^t\nabla\dd\right)
=\lambda\nabla\left(\frac{1}{2}|\nabla\dd|^2\right)
-\lambda\left(\nabla\dd\right)^t\left(-\Delta\dd\right).
\end{equation*}
Since $\nabla F_{\varepsilon}=\left(\nabla\dd\right)^t\nabla_{\dd}F_{\varepsilon}=\left(\nabla\dd\right)^tf_\varepsilon(\dd)$, we have

\begin{equation}\label{elastic_tensor}
\lambda\nabla\cdot\left(\left(\nabla\dd\right)^t\nabla\dd\right)
=\lambda\nabla\left(\frac{1}{2}|\nabla\dd|^2+F_{\varepsilon}\right)
-\lambda\left(\nabla\dd\right)^t\left(-\Delta\dd+f_\varepsilon(\dd)\right),
\end{equation}
where $\frac{1}{2}|\nabla\dd|^2+F_{\varepsilon}$ can be incorporated as part of the pressure.
Let $\ww=-\Delta\dd+f_\varepsilon(\dd)$, the modified Ericksen-Leslie model \eqref{eq:PEL-1}-\eqref{eq:PEL-3} can be rewritten as \cite{25Guillen}
\begin{numcases}{}
\frac{\partial \dd}{\partial t}+\uu\cdot\nabla \dd +\gamma\ww=0,\label{eq:EL-1}\\
\ww=-\Delta\dd+f_\varepsilon\left(\dd\right),\label{eq:EL-2}\\
\frac{\partial \uu}{\partial t}+\uu\cdot\nabla \uu-\nu\Delta\uu+\nabla p-\lambda\left(\nabla\dd\right)^t\ww=0,\label{eq:EL-3}\\
\nabla\cdot\uu=0.\label{eq:EL-4}
\end{numcases}
Let $q=\frac{1}{\varepsilon^2}\left(\lvert\dd\rvert^2-1\right)$, then $f_{\varepsilon}\left(\dd\right)=q\dd$ and the total energy is
\begin{equation}\label{energy_q}
W=\frac{1}{2}\lVert \uu\rVert^2
+\frac{\lambda}{2}\lVert\nabla\dd\rVert^2+\lambda\int_\Omega \frac{\varepsilon^2}{4}q^2\;d\bx.
\end{equation}

To handle the nonlinear convection term in the system \eqref{eq:EL-1}-\eqref{eq:EL-4}, we introduce a scalar auxiliary variable 
$s(t)$, defined by
\[s(t)=\exp(-\frac{t}{T}).\]
The system \eqref{eq:EL-1}-\eqref{eq:EL-4} is equivalent to the following system:
\begin{numcases}{}
\frac{\partial \dd}{\partial t}+\frac{s\left(t\right)}{\exp\left(-\frac{t}{T}\right)}\uu\cdot\nabla \dd +\gamma\ww=0,\label{eq:sav-1}\\
\ww=-\Delta\dd+q\dd,\label{eq:sav-2}\\
\frac{\partial q}{\partial t}=\frac{2}{\varepsilon^2}\left(\dd\right)^t\frac{\partial\dd}{\partial t},\label{eq:sav-qt}\\
\frac{\partial \uu}{\partial t}+\left(\uu\cdot\nabla\right)\uu+\frac{1}{2}\left(\nabla\cdot\uu\right)\uu-\nu\Delta\uu+\nabla p=\lambda\frac{s\left(t\right)}{\exp\left(-\frac{t}{T}\right)}\left(\nabla\dd\right)^t \ww,\label{eq:sav-3}\\
\nabla\cdot\uu=0\label{eq:sav-4},\\
\frac{\partial s}{\partial t}=-\frac{s}{T}+\frac{1}{\exp\left(-\frac{t}{T}\right)}\int_\Omega \left(\uu\cdot\nabla\dd\right)\ww-\left(\left(\nabla\dd\right)^t\ww\right)\uu\;d\bx\label{eq:sav-5}.
\end{numcases}
Here, the additional term $\frac{1}{\exp\left(-\frac{t}{T}\right)}\int_\Omega \left(\uu\cdot\nabla\dd\right)\ww-\left(\left(\nabla\dd\right)^t\ww\right)\uu\;d\bx$ in \eqref{eq:sav-5} is identical to zero because $\left(\uu\cdot\nabla\dd,\ww\right)=\left(\left(\nabla\dd\right)^t\ww,\uu\right)$.

Taking the inner products of \eqref{eq:sav-1}-\eqref{eq:sav-3} with $\lambda\ww$, $-\lambda\frac{\partial \dd}{\partial t}$,  $\lambda\frac{\varepsilon^2}{2}q$, $\uu$ respectively, and multiplying 
\eqref{eq:sav-5} with $s$, subject to some algebraic manipulation, we obtain the modified energy law:
$$
\frac{d}{dt}\tilde{W}(\uu,\dd)
=-\nu\lVert \nabla \uu\rVert^2-\lambda\gamma\|\ww\|^2,
$$
where 
$$\tilde{W}(\uu,\dd)=\frac{1}{2}\lVert \uu\rVert^2
+\frac{\lambda}{2}\lVert\nabla\dd\rVert^2+\frac{\lambda\varepsilon^2}{4}\lVert q\rVert^2+\frac{1}{2}s^2$$
represents the corresponding total energy of the system \eqref{eq:sav-1}-\eqref{eq:sav-5}. 

\subsection{PCSAV semi-discrete scheme}
Let $\{t^n| t^n=n\Delta t, n= 0,1,\cdots, N\}$ be a uniform partition in interval $[0,T]$ with the time step $\Delta t=T/N$, where $N$ is a positive integer. The semi-discrete scheme to \eqref{eq:sav-1}-\eqref{eq:sav-5} based on the BDF2 method for temporal discretization is stated as follows.

\begin{sch}\label{2stPCSAV}
Given $\uu^{n-1}$, $\uu^n$, $\dd^{n-1}$, $\dd^{n}$, $p^{n-1}$, $p^n$, $q^{n-1}$, $q^n$, $s^{n-1}$, $s^n$, find $\uu^{n+1}$, $\dd^{n+1}$, $p^{n+1}$, $q^{n+1}$, and $s^{n+1}$ satisfying
\end{sch}
\begin{align}
&\frac{3\dd^{n+1}-4\dd^{n}+\dd^{n-1}}{2\Delta t}+\frac{s^{n+1}}{\exp\left(-\frac{t^{n+1}}{T}\right)} \tilde{\uu}^{n+1}\cdot\nabla \tilde{\dd}^{n+1} +\gamma\ww^{n+1}=0,\label{eq:2stPCSAV2}\\
&\ww^{n+1} = -\Delta\dd^{n+1}+q^{n+1}\tilde{\dd}^{n+1},\label{eq:2weq}\\
&\frac{3q^{n+1}-4q^{n}+q^{n-1}}{2\Delta t}=\frac{2}{\varepsilon^2}\left( \tilde{\dd}^{n+1}\right)^t \frac{3\dd^{n+1}-4\dd^{n}+\dd^{n-1}}{2\Delta t}, \label{eq:2stPCSAV3}\\
&\frac{3\uu_*^{n+1}-4\uu^{n}+\uu^{n-1}}{2\Delta t}+
\left(\tilde{\uu}^{n+1}\cdot\nabla\right)\uu_*^{n+1}+\frac{1}{2}\left(\nabla\cdot\tilde{\uu}^{n+1}\right)\uu_*^{n+1}
-\nu\Delta\uu_*^{n+1}+\nabla p^{n} \nonumber \\
&\qquad\qquad\qquad\qquad\qquad =\lambda\frac{s^{n+1}}{\exp\left(-\frac{t^{n+1}}{T}\right)} \left(\nabla\tilde{\dd}^{n+1}\right)^t \tilde{\ww}^{n+1},\quad\uu_*^{n+1}|_{\partial\Omega}=0,\label{eq:2stPCSAV4}\\
&\frac{3\uu^{n+1}-3\uu_*^{n+1}}{2\Delta t}
+\nabla\left(p^{n+1}-p^n+\nu\nabla\cdot\uu_*^{n+1}\right)=0,\label{eq:2stPCSAV5}\\
&\nabla\cdot\uu^{n+1}=0,\qquad \uu^{n+1}\cdot\nn|_{\partial\Omega}=0,\label{eq:nabu0}\\
&\frac{3s^{n+1}-4s^{n}+s^{n-1}}{2\Delta t}
= -\frac{s^{n+1}}{T} + \frac{1}{\exp\left(-\frac{t^{n+1}}{T}\right)} \int_{\Omega}\left(\tilde{\uu}^{n+1}\cdot \nabla\tilde{\dd}^{n+1}\right)\ww^{n+1}\;d\bx\nonumber\\
&\qquad\qquad\qquad\quad\quad\quad\quad-\frac{1}{\exp\left(-\frac{t^{n+1}}{T}\right)} \int_{\Omega}\left(\uu_*^{n+1}\cdot \nabla\tilde{\dd}^{n+1}\right)\tilde{\ww}^{n+1} \, d\bx,\label{eq:2stPCSAV6}
\end{align}
where $\tilde{\uu}^{n+1}=2\uu^n-\uu^{n-1}$, $\tilde{\dd}^{n+1}=2\dd^n-\dd^{n-1}$, $\tilde{\ww}^{n+1}=2\ww^n-\ww^{n-1}$. 

\begin{remark}
Note that the BDF2 time-stepping method requires initial conditions for the first two time steps. From initial condition \eqref{inidata}, we take 
\begin{equation}\label{inidatadis}
\dd^0=\dd_0, \quad \uu^0=\uu_0,\quad p^0=0,\quad q^0=\frac{1}{\varepsilon^2}\left(\lvert\dd^0\rvert^2-1\right),\quad s^0=\exp\left(-\frac{t^{0}}{T}\right).
\end{equation}
And we obtain $\uu^1, p^1, \dd^1$ by the following scheme:
\begin{align}
&\frac{\dd^{1}-\dd^{0}}{\Delta t}+ \uu^{1}\cdot \nabla \dd^{0}-\gamma\left( \Delta \dd^{1}-f_{\varepsilon}^0(\dd^{1}, \dd^{0})\right)=0,\label{initial01}\\
&\frac{\uu^{1}-\uu^{0}}{\Delta t}+ \uu^{0}\cdot\nabla \uu^{1}-\nu \Delta \uu^{1}+\nabla p^{1} +\frac{\lambda}{\gamma}\left(\nabla\dd^{0}\right)^t\left(\frac{\dd^{1}-\dd^{0}}{\Delta t}+ \uu^{1}\cdot \nabla \dd^{0}\right)=0,\label{initial03}\\
&\nabla \cdot \uu^{1}=0,\label{initial04}
\end{align}
where 
\[f_{\varepsilon}^0(\dd^{1}, \dd^{0})=\frac{1}{\varepsilon^2}|\dd^{1}|^2\dd^{1}-\frac{1}{\varepsilon^2}\dd^0.\]
Furthermore, we take
$$q^1=\frac{1}{\varepsilon^2}\left(\lvert\dd^1\rvert^2-1\right),\quad s^1=\exp\left(-\frac{t^{1}}{T}\right), \quad \ww^1=-\Delta\dd^1+q^1\dd^1.$$

We now prove that the scheme \eqref{initial01}-\eqref{initial04} satisfies the following energy dissipation law
\begin{align}\label{1stenergy}
&W\left(\uu^{1},\dd^{1}\right)-W\left(\uu^{0},\dd^{0}\right)\nonumber
\\
\leq& -\Delta t\nu\lVert\nabla\uu^{1}\rVert^2-\frac{1}{2} \lVert \uu^{1}-\uu^0 \rVert^2
-\frac{\lambda}{2}\lVert \nabla\dd^{1}-\nabla\dd^0 \rVert^2-\frac{\lambda}{2\varepsilon^2}\left(|\dd^1|^2+1\right)\lVert\dd^1-\dd^0\rVert^2.
\end{align}
where 
$$W\left(\uu^{0},\dd^{0}\right)=\frac{\lambda}{2}\lVert\nabla\dd^0\rVert^2+\frac{1}{2}\lVert\uu^0\rVert^2+\lambda\int_\Omega F_\varepsilon(\dd^0)\;d\bx.$$

Let 
\begin{equation}\label{initial02}
\xi=\frac{\dd^{1}-\dd^{0}}{\Delta t}+ \uu^{1}\cdot \nabla \dd^{0},
\end{equation}
taking the inner product of \eqref{initial01} with $\frac{\lambda}{\gamma}(\dd^{1}-\dd^0)$, \eqref{initial02} with $\frac{\lambda}{\gamma}\Delta t\xi$, \eqref{initial03} with $\Delta t\uu^1$, \eqref{initial04} with $-\Delta t p^1$, respectively, and using the identity \begin{equation}\label{a-ba}
(a-b,a)=\frac{1}{2}\left(|a|^2-|b|^2+|a-b|^2\right),
\end{equation} 
we can obtain
\begin{align}\label{1stenergyeq1}
\frac{\lambda}{2}
\left(\lVert\nabla\dd^{1}\rVert^2-\lVert\nabla\dd^{0}\rVert^2+\lVert\nabla\dd^{1}-\nabla\dd^0\rVert^2\right)
+&\frac{\gamma}{\lambda}\Delta t\lVert\xi\rVert^2
+\frac{1}{2}
\left(\lVert\uu^{1}\rVert^2-\lVert\uu^{0}\rVert^2+\lVert\uu^{1}-\uu^0\rVert^2\right)\nonumber\\
&+\Delta t\nu\lVert\nabla\uu^{1}\rVert^2
+\lambda\left(f_\varepsilon^0,\dd^1-\dd^0\right)=0.
\end{align}
Next, we decompose the last term of \eqref{1stenergyeq1} as follows:
\begin{align*}
\lambda\left(f_\varepsilon^0,\dd^1-\dd^0\right)&=\frac{\lambda}{\varepsilon^2}\left(|\dd^{1}|^2\dd^{1}-\dd^0,\dd^1-\dd^0\right)\\
&=\frac{\lambda}{\varepsilon^2}\left(\left(|\dd^{1}|^2-1\right)\dd^{1},\dd^1-\dd^0\right)
+\frac{\lambda}{\varepsilon^2}\left(\dd^1-\dd^0,\dd^1-\dd^0\right)\\
&:=I_1+I_2.
\end{align*}
Rewriting $I_1$ as
\begin{align*}
I_1&=\frac{\lambda}{2\varepsilon^2}\int_{\Omega}\left(|\dd^{1}|^2-1\right)
\left(|\dd^{1}|^2-|\dd^{0}|^2+|\dd^1-\dd^0|^2\right)\;d\bx\\
&=\frac{\lambda}{2\varepsilon^2}\int_{\Omega}\left(|\dd^{1}|^2-1\right)\left(\left(|\dd^{1}|^2-1\right)-\left(|\dd^{0}|^2-1\right)\right)\;d\bx
+\frac{\lambda}{2\varepsilon^2}\int_{\Omega}\left(|\dd^{1}|^2-1\right)\lvert\dd^1-\dd^0\rvert^2\;d\bx\\
&=\frac{\lambda}{4\varepsilon^2}\int_{\Omega}\left(|\dd^{1}|^2-1\right)^2-\left(|\dd^{0}|^2-1\right)^2+\left(|\dd^{1}|^2-|\dd^{0}|^2\right)^2\;d\bx
+\frac{\lambda}{2\varepsilon^2}\int_{\Omega}\left(|\dd^{1}|^2-1\right)\lvert\dd^1-\dd^0\rvert^2\;d\bx
\end{align*}
and 
$I_2=\frac{\lambda}{\varepsilon^2}\lVert\dd^1-\dd^0\rVert^2$, we arrive at the equality
\begin{align}\label{1stenergyeq2}
\lambda\left(f_\varepsilon^0,\dd^1-\dd^0\right)&=\lambda\int_\Omega F_\varepsilon(\dd^1)\;d\bx-
\lambda\int_\Omega F_\varepsilon(\dd^0)\;d\bx\nonumber\\
&+
\frac{\lambda}{2\varepsilon^2}\int_{\Omega}\left(|\dd^{1}|^2+1\right)\lvert\dd^1-\dd^0\rvert^2\;d\bx
+\frac{\lambda}{4\varepsilon^2}\int_\Omega\left(|\dd^{1}|^2-|\dd^{0}|^2\right)^2\;d\bx.
\end{align}
Substituting \eqref{1stenergyeq2} into \eqref{1stenergyeq1}, we obtain \eqref{1stenergy} immediately. 
\end{remark}

\subsection{Discrete energy dissipation property}
The semi-discrete scheme \eqref{eq:2stPCSAV2}-\eqref{eq:2stPCSAV6} satisfies the energy dissipation law at the discrete level. 

\begin{theorem}\label{pcsavthm}
The scheme \eqref{eq:2stPCSAV2}-\eqref{eq:2stPCSAV6} is unconditionally energy stable in the sense that
\begin{equation}\label{2stenergy}
\begin{aligned}
&W^*\left(\uu^{n+1},\dd^{n+1},q^{n+1},p^{n+1},s^{n+1}\right)-W^*\left(\uu^{n},\dd^{n},q^{n},p^{n},s^{n}\right)\\
\leq &-2\Delta t\lambda\gamma\lVert \ww^{n+1}\rVert^2-\nu\Delta t\lVert \nabla\uu_*^{n+1}\rVert^2
-\nu\Delta t \lVert \nabla\times\uu_*^{n+1}\rVert^2
-\frac{2\Delta t}{T}\lambda\lvert s^{n+1}\rvert^2,
\end{aligned}
\end{equation}
where
\begin{align*}
W^*\left(\uu^{n},\dd^{n},q^{n},p^{n},s^{n}\right)=&\frac{1}{2}\lVert\uu^{n}\rVert^2+\frac{1}{2}\lVert2\uu^{n}-\uu^{n-1}\rVert^2+\frac{2}{3}\Delta t^2\lVert\nabla H^{n}\rVert^2
+\nu^{-1}\Delta t\lVert g^{n}\rVert^2\\
&+\frac{\lambda}{2}\lVert\nabla\dd^{n}\rVert^2
+\frac{\lambda}{2}\lVert\nabla\left(2\dd^{n}-\dd^{n-1}\right)\rVert^2\\
&+\frac{\lambda\varepsilon^2}{4}\lVert q^{n}\rVert^2
+\frac{\lambda\varepsilon^2}{4}\lVert 2q^{n}-q^{n-1}\rVert^2\\
&+\frac{\lambda}{2}\lvert s^{n}\rvert^2
+\frac{\lambda}{2}\lvert 2s^{n}-s^{n-1}\rvert^2,
\end{align*}
and $\{g^n,H^n\}$ are defined by
\begin{equation}\label{gh-eq}
g^0=0,\quad g^{n+1}=\nu\nabla\cdot\uu_*^{n+1}+g^n,\quad H^{n+1}=p^{n+1}+g^{n+1},\quad n\ge 0.
\end{equation}
\end{theorem}

\begin{proof}
Taking the inner product of \eqref{eq:2stPCSAV2} with $2\Delta t\lambda\ww^{n+1}$, \eqref{eq:2weq} with $-\lambda\left(3\dd^{n+1}-4\dd^{n}+\dd^{n-1}\right)$, \eqref{eq:2stPCSAV3} with $\Delta t\lambda\varepsilon^2q^{n+1}$, respectively, and using the identity 
\begin{equation}\label{ep1}
(3a-4b+c,a)=\frac{1}{2}\left(|a|^2+|2a-b|^2-|b|^2-|2b-c|^2+|a-2b+c|^2\right),
\end{equation}
we can obtain
{\small \begin{align}\label{ep2}
2\Delta t\lambda\gamma\|\ww^{n+1}\|^2&+2\Delta t\lambda\frac{s^{n+1}}{\exp\left(-\frac{t^{n+1}}{T}\right)}\left(\tilde{\uu}^{n+1}\cdot\nabla \tilde{\dd}^{n+1},\ww^{n+1}\right)\\
&+\frac{\lambda}{2}\left(\lVert\nabla\dd^{n+1}\rVert^2+\lVert\nabla\left(2\dd^{n+1}-\dd^n\right)\rVert^2-\lVert\nabla\dd^{n}\rVert^2-\lVert\nabla\left(2\dd^{n}-\dd^{n-1}\right)\rVert^2
\right)\nonumber\\
&+\frac{\lambda\varepsilon^2}{4}\left(\|q^{n+1}\|^2+\|2q^{n+1}-q^n\|^2-\|q^{n}\|^2-\|2q^n-q^{n-1}\|^2\right)\nonumber\\
&+\frac{\lambda}{2}\lVert\nabla\left(\dd^{n+1}-2\dd^n+\dd^{n-1}\right)\rVert^2+\frac{\lambda\varepsilon^2}{4}\|q^{n+1}-2q^n+q^{n-1}\|^2=0.\nonumber
\end{align}}
Taking the inner product of \eqref{eq:2stPCSAV4} with $2\Delta t\uu_*^{n+1}$ leads to
\begin{align}\label{ep3}
\left(3\uu_*^{n+1}-4\uu^{n}+\uu^{n-1},\uu_*^{n+1}\right)&+2\Delta t\nu\Vert \nabla\uu_*^{n+1}\Vert^2+2\Delta t\left(\nabla p^n,\uu_*^{n+1}\right)\\
&= 2\Delta t\lambda\frac{s^{n+1}}{2\exp\left(-\frac{t^{n+1}}{T}\right)}\left(\left(\nabla\tilde{\dd}^{n+1}\right)^t\tilde{\ww}^{n+1},\uu_*^{n+1}\right).\nonumber
\end{align}
Recalling \eqref{ep1} and \eqref{eq:2stPCSAV5}, applying the identity \eqref{a-ba} to the first term on the left-hand side of \eqref{ep3}, we have
{\begin{align}\label{ep32}
\left(3\uu_*^{n+1}-4\uu^{n}+\uu^{n-1},\uu_*^{n+1}\right)&=\left(3\left(\uu_*^{n+1}-\uu^{n+1}\right)+3\uu^{n+1}-4\uu^{n}+\uu^{n-1},\uu_*^{n+1}\right)\nonumber\\
&=3\left(\uu_*^{n+1}-\uu^{n+1},\uu_*^{n+1}\right)+
\left(3\uu^{n+1}-4\uu^{n}+\uu^{n-1},\uu^{n+1}\right)\nonumber\\
&\quad+\left(3\uu^{n+1}-4\uu^{n}+\uu^{n-1},\uu_*^{n+1}-\uu^{n+1}\right)\nonumber\\
&=3\left(\uu_*^{n+1}-\uu^{n+1},\uu_*^{n+1}\right)+
\left(3\uu^{n+1}-4\uu^{n}+\uu^{n-1},\uu^{n+1}\right)\nonumber\\
&\quad+\left(3\uu^{n+1}-4\uu^{n}+\uu^{n-1},\frac{2\Delta t}{3}\nabla\left(p^{n+1}-p^n+\nu\nabla\cdot\uu_*^{n+1}\right)\right)\nonumber\\
&=\frac{3}{2}\left(\|\uu_*^{n+1}\|^2-\|\uu^{n+1}\|^2+\|\uu_*^{n+1}-\uu^{n+1}\|^2\right)+\frac{1}{2}\|\uu^{n+1}\|^2\nonumber\\
&\quad+\frac{1}{2}\|2\uu^{n+1}-\uu^n\|^2
-\frac{1}{2}\|\uu^n\|^2-\frac{1}{2}\|2\uu^{n}-\uu^{n-1}\|^2\nonumber\\
&\quad+\frac{1}{2}\|\uu^{n+1}-2\uu^n+\uu^{n-1}\|^2,
\end{align}}
where $\left(\nabla p^{n+1},\uu^{n+1}\right)=-\left(p^{n+1},\nabla\cdot\uu^{n+1}\right)=0$.

Thanks to \eqref{gh-eq}, we can recast \eqref{eq:2stPCSAV5} as
\begin{align}\label{new2stPCSAV5}
\sqrt{3}\uu^{n+1}+\frac{2}{\sqrt{3}}\Delta t\nabla H^{n+1}=\sqrt{3}\uu_*^{n+1}+\frac{2}{\sqrt{3}}\Delta t\nabla H^{n}.
\end{align}
Taking the inner product of both sides of \eqref{new2stPCSAV5} with itself, we get
\begin{align}\label{ep33}
3\|\uu^{n+1}\|^2&+\frac{4}{3}\Delta t^2\|\nabla H^{n+1}\|^2+4\Delta t\left(\uu^{n+1},\nabla H^{n+1}\right)\\
&=3\|\uu_*^{n+1}\|^2+\frac{4}{3}\Delta t^2\|\nabla H^{n}\|^2+4\Delta t\left(\uu_*^{n+1},\nabla p^{n}\right)+4\Delta t\left(\uu_*^{n+1},\nabla g^{n}\right)\nonumber.
\end{align}
In view of the definition of $g^{n+1}$ in \eqref{gh-eq} and the equation \eqref{a-ba}, we have 
\begin{align}
4\Delta t\left(\uu_*^{n+1},\nabla g^{n}\right)&=-4\Delta t\nu^{-1}\left(g^{n+1}-g^n,g^n\right)\nonumber\\
&=2\Delta t\nu^{-1}\|g^n\|^2-2\Delta t\nu^{-1}\|g^{n+1}\|^2+2\Delta t\nu\|\nabla\cdot\uu_*^{n+1}\|^2.\nonumber
\end{align}
Note that 
$$\|\nabla\times\vv\|^2+\|\nabla\cdot\vv\|^2=\|\nabla\vv\|^2,\quad\forall\vv\in \mathbf{H}_0^1(\Omega),$$
we get
\begin{align}\label{ep34}
4\Delta t\left(\uu_*^{n+1},\nabla g^{n}\right)
=&2\Delta t\nu^{-1}\|g^n\|^2-2\Delta t\nu^{-1}\|g^{n+1}\|^2\\
&+2\Delta t\nu\|\nabla\uu_*^{n+1}\|^2-2\Delta t\nu\|\nabla\times\uu_*^{n+1}\|^2.\nonumber
\end{align}
Substitute \eqref{ep34} into \eqref{ep33}, we obtain 
\begin{align}\label{ep35}
\frac{3}{2}\|\uu^{n+1}\|^2+\frac{2}{3}\Delta t^2\|\nabla H^{n+1}\|^2
=&\frac{3}{2}\|\uu_*^{n+1}\|^2+\frac{2}{3}\Delta t^2\|\nabla H^{n}\|^2
+2\Delta t\left(\uu_*^{n+1},\nabla p^{n}\right)\nonumber\\
&+\Delta t\nu^{-1}\|g^n\|^2-\Delta t\nu^{-1}\|g^{n+1}\|^2\\
&+\Delta t\nu\|\nabla\uu_*^{n+1}\|^2-\Delta t\nu\|\nabla\times\uu_*^{n+1}\|^2\nonumber.
\end{align}
Combining \eqref{ep3} with \eqref{ep32} and \eqref{ep35}, we have
\begin{align}\label{ep5}
\frac{1}{2}\|\uu^{n+1}\|^2&+\frac{1}{2}\|2\uu^{n+1}-\uu^n\|^2+\frac{2}{3}\Delta t^2\|\nabla H^{n+1}\|^2+\Delta t\nu^{-1}\|g^{n+1}\|^2\nonumber\\
&+\frac{3}{2}\|\uu_*^{n+1}-\uu^{n+1}\|^2+\Delta t\nu\Vert \uu_*^{n+1}\Vert^2+\Delta t\nu\|\nabla\times\uu_*^{n+1}\|^2\nonumber\\
&+\frac{1}{2}\|\uu^{n+1}-2\uu^n+\uu^{n-1}\|^2\nonumber\\
=&\frac{1}{2}\|\uu^{n}\|^2+\frac{1}{2}\|2\uu^{n}-\uu^{n-1}\|^2+\frac{2}{3}\Delta t^2\|\nabla H^{n}\|^2+\Delta t\nu^{-1}\|g^{n}\|^2\nonumber\\
&\quad+2\Delta t\lambda\frac{s^{n+1}}{2\exp\left(-\frac{t^{n+1}}{T}\right)}\left(\uu_*^{n+1}\cdot\nabla\tilde{\dd}^{n+1},\tilde{\ww}^{n+1}\right).
\end{align}
Multiplying \eqref{eq:2stPCSAV6} with $2\Delta t\lambda s^{n+1}$ and using \eqref{ep1}, we get 
\begin{align}\label{ep6}
\frac{\lambda}{2}|s^{n+1}|^2&+\frac{\lambda}{2}|2s^{n+1}-s^{n}|^2-\frac{\lambda}{2}|s^{n}|^2-\frac{\lambda}{2}|2s^{n}-s^{n-1}|^2+\frac{\lambda}{2}|s^{n+1}-2s^n+s^{n-1}|^2\nonumber\\
&=-\frac{2\lambda\Delta t}{T}|s^{n+1}|^2+
2\Delta t\lambda\frac{s^{n+1}}{2\exp\left(-\frac{t^{n+1}}{T}\right)} \left(\tilde{\uu}^{n+1}\cdot\nabla\tilde{\dd}^{n+1},\ww^{n+1}\right)\nonumber\\
&\quad-2\Delta t\lambda\frac{s^{n+1}}{2\exp\left(-\frac{t^{n+1}}{T}\right)} \left(\uu_*^{n+1}\cdot\nabla\tilde{\dd}^{n+1},\tilde{\ww}^{n+1}\right).
\end{align}
Adding the equations \eqref{ep2}, \eqref{ep5} and \eqref{ep6} together after some algebraic manipulation, we obtain \eqref{2stenergy} immediately.
This proof is completed.
\end{proof}

\subsection{Implementation of the PCSAV scheme} \label{sec:PCSAV.1}
Scheme \ref{2stPCSAV} is linear but $\uu^{n+1}$ and $\dd^{n+1}$ are coupled with $s^{n+1}$. An effective splitting technique \cite{56Shen} can be applied to decoupled the system \eqref{eq:2stPCSAV2}-\eqref{eq:2stPCSAV6} that allows to solve each variable separately. 

Let
\begin{equation}\label{Kn+1}
K^{n+1}=\frac{s^{n+1}}{\exp\left(-\frac{t^{n+1}}{T}\right)},
\end{equation}
and set
\begin{equation} \label{hatbreve} 
\left\{  
\begin{aligned} 
&\dd^{n+1}=\hat{\dd}^{n+1}+K^{n+1}\breve{\dd}^{n+1},\\&\uu_*^{n+1}=\hat{\uu}_*^{n+1}+K^{n+1}\breve{\uu}_*^{n+1},\\ &\uu^{n+1}=\hat{\uu}^{n+1}+K^{n+1}\breve{\uu}^{n+1},\\
&p^{n+1}=\hat{p}^{n+1}+K^{n+1}\breve{p}^{n+1}.
\end{aligned}  
\right.  
\end{equation} 
From \eqref{eq:2stPCSAV3}, we get
\begin{equation}\label{q2stPCSAV}
q^{n+1}=\frac{4}{3}q^n-\frac{1}{3}q^{n-1}+\frac{2}{3\varepsilon^2} \left(\tilde{\dd}^{n+1}\right)^t \left(3\dd^{n+1}-4\dd^{n}+\dd^{n-1}\right), 
\end{equation}
and \eqref{eq:2weq} becomes
\begin{align}\label{BDFneww}
\ww^{n+1} = &-\Delta\dd^{n+1}+\left(\frac{4}{3}q^n-\frac{1}{3}q^{n-1}\right)\tilde{\dd}^{n+1}\nonumber\\
&+\frac{2}{3\varepsilon^2} \left(\tilde{\dd}^{n+1}\right)^t \left(3\dd^{n+1}-4\dd^{n}+\dd^{n-1}\right)\tilde{\dd}^{n+1}.
\end{align}
Substituting it into \eqref{eq:2stPCSAV2}, we have
\begin{align}
\frac{3\dd^{n+1}-4\dd^{n}+\dd^{n-1}}{2\Delta t}&+\frac{s^{n+1}}{\exp\left(-\frac{t^{n+1}}{T}\right)} \tilde{\uu}^{n+1}\cdot\nabla \tilde{\dd}^{n+1}
-\gamma\Delta\dd^{n+1}+\gamma\left(\frac{4}{3}q^n-\frac{1}{3}q^{n-1}\right)\tilde{\dd}^{n+1}\nonumber\\
&+\frac{2\gamma}{3\varepsilon^2} \left(\tilde{\dd}^{n+1}\right)^t \left(3\dd^{n+1}-4\dd^{n}+\dd^{n-1}\right)\tilde{\dd}^{n+1}
 =0.  \label{eq:2stPCSAV2new} 
\end{align}
Plugging \eqref{hatbreve} in \eqref{eq:2stPCSAV2new}, \eqref{eq:2stPCSAV4}-\eqref{eq:2stPCSAV5}, and collecting terms without $K^{n+1}$, we can obtain $(\hat{\dd}^{n+1}, \breve{\dd}^{n+1})$,  $(\hat{\uu}_*^{n+1}, \breve{\uu}_*^{n+1})$,  $(\hat{\uu}^{n+1}, \breve{\uu}^{n+1})$,  $(\hat{p}^{n+1}, \breve{p}^{n+1})$ and $(\hat{\ww}^{n+1},\breve{\ww}^{n+1})$ separately and individually.

\textbf{Step 1:} Find $(\hat{\dd}^{n+1},\breve{\dd}^{n+1})$ such that
\begin{equation}\label{PCSAV Subproblem 1}
\begin{split}
\frac{3\hat{\dd}^{n+1}-4\dd^n+\dd^{n-1}}{2\Delta t}&-\gamma\left( \Delta \hat{\dd}^{n+1}\right)
+\frac{\gamma}{3}\left(4q^n-q^{n-1}\right)\tilde{\dd}^{n+1}\\
&+\frac{2\gamma}{3\varepsilon^2}\left(\tilde{\dd}^{n+1}\right)^t\left(3\hat{\dd}^{n+1}-4\dd^n+\dd^{n-1}\right)\tilde{\dd}^{n+1}=0,
\end{split}
\end{equation}
\begin{equation}\label{PCSAV Subproblem 2}
\frac{3\breve{\dd}^{n+1}}{2\Delta t}+ \tilde{\uu}^{n+1}\cdot\nabla \tilde{\dd}^{n+1}-\gamma\left( \Delta \breve{\dd}^{n+1}\right)+
\frac{2\gamma}{\varepsilon^2}\left(\left(\tilde{\dd}^{n+1}\right)^t\breve{\dd}^{n+1}\right)\tilde{\dd}^{n+1}=0.
\end{equation}

\textbf{Step 2:} Find $(\hat{\uu}_*^{n+1},\breve{\uu}_*^{n+1})$ such that
\begin{equation}\label{PCSAV Subproblem 3}
\begin{split}
\frac{3\hat{\uu}_*^{n+1}-4\uu^n+\uu^{n-1}}{2\Delta t} + \left(\tilde{\uu}^{n+1}\cdot\nabla\right)\hat{\uu}_*^{n+1}&+ \frac{1}{2}\left(\nabla\cdot \tilde{\uu}^{n+1}\right)\hat{\uu}_*^{n+1}-\nu\Delta\hat{\uu}_*^{n+1}\\
&+\nabla p^n = 0,\quad\hat{\uu}_*^{n+1}|_{\partial\Omega}=0,
\end{split}
\end{equation}
\begin{equation}\label{PCSAV Subproblem 4}
\begin{split}
\frac{3\breve{\uu}_*^{n+1}}{2\Delta t} + \left(\tilde{\uu}^{n+1}\cdot\nabla\right)\breve{\uu}_*^{n+1} &+ \frac{1}{2}\left(\nabla\cdot \tilde{\uu}^{n+1}\right)\breve{\uu}_*^{n+1}
\\
&-\nu\Delta\breve{\uu}_*^{n+1}= \lambda\left(\nabla\tilde{\dd}^{n+1}\right)^t\tilde{\ww}^{n+1},\quad\breve{\uu}_*^{n+1}|_{\partial\Omega}=0.
\end{split}
\end{equation}

\textbf{Step 3:} Find $(\hat{\uu}^{n+1},\breve{\uu}^{n+1})$ and $(\hat{p}^{n+1},\breve{p}^{n+1})$ such that
\begin{equation}\label{PCSAV Subproblem 5}
\left\{\begin{aligned}
&3\hat{\uu}^{n+1}-3\hat{\uu}_*^{n+1}+
2\Delta t\nabla\left(\hat{p}^{n+1}-p^n+\nu\nabla\cdot\hat{\uu}_*^{n+1}\right)
= 0,\\
&\nabla\cdot\hat{\uu}^{n+1}=0,\quad \hat{\uu}^{n+1}\cdot\nn|_{\partial\Omega}=0.
\end{aligned}\right.
\end{equation}
\begin{equation}\label{PCSAV Subproblem 6}
\left\{\begin{aligned}
&3\breve{\uu}^{n+1}-3\breve{\uu}_*^{n+1}+
2\Delta t\nabla\left(\breve{p}^{n+1}+\nu\nabla\cdot\breve{\uu}_*^{n+1}\right)
= 0,\\
&\nabla\cdot\breve{\uu}^{n+1}=0,\quad \breve{\uu}^{n+1}\cdot\nn|_{\partial\Omega}=0.
\end{aligned}\right.
\end{equation}
By taking the divergence operator on each of the above system \eqref{PCSAV Subproblem 5} and \eqref{PCSAV Subproblem 6}, we obtain:
\begin{align}
&\Delta\left(\hat{p}^{n+1}-p^n\right)=\frac{3}{2\Delta t}\nabla\cdot\hat{\uu}_*^{n+1}-\nabla\cdot\nabla\left(\nu\nabla\cdot\hat{\uu}_*^{n+1}\right),\label{BDFphat}\\
&\Delta\breve{p}^{n+1}=\frac{3}{2\Delta t}\nabla\cdot\breve{\uu}_*^{n+1}
-\nabla\cdot\nabla\left(\nu\nabla\cdot\breve{\uu}_*^{n+1}\right).\label{BDFpbreve}
\end{align}
We find that $\hat{p}^{n+1}$ and $\breve{p}^{n+1}$ can be determined by solving a Poisson equation with homogeneous boundary conditions, and then $\hat{\uu}^{n+1}$ and  $\breve{\uu}^{n+1}$ can be updated by
\begin{align}
&\hat{\uu}^{n+1}=\hat{\uu}_*^{n+1}-\frac{2\Delta t}{3}\nabla\left(\hat{p}^{n+1}-p^n\right)
-\frac{2\Delta t}{3}\nabla\left(\nu\nabla\cdot\hat{\uu}_*^{n+1}\right),\label{BDFuhat}\\
&\breve{\uu}^{n+1}=\breve{\uu}_*^{n+1}-\frac{2\Delta t}{3}\nabla\breve{p}^{n+1}
-\frac{2\Delta t}{3}\nabla\left(\nu\nabla\cdot\breve{\uu}_*^{n+1}\right).\label{BDFubreve}
\end{align}

With the decomposition of $\dd^{n+1}$ and \eqref{eq:2stPCSAV2}, $\ww^{n+1}$ can be divided into
\begin{equation}\label{eqw}
\ww^{n+1} = \hat{\ww}^{n+1} + K^{n+1}\breve{\ww}^{n+1},
\end{equation}

\textbf{Step 4:} Find $(\hat{\ww}^{n+1},\breve{\ww}^{n+1})$ such that
\begin{align}
& \hat{\ww}^{n+1} = -\frac{1}{\gamma}\left(\frac{3\hat{\dd}^{n+1}-4\dd^n+\dd^{n-1}}{2\Delta t}\right),\label{2stwhat}\\
& \breve{\ww}^{n+1} = -\frac{1}{\gamma}\left(\frac{3\breve{\dd}^{n+1}}{2\Delta t}+\tilde{\uu}^{n+1}\cdot\nabla\tilde{\dd}^{n+1}\right). \label{2stwbre}
\end{align}

As $\hat{\dd}^{n+1}$, $\breve{\dd}^{n+1}$,  $\hat{\uu}_*^{n+1}$, $\breve{\uu}_*^{n+1}$,  $\hat{\uu}^{n+1}$, $\breve{\uu}^{n+1}$,  $\hat{p}^{n+1}$ and $\breve{p}^{n+1}$ are known, we now derive the expression for $K^{n+1}$. From \eqref{Kn+1}, we have
\begin{align}
s^{n+1}=\exp\left(-\frac{t^{n+1}}{T}\right)K^{n+1}.\label{2stseq}
\end{align}

\textbf{Step 5:} 
Plugging \eqref{2stseq} into \eqref{eq:2stPCSAV6}, and replacing $\ww^{n+1} = \hat{\ww}^{n+1} + K^{n+1}\breve{\ww}^{n+1}$ and $\uu_*^{n+1} = \hat{\uu}_*^{n+1}+K^{n+1}\breve{\uu}_*^{n+1}$, then $K^{n+1}$ can be explicitly determined by the following equation:
\begin{align}\label{2stK^n+1}
A^{n+1}K^{n+1} = B^{n+1} ,
\end{align}
where
\begin{align}\label{An+1}
A^{n+1} = 
\left(\frac{3}{2\Delta t}+\frac{1}{T}\right)\exp\left(-\frac{2t^{n+1}}{T}\right)
&-\left( \tilde{\uu}^{n+1} \cdot \nabla \tilde{\dd}^{n+1}, \breve{\ww}^{n+1}\right)\nonumber\\
&+\left(\breve{\uu}_*^{n+1}\cdot \nabla \tilde{\dd}^{n+1},\tilde{\ww}^{n+1}\right),
\end{align}
and
\begin{align}\label{Bn+1}
B^{n+1} =  \left(\frac{2s^{n}}{\Delta t}-\frac{s^{n-1}}{2\Delta t}\right) \exp\left(-\frac{t^{n+1}}{T}\right)
&+ \left( \tilde{\uu}^{n+1}\cdot \nabla \tilde{\dd}^{n+1},\hat{\ww}^{n+1}\right)\nonumber\\
&- \left(\hat{\uu}_*^{n+1} \cdot \nabla \tilde{\dd}^{n+1},\tilde{\ww}^{n+1}\right) .
\end{align}

In the following, we prove that the linear equation \eqref{2stK^n+1} is uniquely solvable.
Taking the inner product of \eqref{2stwbre} with $\breve{\ww}^{n+1}$, we get
\[
-\left( \tilde{\uu}^{n+1} \cdot \nabla \tilde{\dd}^{n+1} ,\breve{\ww}^{n+1}\right)   
=\gamma\|\breve{\ww}^{n+1}\|^2 + \frac{3}{2\Delta t}(\breve{\dd}^{n+1}, \breve{\ww}^{n+1}).
\]
Substituting \eqref{2stwbre} into \eqref{PCSAV Subproblem 2}, and taking the inner product with $\breve{\dd}^{n+1}$, yields
\[
(\breve{\dd}^{n+1},\breve{\ww}^{n+1}) = \|\nabla\breve{\dd}^{n+1}\|^2 + \frac{2}{\varepsilon^2}\|\breve{\dd}^{n+1}\cdot\tilde{\dd}^{n+1}\|^2.
\]
Combining the above two equations, we obtain
\[
-\left( \tilde{\uu}^{n+1} \cdot \nabla \tilde{\dd}^{n+1} ,\breve{\ww}^{n+1}\right) 
=\gamma\|\breve{\ww}^{n+1}\|^2 + \frac{3}{2\Delta t}\|\nabla\breve{\dd}^{n+1}\|^2 + \frac{3}{\Delta t\varepsilon^2}\|\breve{\dd}^{n+1}\cdot\tilde{\dd}^{n+1}\|^2 \geq 0.
\]
Taking the inner product of \eqref{PCSAV Subproblem 4} with $\frac{\breve{\uu}_*^{n+1}}{\lambda}$, we have
\[
\left(\breve{\uu}_*^{n+1} \cdot \nabla \tilde{\dd}^{n+1} ,\tilde{\ww}^{n+1}\right) = \frac{3}{2\lambda\Delta t}\|\breve{\uu}_*^{n+1}\|^2 + \frac{\nu}{\lambda}\|\nabla\breve{\uu}_*^{n+1}\|^2 \geq 0.
\]
The non-negativity of the above two terms in the right hand side of \eqref{An+1} shows $A^{n+1} > 0$. Hence, $A^{n+1}K^{n+1} + B^{n+1} = 0$ contains a unique solution $K^{n+1}$.	


\section{PCSAV method with explicit scheme for convection term}\label{sec:PCSAV-ECT}
The nonlinear convection term $\uu\cdot\nabla\uu$ in \eqref{eq:EL-3} is discretized using a semi-implicit scheme, resulting in time-dependent coefficient matrices for linear equations of the fully discrete system. 
To make the numerical scheme more efficient, we can treat the convection term explicitly, which leads to a constant coefficient matrix for solving \eqref{eq:EL-3}.

We first reformulate the system \eqref{eq:sav-1}-\eqref{eq:sav-5} to a equivalent system with the scalar auxiliary variable $s(t)$.  
Notice that 
\[
\int_\Omega\uu\cdot\nabla\uu\cdot\uu\;d\bx=\frac{1}{2}\int_\Omega\uu\cdot\nabla|\uu|^2\;d\bx=0,\]
system \eqref{eq:sav-1}-\eqref{eq:sav-5} is equivalent to 
\begin{numcases}{}
\text{Equations}\;\eqref{eq:sav-1}-\eqref{eq:sav-qt} \text{, Equation}\;\eqref{eq:sav-4}, \nonumber\\
\frac{\partial \uu}{\partial t}+\frac{s\left(t\right)}{\exp\left(-\frac{t}{T}\right)}\left(\uu\cdot\nabla\right)\uu-\nu\Delta\uu+\nabla p=\lambda\frac{s\left(t\right)}{\exp\left(-\frac{t}{T}\right)}\left(\nabla\dd\right)^t \ww,\label{eq:ECTPCSAV-NS1}\\
\frac{\partial s}{\partial t}=-\frac{s}{T}+\frac{1}{\exp\left(-\frac{t}{T}\right)}\int_\Omega \left(\uu\cdot\nabla\dd\right)\ww-\left(\left(\nabla\dd\right)^t\ww\right)\uu\;d\bx\label{eq:ECTPCSAV-drdt}\nonumber\\
\qquad\quad\qquad+\frac{1}{\lambda}\frac{1}{\exp\left(-\frac{t}{T}\right)}\int_\Omega\uu\cdot\nabla\uu\cdot\uu\;d\bx.
\end{numcases}

The PCSAV method with an explicit time-stepping scheme for the convection term (PCSAV-ECT) can be stated as follows.
\begin{sch}\label{2stPCSAV-ect}
Given $\uu^{n-1}$, $\uu^n$, $\dd^{n-1}$, $\dd^n$, $p^{n-1}$, $p^n$, $q^{n-1}$, $q^n$, $s^{n-1}$, $s^n$, find $\uu^{n+1}$, $\dd^{n+1}$, $p^{n+1}$, $q^{n+1}$, and $s^{n+1}$ satisfying equation \eqref{eq:2stPCSAV2}-\eqref{eq:2stPCSAV3}, equation \eqref{eq:2stPCSAV5}, equation \eqref{eq:nabu0}, and
\end{sch}
\begin{align}
\frac{3\uu_*^{n+1}-4\uu^{n}+\uu^{n-1}}{2\Delta t}&+\frac{s^{n+1}}{\exp\left(-\frac{t^{n+1}}{T}\right)}\left(\tilde{\uu}^{n+1}\cdot\nabla\right)\tilde{\uu}^{n+1}
-\nu\Delta\uu_*^{n+1}
+\nabla p^{n} \nonumber\\
=& \lambda\frac{s^{n+1}}{\exp\left(-\frac{t^{n+1}}{T}\right)} \left(\nabla\tilde{\dd}^{n+1}\right)^t \tilde{\ww}^{n+1},\quad\uu_*^{n+1}|_{\partial\Omega}=0;\label{eq:2stPCSAV4-ect}\\
\frac{3s^{n+1}-4s^{n}+s^{n-1}}{2\Delta t}
=&-\frac{s^{n+1}}{T} + \frac{1}{\exp\left(-\frac{t^{n+1}}{T}\right)} \int_{\Omega}\left(\tilde{\uu}^{n+1}\cdot \nabla\tilde{\dd}^{n+1}\right)\ww^{n+1}\;d\bx\nonumber\\
&-\frac{1}{\exp\left(-\frac{t^{n+1}}{T}\right)} \int_{\Omega}\left(\uu_*^{n+1}\cdot \nabla\tilde{\dd}^{n+1}\right)\tilde{\ww}^{n+1} \, d\bx\nonumber\\
&+\frac{1}{\lambda}\frac{1}{\exp\left(-\frac{t^{n+1}}{T}\right)}
\int_{\Omega}\left(\tilde{\uu}^{n+1}\cdot \nabla\tilde{\uu}^{n+1}\right)\uu_*^{n+1} \, d\bx,\label{eq:2stPCSAV6-ect}
\end{align}
where $\tilde{\uu}^{n+1}=2\uu^n-\uu^{n-1}$, $\tilde{\dd}^{n+1}=2\dd^n-\dd^{n-1}$, $\tilde{\ww}^{n+1}=2\ww^n-\ww^{n-1}$.

Following the proof procedure of the stability of Scheme \ref{2stPCSAV} in Theorem \ref{pcsavthm}, we can derive the following theorem.

\begin{theorem}\label{pcsav-ectthm} The Scheme \ref{2stPCSAV-ect} is unconditionally energy stable in the sense that
\begin{equation}\label{2stenergy-ect}
\begin{aligned}
&E^*\left(\uu^{n+1},\dd^{n+1},q^{n+1},p^{n+1},s^{n+1}\right)+2\Delta t\lambda M\lVert \ww^{n+1}\rVert^2+\nu\Delta t\lVert \nabla\uu_*^{n+1}\rVert^2
\\
&\qquad\qquad\qquad+\nu\Delta t \lVert \nabla\times\uu_*^{n+1}\rVert^2
+\lambda\frac{2\Delta t}{T}\lvert s^{n+1}\rvert^2
\leq E^*\left(\uu^{n},\dd^{n},q^{n},p^{n},s^{n}\right),
\end{aligned}
\end{equation}
where
\begin{align*}
E^*\left(\uu^{n},\dd^{n},q^{n},p^{n},s^{n}\right)=&\frac{1}{2}\lVert\uu^{n}\rVert^2+\frac{1}{2}\lVert2\uu^{n}-\uu^{n-1}\rVert^2+\frac{2}{3}\Delta t^2\lVert\nabla H^{n}\rVert^2
+\nu^{-1}\Delta t\lVert g^{n}\rVert^2\\
&+\frac{\lambda}{2}\lVert\nabla\dd^{n}\rVert^2
+\frac{\lambda}{2}\lVert\nabla\left(2\dd^{n}-\dd^{n-1}\right)\rVert^2\\
&+\frac{\lambda\varepsilon^2}{4}\lVert q^{n}\rVert^2
+\frac{\lambda\varepsilon^2}{4}\lVert 2q^{n}-q^{n-1}\rVert^2\\
&+\frac{\lambda}{2}\lvert s^{n}\rvert^2
+\frac{\lambda}{2}\lvert 2s^{n}-s^{n-1}\rvert^2.
\end{align*}
\end{theorem}

Similar to the implementation of PCSAV algorithm, we apply the same techniques to Scheme \ref{2stPCSAV-ect} to obtain a fully decoupled system.

Plugging \eqref{hatbreve} in \eqref{eq:2stPCSAV2}, \eqref{eq:2stPCSAV4-ect} and \eqref{eq:2stPCSAV5}, and collecting terms without $K^{n+1}$, we can obtain $\hat{\dd}^{n+1}$, $\breve{\dd}^{n+1}$,  $\hat{\uu}_*^{n+1}$, $\breve{\uu}_*^{n+1}$,  $\hat{\uu}^{n+1}$, $\breve{\uu}^{n+1}$,  $\hat{p}^{n+1}$ and $\breve{p}^{n+1}$ as follows:

\textbf{Step 1:} Solve \eqref{PCSAV Subproblem 1} and \eqref{PCSAV Subproblem 2} to obtain $\hat{\dd}^{n+1}$, $\breve{\dd}^{n+1}$.

\textbf{Step 2:} Find $(\hat{\uu}_*^{n+1},\breve{\uu}_*^{n+1})$ such that
\begin{equation}\label{PCSAV-ECT Subproblems 3}
\begin{split}
\frac{3\hat{\uu}_*^{n+1}-4\uu^n+\uu^{n}}{2\Delta t} -\nu\Delta\hat{\uu}_*^{n+1}+\nabla p^n = 0,\;\hat{\uu}_*^{n+1}|_{\partial\Omega}=0,
\end{split}
\end{equation}
\begin{equation}\label{PCSAV-ECT Subproblems 4}
\begin{split}
\frac{3\breve{\uu}_*^{n+1}}{2\Delta t} + \left(\tilde{\uu}^{n+1}\cdot\nabla\right)\tilde{\uu}^{n+1}-\nu\Delta\breve{\uu}_*^{n+1}= \lambda\left(\nabla\tilde{\dd}^{n+1}\right)^t\tilde{\ww}^{n+1},\;\breve{\uu}_*^{n+1}|_{\partial\Omega}=0.
\end{split}
\end{equation}

\textbf{Step 3:} Solve \eqref{BDFphat} and \eqref{BDFpbreve} to get
$\hat{p}^{n+1}$ and $\breve{p}^{n+1}$. Solve \eqref{BDFuhat} and \eqref{BDFubreve} to obtain
$\hat{\uu}^{n+1}$ and $\breve{\uu}^{n+1}$.

\textbf{Step 4:} Solve 
\eqref{2stwhat} and \eqref{2stwbre} to get
$\hat{\ww}^{n+1}$ and $\breve{\ww}^{n+1}$.

\textbf{Step 5:} Taking \eqref{2stseq} into \eqref{eq:2stPCSAV6-ect}, and replacing $\ww^{n+1}$, $\uu_*^{n+1}$ with \eqref{hatbreve}, 
one can get $K^{n+1}$ by the following equation:
\begin{align}\label{2stK^n+1-ect}
A^{n+1}K^{n+1} = B^{n+1} ,
\end{align}
where
\begin{align}\label{2stAn+1-ect}
A^{n+1} =&  
\left(\frac{3}{2\Delta t}+\frac{1}{T}\right)\exp\left(-\frac{2t^{n+1}}{T}\right) 
-\left( \tilde{\uu}^{n+1} \cdot \nabla \tilde{\dd}^{n+1}, \breve{\ww}^{n+1}\right)\nonumber\\
&+\left(\breve{\uu}_*^{n+1}\cdot \nabla \tilde{\dd}^{n+1},\tilde{\ww}^{n+1}\right)
-\frac{1}{\lambda}\left(\tilde{\uu}^{n+1} \cdot \nabla \tilde{\uu}^{n+1} ,\breve{\uu}_*^{n+1}\right)
,
\end{align}
\begin{align}
B^{n+1} =&  \left(\frac{2s^{n}}{\Delta t}-\frac{s^{n-1}}{2\Delta t}\right) \exp\left(-\frac{t^{n+1}}{T}\right)
+ \left( \tilde{\uu}^{n+1}\cdot \nabla \tilde{\dd}^{n+1},\hat{\ww}^{n+1}\right)\nonumber\\
&- \left(\hat{\uu}_*^{n+1} \cdot \nabla \tilde{\dd}^{n+1},\tilde{\ww}^{n+1}\right)
+\frac{1}{\lambda}\left(\tilde{\uu}^{n+1} \cdot \nabla \tilde{\uu}^{n+1} ,\hat{\uu}_*^{n+1}\right).
\end{align}

\section{Numerical Experiments}\label{sec:result}
In this section, we carry out several numerical tests to show the accuracy and stability of the proposed PCSAV Scheme \ref{2stPCSAV} and PCSAV-ECT Scheme \ref{2stPCSAV-ect}. For the spatial discretization, we use the Taylor-Hood element pair $P^2-P^1$ to solve the system, i.e., the director field $\dd$ and velocity $\uu$ are approximated with continuous $P^2$ finite element and the pressure $p$ is approximated with continuous $P^1$ finite element. Furthermore, we compare the CPU time between PCSAV and PCSAV-ECT scheme to verify the superior computational efficiency of PCSAV-ECT scheme.

\begin{example}\label{621}
In this example, we prescribe the right-hand sides of equations \eqref{eq:EL-1} and \eqref{eq:EL-3} such that the following functions be an exact solution
\begin{numcases}{}
\dd=\left(\cos\left(a\right),\sin\left(a\right)\right)^t,\quad\text{where}\;a:=0.25\pi\left(1-\cos\left(2\pi x\right)\right)
+\pi t,\nonumber\\
\uu=\left(\sin(\pi x)^2\sin(2\pi y)\sin(t),-\sin(2\pi x)\sin(\pi y)^2\sin(t)\right)^t,\nonumber\\
p=\left(xy-0.25\right)\cos(\pi t),\nonumber
\end{numcases}
in $\Omega = [0,1]^2$, with parameters $\nu=0.01, \lambda=0.1$, $\gamma=1$, $\varepsilon=0.01$ and $T=0.2$.
\end{example} 

For the temporal convergence analysis, the spatial grid was fixed at $h = 1/200$, and the time step was successively halved from an initial value of $0.05$. The $L^2$ norm errors between the numerical and exact solutions were computed for the velocity field $\uu$, director field $\dd$, and pressure $p$.
As shown in Tables \ref{tab:PCSAV-temporal-convergence} and \ref{tab:PCSAV-ect-temporal-convergence}, all three variables exhibit second-order accuracy in time.
To further verify the spatial accuracy of the proposed schemes, the time step was fixed at $\Delta t = 0.0001$, and the spatial mesh was successively refined from an initial size of $h = 1/20$ by halving at each refinement level.
The expected convergence rates are of second order for $\dd$ in the $H^1$ semi-norm, $\uu$ in the $H^1$ semi-norm, and $p$ in the $L^2$ norm.
The numerical results, summarized in Table \ref{tab:PCSAV-spatial-convergence} and \ref{tab:PCSAV-ect-spatial-convergence}, confirm that the proposed schemes achieve second-order accuracy in spatial discretization.

\begin{table}[htbp]
\centering
\caption{\noindent
Example \ref{621}, temporal convergence test by the PCSAV Scheme.}
\label{tab:PCSAV-temporal-convergence}
\begin{tabular}{|c||c|c||c|c||c|c|}
\hline
$\Delta t$ 
& $\Vert \dd(t^{n+1})-\dd^{n+1}\Vert$ & rate 
& $\Vert \uu(t^{n+1})-\uu^{n+1}\Vert$ & rate 
& $\Vert p(t^{n+1})-p^{n+1}\Vert$ & rate \\ 
\hline
$0.05$ & 0.002112271486 & --- & 0.0004804149703 & --- & 0.007190496554 & --- \\ 
$0.025$ & 0.0005320806841 & 1.99 & 0.000127411606 & 1.91 & 0.002995586478 & 1.26 \\ 
$0.0125$ & 0.00012842493 & 2.05 & 3.351759872e-05 & 1.93 & 0.0007849692449 & 1.93 \\ 
$0.00625$ & 3.153682278e-05 & 2.03 & 8.764377154e-06 & 1.94 & 0.0001905359879 & 2.04 \\ 
\hline
\end{tabular}
\end{table}

\begin{table}[htbp]
\centering
\caption{\noindent
Example \ref{621}, temporal convergence test by the PCSAV-ECT Scheme.}
\label{tab:PCSAV-ect-temporal-convergence}
\begin{tabular}{|c||c|c||c|c||c|c|}
\hline
$\Delta t$ 
& $\Vert \dd(t^{n+1})-\dd^{n+1}\Vert$ & rate 
& $\Vert \uu(t^{n+1})-\uu^{n+1}\Vert$ & rate 
& $\Vert p(t^{n+1})-p^{n+1}\Vert$ & rate \\ 
\hline
$0.05$ & 0.002112274066 & --- & 0.0004826303364 & --- & 0.007229155041 & --- \\ 
$0.025$ & 0.0005320803668 & 1.99 & 0.0001274986405 & 1.92 & 0.003002408943 & 1.27 \\ 
$0.0125$ & 0.0001284248925 & 2.05 & 3.350616432e-05 & 1.93 & 0.0007869641031 & 1.93 \\ 
$0.00625$ & 3.153682042e-05 & 2.03 & 8.757893628e-06 & 1.94 & 0.0001911005545 & 2.04 \\ 
\hline
\end{tabular}
\end{table}

\begin{table}[htbp]
\centering
\caption{\noindent
Example \ref{621}, spatial convergence test by the PCSAV Scheme.}
\label{tab:PCSAV-spatial-convergence}
\begin{tabular}{|c||c|c||c|c||c|c|}
\hline
$h$ 
& $\Vert \nabla\dd(t^{n+1})-\nabla\dd^{n+1}\Vert$ & rate 
& $\Vert \nabla\uu(t^{n+1})-\nabla\uu^{n+1}\Vert$ & rate 
& $\Vert p(t^{n+1})-p^{n+1}\Vert$ & rate \\ 
\hline
$\frac{1}{20}$ & 0.02390803842 & --- & 0.009717436416 & --- & 0.0003171025928 & --- \\ 
$\frac{1}{40}$ & 0.006020579636 & 1.99 & 0.001817095916 & 2.42 & 8.045013385e-05 & 1.98 \\ 
$\frac{1}{80}$ & 0.00150892181 & 2.00 & 0.0004189245682 & 2.12 & 2.018082949e-05 & 2.00 \\ 
$\frac{1}{160}$ & 0.000377602901 & 2.00 & 0.0001025831416 & 2.03 & 5.021425135e-06 & 2.01 \\ 
\hline
\end{tabular}
\end{table}

\begin{table}[htbp]
\centering
\caption{\noindent
Example \ref{621}, spatial convergence test by the PCSAV-ECT Scheme.}
\label{tab:PCSAV-ect-spatial-convergence}
\begin{tabular}{|c||c|c||c|c||c|c|}
\hline
$h$ 
& $\Vert \nabla\dd(t^{n+1})-\nabla\dd^{n+1}\Vert$ & rate 
& $\Vert \nabla\uu(t^{n+1})-\nabla\uu^{n+1}\Vert$ & rate 
& $\Vert p(t^{n+1})-p^{n+1}\Vert$ & rate \\ 
\hline
$\frac{1}{20}$ & 0.02390803842 & --- & 0.00972741342 & --- & 0.0003175228991 & --- \\ 
$\frac{1}{40}$ & 0.006020579635 & 1.99 & 0.001817704732 & 2.42 & 8.051509003e-05 & 1.98 \\ 
$\frac{1}{80}$ & 0.00150892181 & 2.00 & 0.000418960855 & 2.12 & 2.019660692e-05 & 2.00 \\ 
$\frac{1}{160}$ & 0.0003776029009 & 2.00 & 0.0001025838076 & 2.03 & 5.025226243e-06 & 2.01 \\ 
\hline
\end{tabular}
\end{table}

\begin{example}\label{62}
In this example, we apply the PCSAV Scheme \ref{2stPCSAV} and PCSAV-ECT Scheme \ref{2stPCSAV-ect} for numerically solving the modified Ericksen-Leslie model \eqref{eq:PEL-1}-\eqref{eq:PEL-3} with $\Omega = [-1,1]^2$, $\nu=0.1, \lambda=1, \gamma=1$, and the initial condition:
\begin{numcases}{}
\dd_0=\left(\sin\left(a\right),\cos\left(a\right)\right)^t,\quad\text{where}\;a:=2.0\pi\left(\cos\left(x\right)-\sin\left(y\right)\right)\nonumber\\
\uu_0=\left(0,0\right)^t,\nonumber\\
p_0=0.\nonumber
\end{numcases}
\end{example} 

We verify the convergence rate of the numerical schemes by approximating the smooth solutions of \eqref{eq:sav-2}-\eqref{eq:sav-5} under the same initial conditions in \cite{24Becker}. 
We run the numerical tests on four successively refined meshes with mesh size $h=\frac{2\sqrt{2}}{5\times 2^{l-1}}, l=1,..,4$. 
We also take a
constant ratio of the time step and mesh size so that 
$\Delta t = \frac{0.005}{2\sqrt{2}}h$. 
Then we calculate the Cauchy difference of the solutions at two successive levels, namely $v^{n+1}_{l+1}-v^{n+1}_{l}$ at $t^{n+1}=(n+1)\Delta t=0.1$, where $v= \dd_1, \dd_2, \uu_1, \uu_2$, $l=1,2,3$. 
The Cauchy error results computed by the PCSAV scheme are listed in Table \ref{tab:PROSAV}. 
The results demonstrate that the PCSAV scheme exhibits second-order convergence in time.
For the PCSAV-ECT scheme, the Cauchy error result listed in Table \ref{tab:PROSAV-ECT}, also showing second order in time convergent.

\begin{table}[tp]
\centering
{\small
\caption{\noindent
Example \ref{62}, Cauchy difference of numerical solutions calculated by the PCSAV Scheme \ref{2stPCSAV}.}
\label{tab:PROSAV}
\begin{tabular}{|c||c|c||c|c||c|c|}
\hline
$l$& $\Vert \nabla \dd^{n+1}_{l+1}-\nabla \dd^{n+1}_{l}\Vert$ & rate&  $\Vert \nabla \uu^{n+1}_{l+1}-\nabla \uu^{n+1}_{l}\Vert$ & rate& $\Vert p^{n+1}_{l+1}-p^{n+1}_{l}\Vert$ & rate \\
\hline
1  &0.4164495743  &   --- & 2.700120977  &   --- & 0.692172455 &   ---  \\
2  & 0.09817322033 &  2.08  &  0.8160214813 &   1.73  & 0.2014638903 &  1.78  \\
3  & 0.02446193817  & 2.00  & 0.1571174499	 &   2.38  & 0.05370746945 &  1.91  \\
4  &  0.006176025916 & 1.99 &  0.03627691036 &   2.11  & 0.01420654234 &   1.92  \\
\hline
\end{tabular}
}
\end{table}

\begin{table}[tp]
\centering
{\small
\caption{\noindent
Example \ref{62}, Cauchy difference of numerical solutions calculated by the PCSAV-ECT Scheme \ref{2stPCSAV-ect}.}
\label{tab:PROSAV-ECT}
\begin{tabular}{|c||c|c||c|c||c|c|}
\hline
$l$& $\Vert \nabla \dd^{n+1}_{l+1}-\nabla \dd^{n+1}_{l}\Vert$ & rate&  $\Vert \nabla \uu^{n+1}_{l+1}-\nabla \uu^{n+1}_{l}\Vert$ & rate& $\Vert p^{n+1}_{l+1}-p^{n+1}_{l}\Vert$ & rate \\
\hline
1  &0.4162841217  &   --- & 2.600968688 &   --- & 0.6978122543 &   ---  \\
2  & 0.09806612727 &  2.09  &  0.7724048347 &   1.75  & 	0.1922718919 &  1.86  \\
3  & 0.02445782532  & 2.00  & 0.1558344576 &   2.31  & 0.05102810161 &  1.91  \\
4  & 0.006175426523 & 1.99 &  0.0365693909 &   2.09  & 0.01394458104 &   1.87  \\
\hline
\end{tabular}
}
\end{table}

We perform the numerical simulations by PCSAV and PCSAV-ECT schemes with a time step size of $\Delta t=0.0025$.
Fig. \ref{fig:smooth_dir} and Fig. \ref{fig:smooth_u} show the evolution of the director field $\dd$ and velocity $\uu$ by PCSAV scheme, respectively.
Similarly, Fig. \ref{fig:NS-smooth_dir} and Fig. \ref{fig:NS-smooth_u} depict the profile state of the director field $\dd$ and velocity field $\uu$ at different times obtained by PCSAV-ECT scheme. 
Notably, the results of the PCSAV-ECT scheme are in perfect agreement with those of the PCSAV scheme, confirming their equivalence. 
Therefore, in the following numerical experiments, we only present the results of the PCSAV-ECT scheme.

To verify the energy stability of the PCSAV and PCSAV-ECT schemes, we present the time evolution of the kinetic energy $W_{kin}$, elastic energy $W_{ela}$, penalty energy $W_{pen}$, modified total energy $\tilde{W}$, and original energy $W_s$ with various penalty parameters $\varepsilon=0.2, 0.1, 0.05, 0.025$, as shown in Fig. \ref{fig:pcsav_energy} and \ref{fig:nspcsav_energy}.
We observe that both $\tilde{W}$ and $W_s$ exhibit a monotonically decreasing trend for all values of $\varepsilon$, indicating that the proposed schemes consistently preserve the energy dissipation property.
Moreover, as the penalty parameter $\varepsilon$ increases, the contribution of the penalty energy $W_{pen}$ becomes more significant, while the kinetic energy $W_{kin}$ is correspondingly diminished. 
This indicates that a large penalty parameter may overly constrain the system, thereby affecting the dynamic behavior and resulting in a loss of physical fidelity.
In contrast, when $\varepsilon$ is small, the modified energy $\tilde{W}$ is closer to the original energy $W_s$, which means the modified model \eqref{eq:PEL-1}-\eqref{eq:PEL-3} better approximates the original model \eqref{eq:SEL-1}-\eqref{eq:SEL-4}. This agrees with the theoretical expectation that as $\varepsilon\to0$, the modified model converges to the original one.
Overall, the results confirm that the PCSAV and PCSAV-ECT schemes maintain energy stability and performs well for a wide range of $\varepsilon$.

Furthermore, Fig. \ref{fig:diffdt_energy} presents the energy evolution of the PCSAV and PCSAV-ECT schemes under different time step sizes. It can be clearly observed that the modified total energy consistently exhibits monotonic decay, further strongly confirming the unconditional stability of the proposed schemes.

\begin{figure}[tp]
\begin{center}
\includegraphics[scale=0.24]{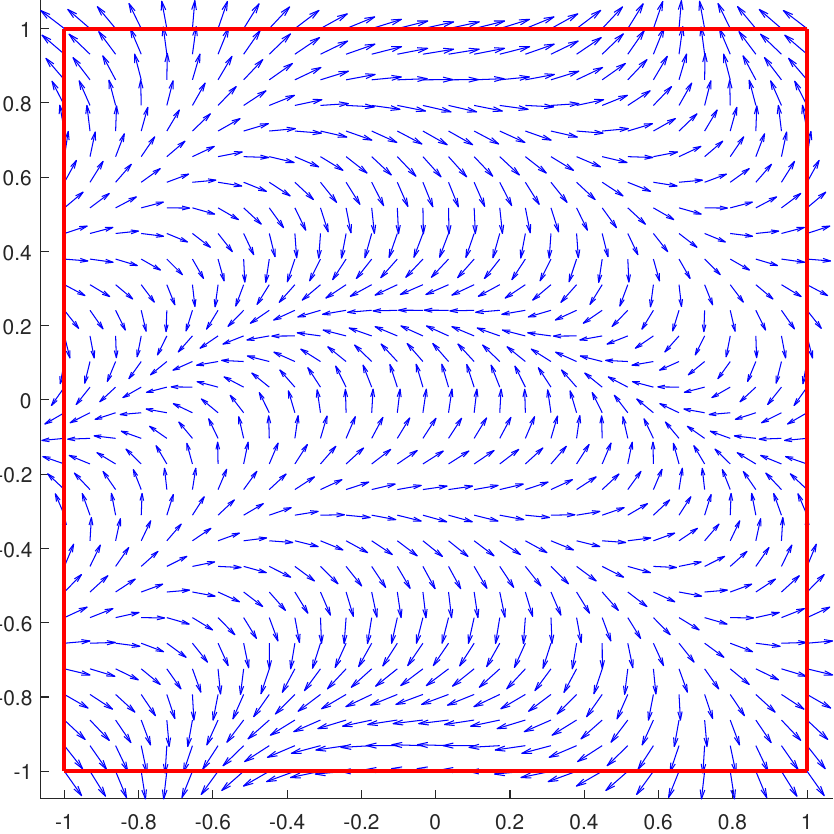}
\includegraphics[scale=0.24]{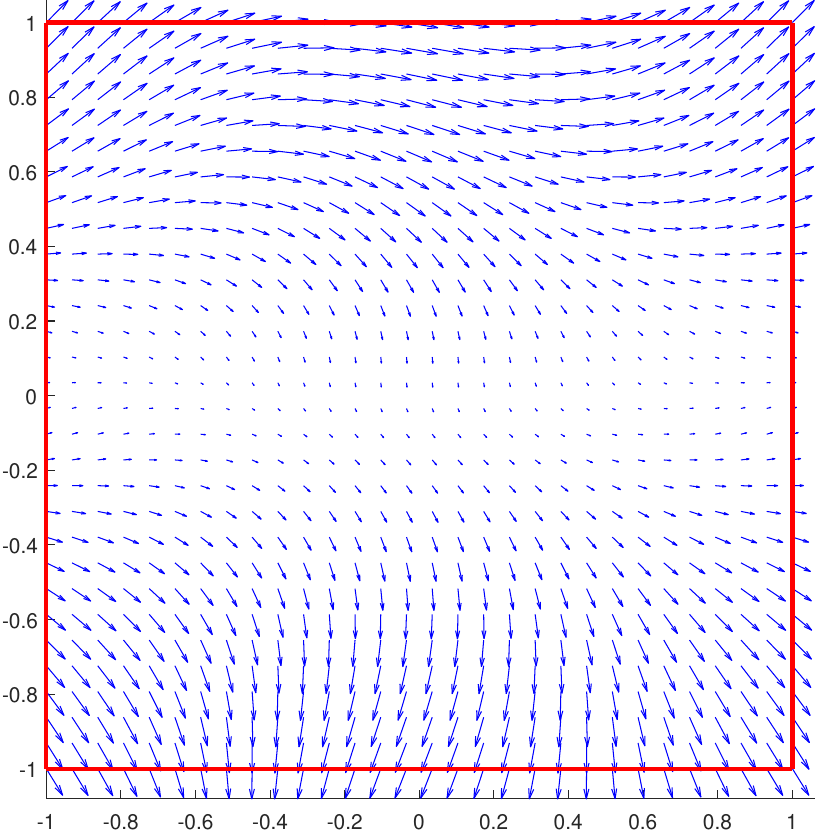}
\includegraphics[scale=0.24]{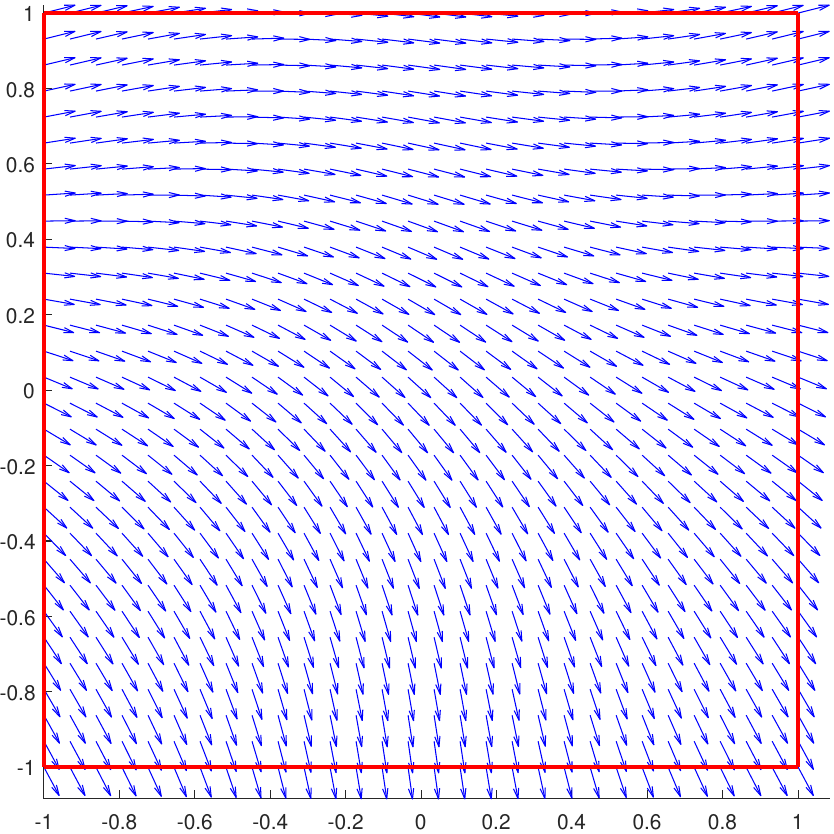}
\includegraphics[scale=0.24]{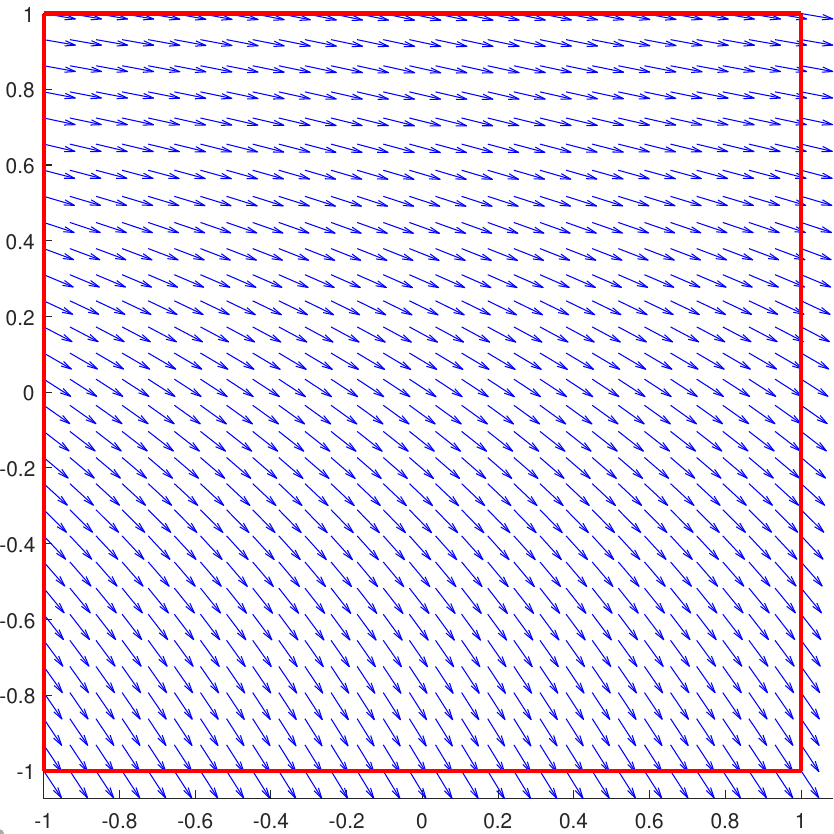}

\includegraphics[scale=0.24]{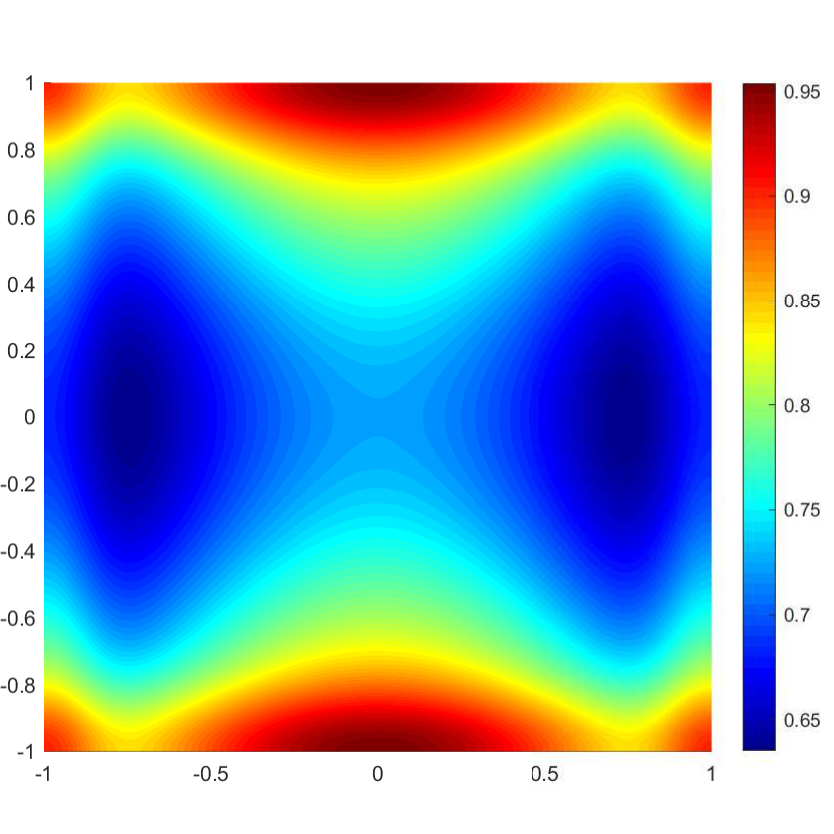}
\includegraphics[scale=0.24]{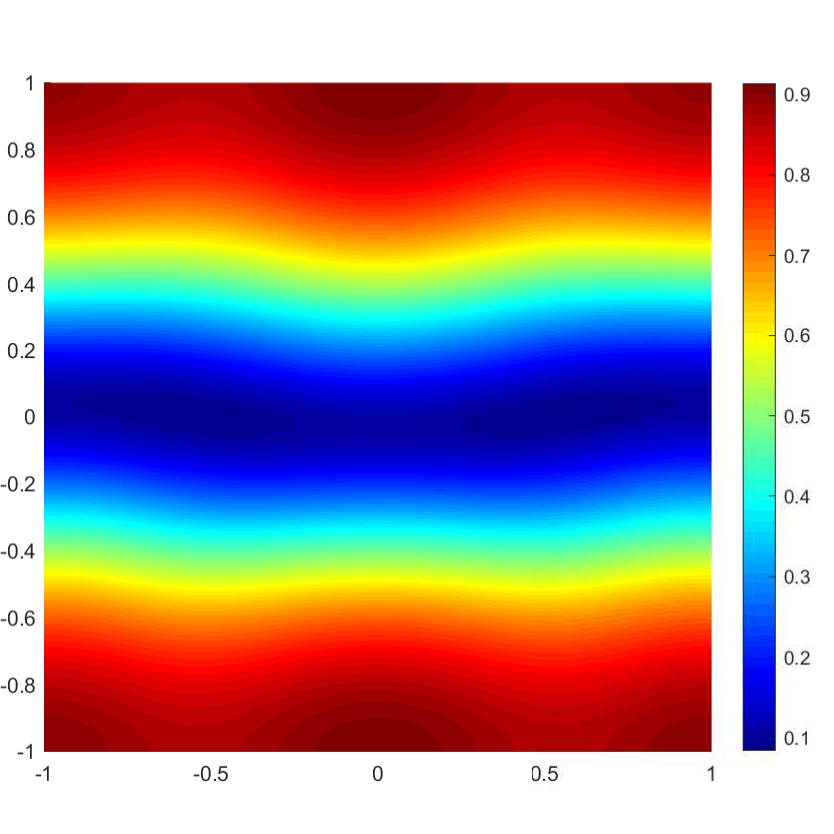}
\includegraphics[scale=0.24]{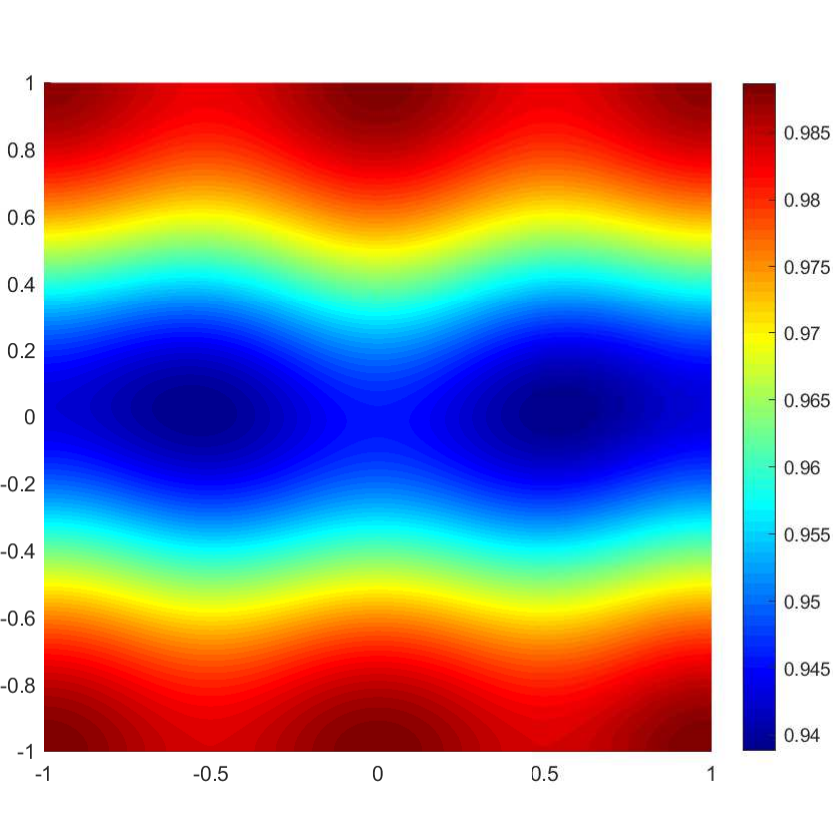}
\includegraphics[scale=0.24]{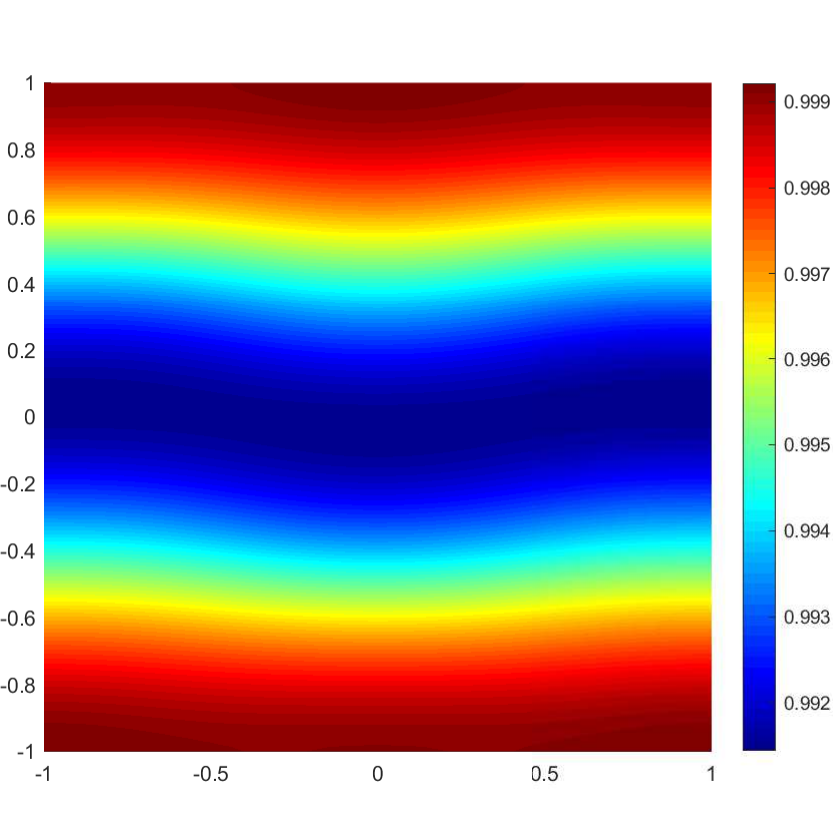}
\caption{Example \ref{62}, images of the director field $\dd$ (first row) and $\lvert\dd\rvert$ (second row) at $t=0.01, 0.1, 0.2, 0.4$ computed by PCSAV scheme.}
\label{fig:smooth_dir}
\end{center}
\end{figure}
\begin{figure}[tp]
\begin{center}
\includegraphics[scale=0.24]{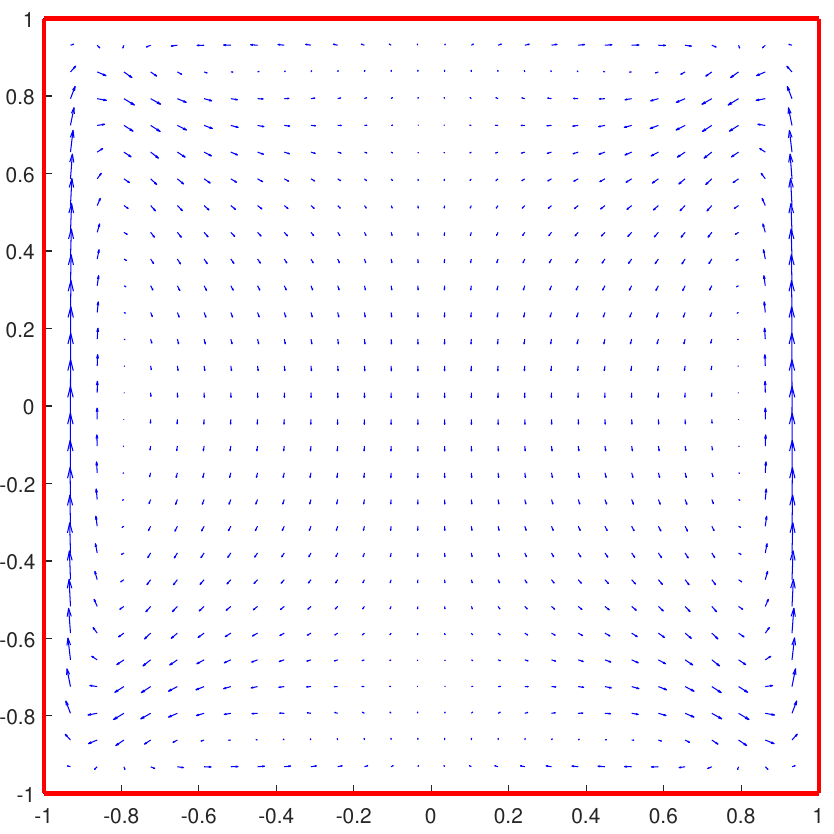}
\includegraphics[scale=0.24]{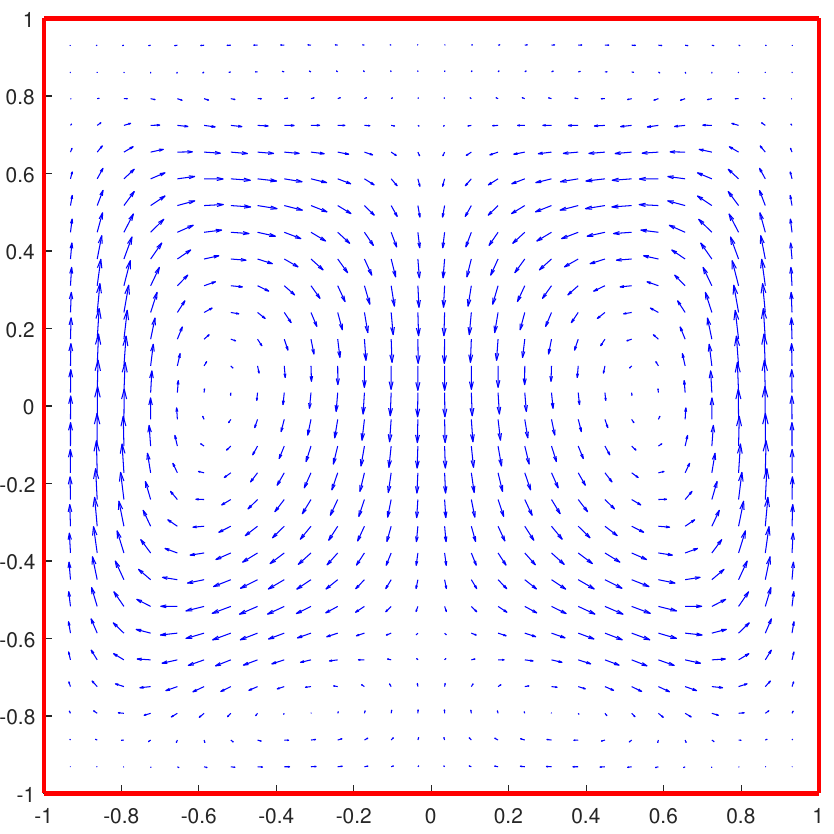}
\includegraphics[scale=0.24]{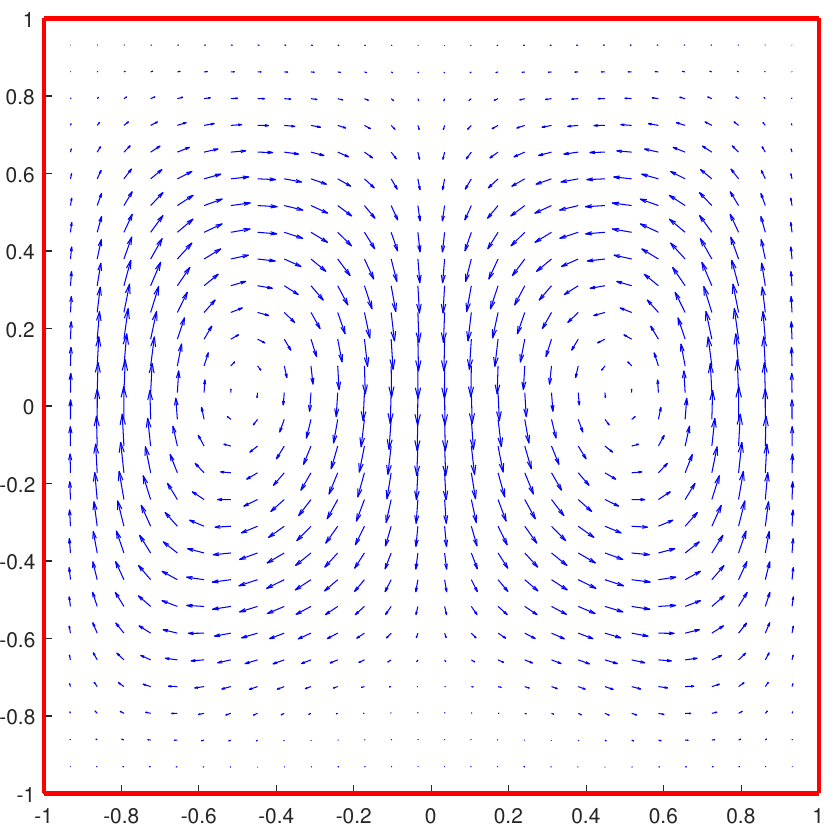}
\includegraphics[scale=0.24]{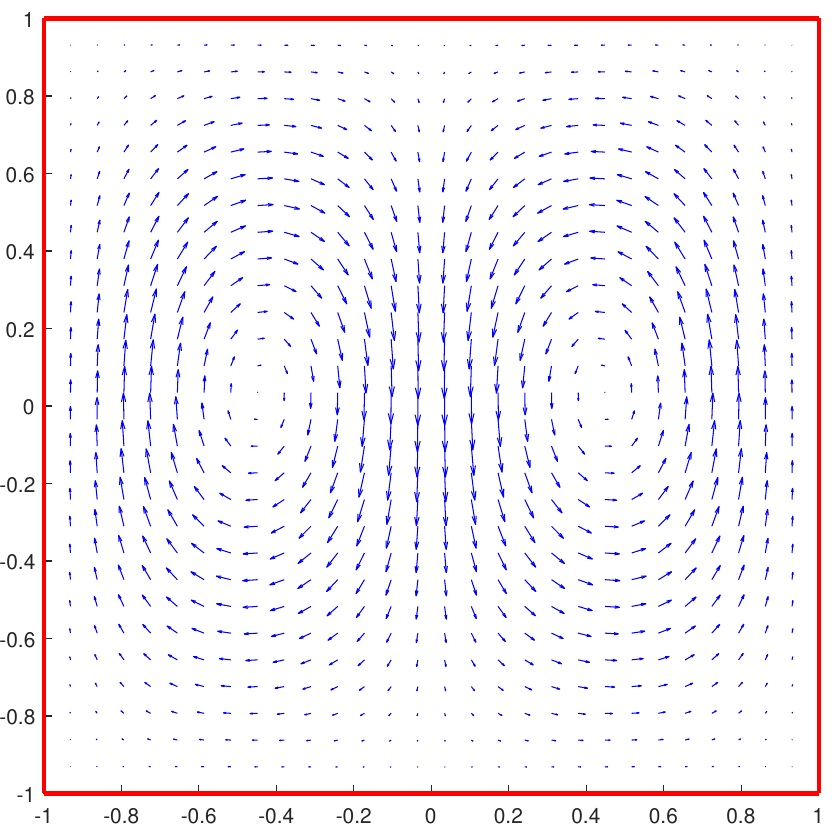}

\includegraphics[scale=0.24]{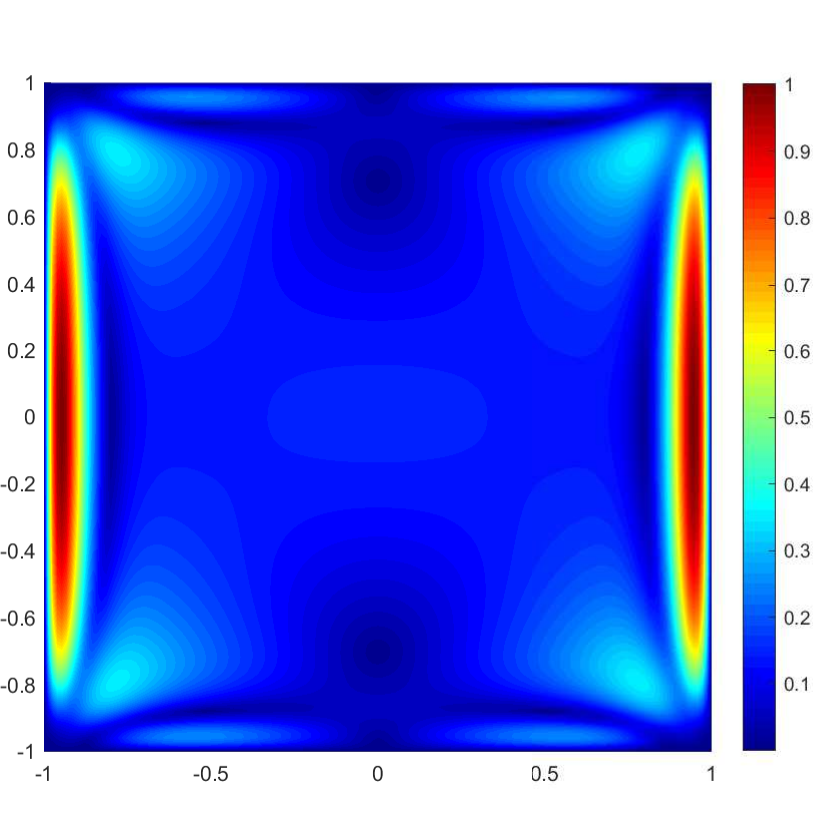}
\includegraphics[scale=0.24]{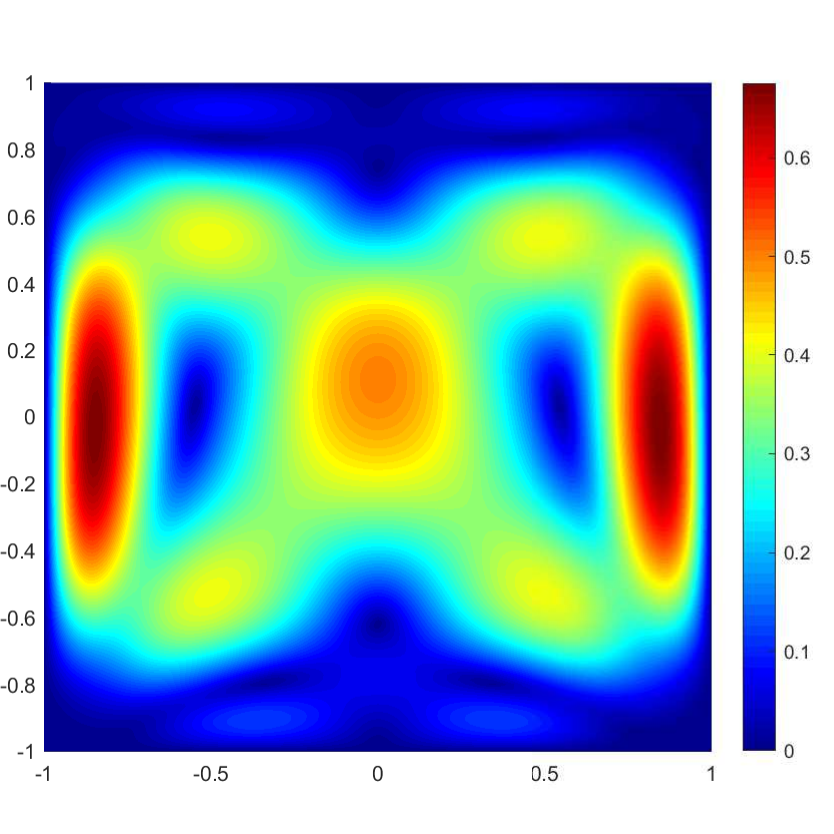}
\includegraphics[scale=0.24]{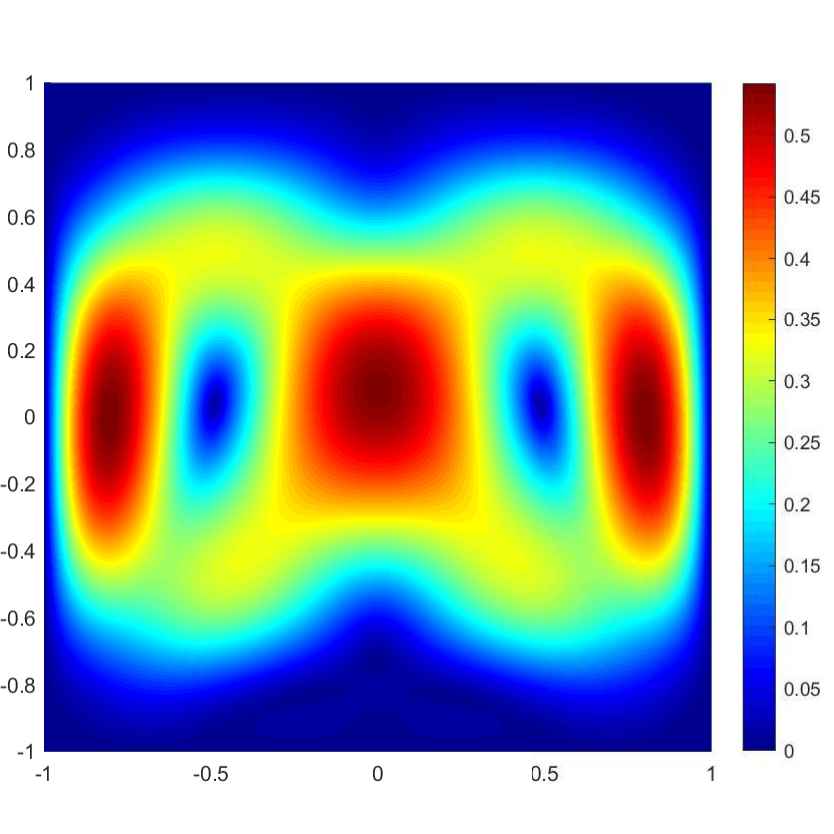}
\includegraphics[scale=0.24]{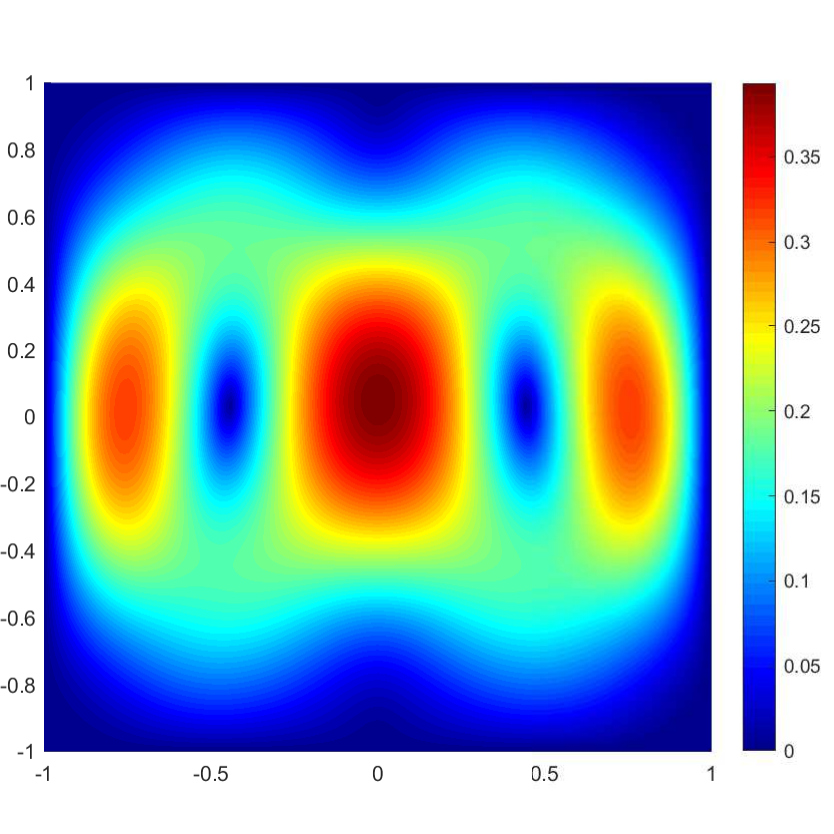}
\caption{Example \ref{62}, images of the velocity $\uu$ (first row) and $\lvert\uu\rvert$ (second row) at $t=0.01, 0.1, 0.2, 0.4$ computed by PCSAV scheme.}
\label{fig:smooth_u}
\end{center}
\end{figure}

\begin{figure}[tp]
\begin{center}
\includegraphics[scale=0.24]{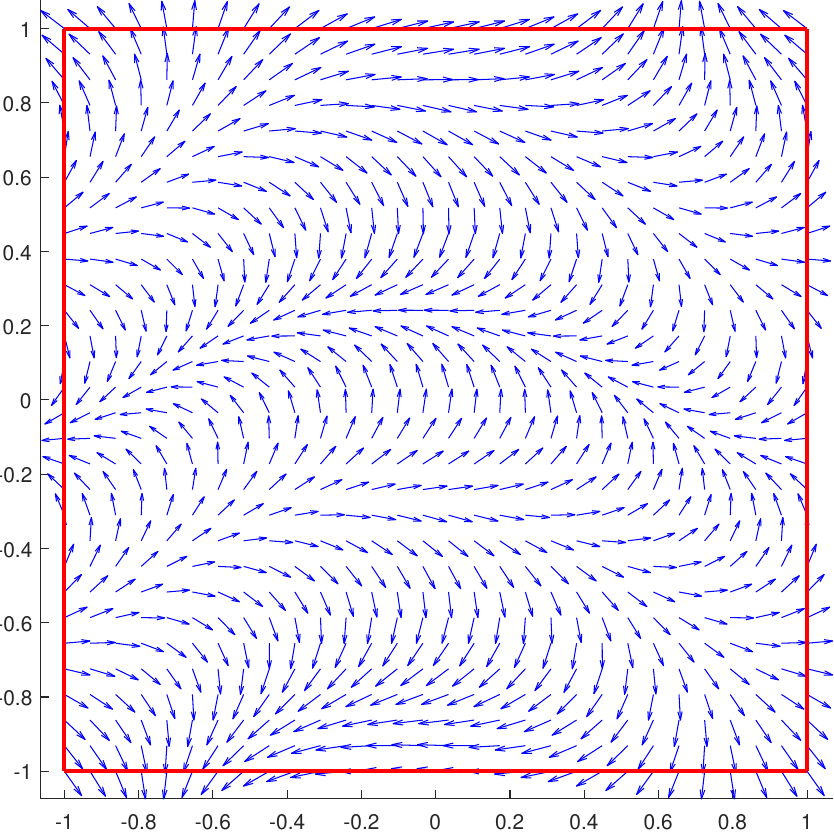}
\includegraphics[scale=0.24]{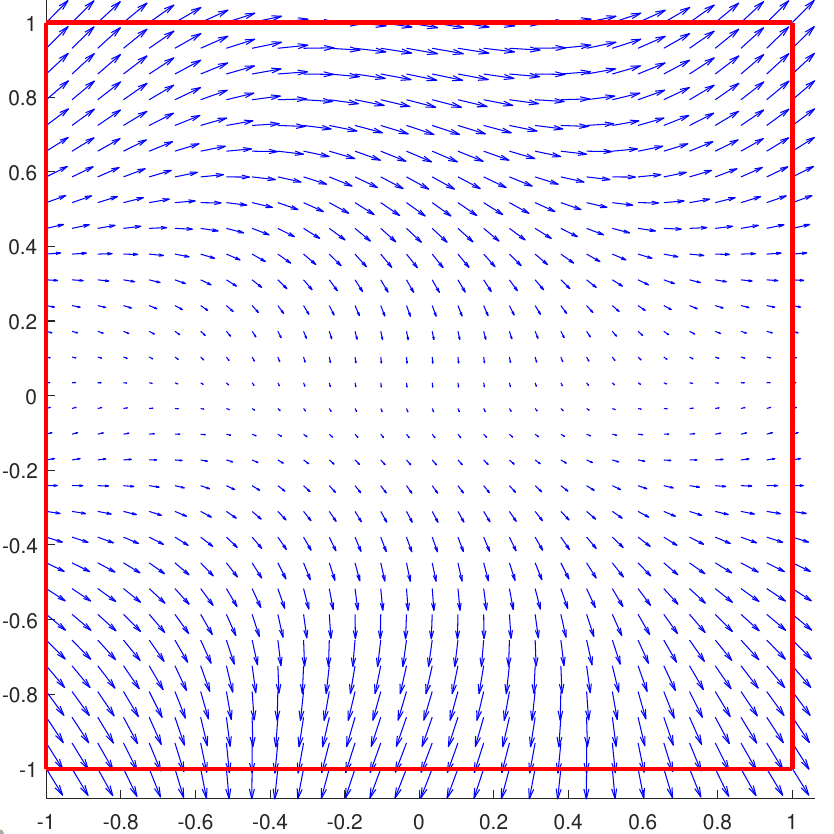}
\includegraphics[scale=0.24]{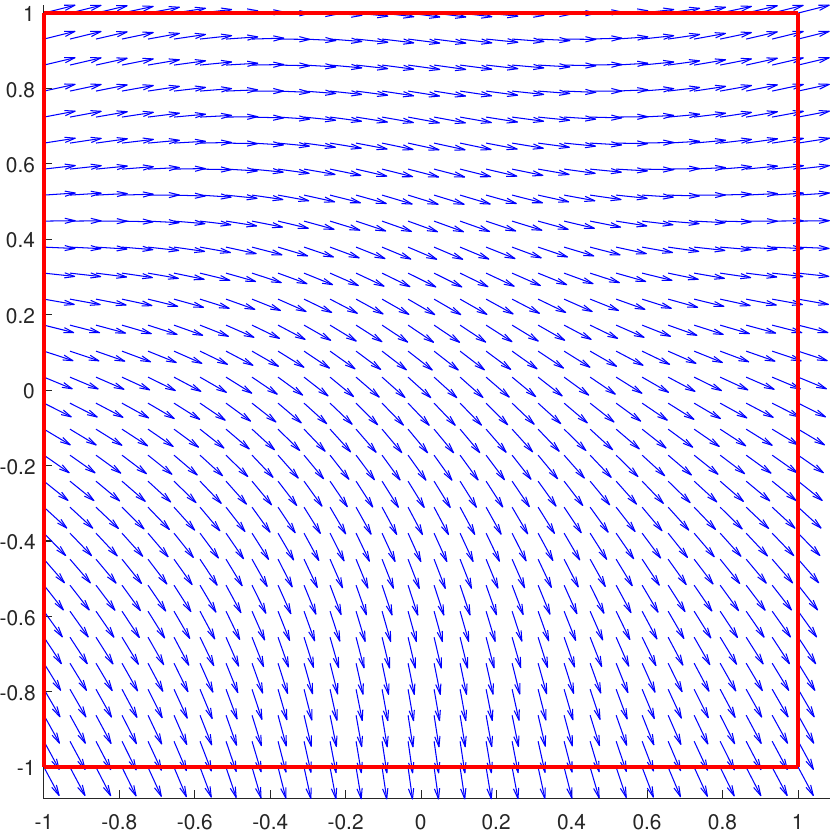}
\includegraphics[scale=0.24]{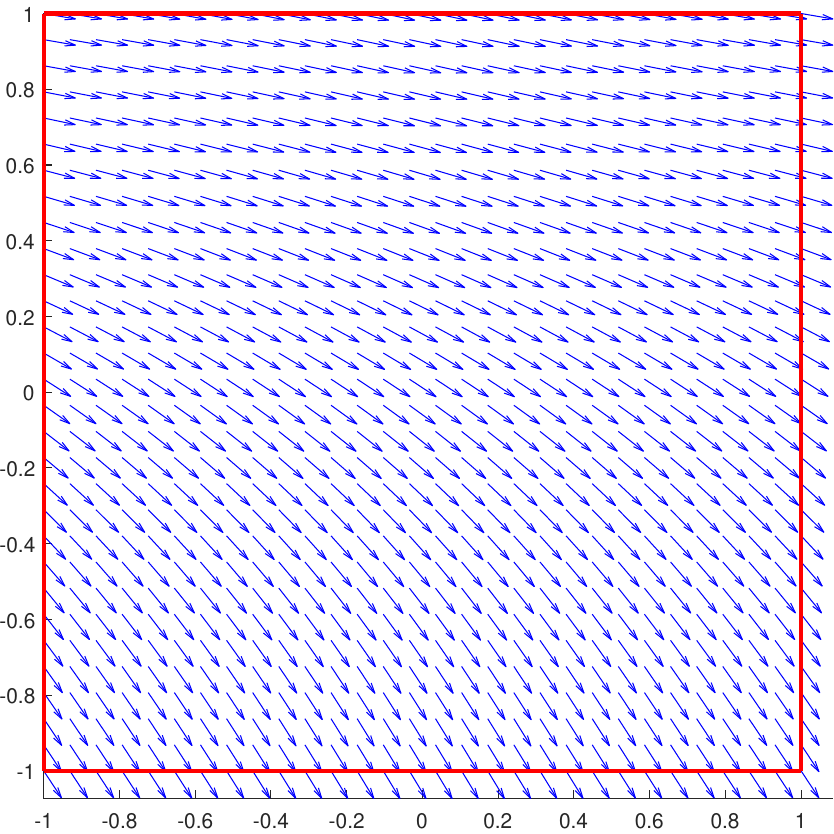}

\includegraphics[scale=0.24]{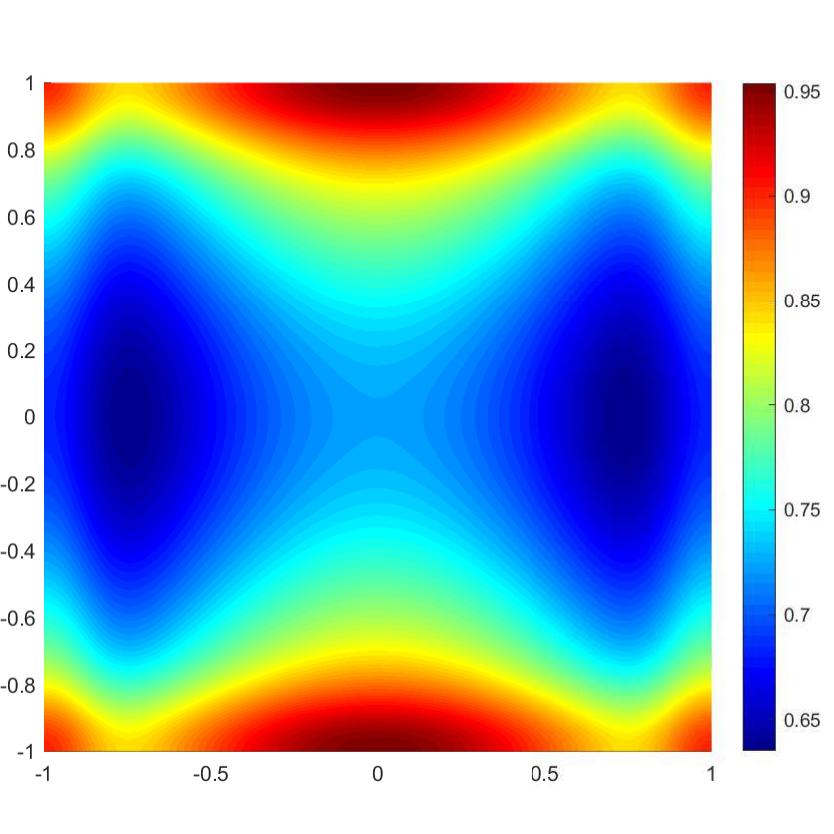}
\includegraphics[scale=0.24]{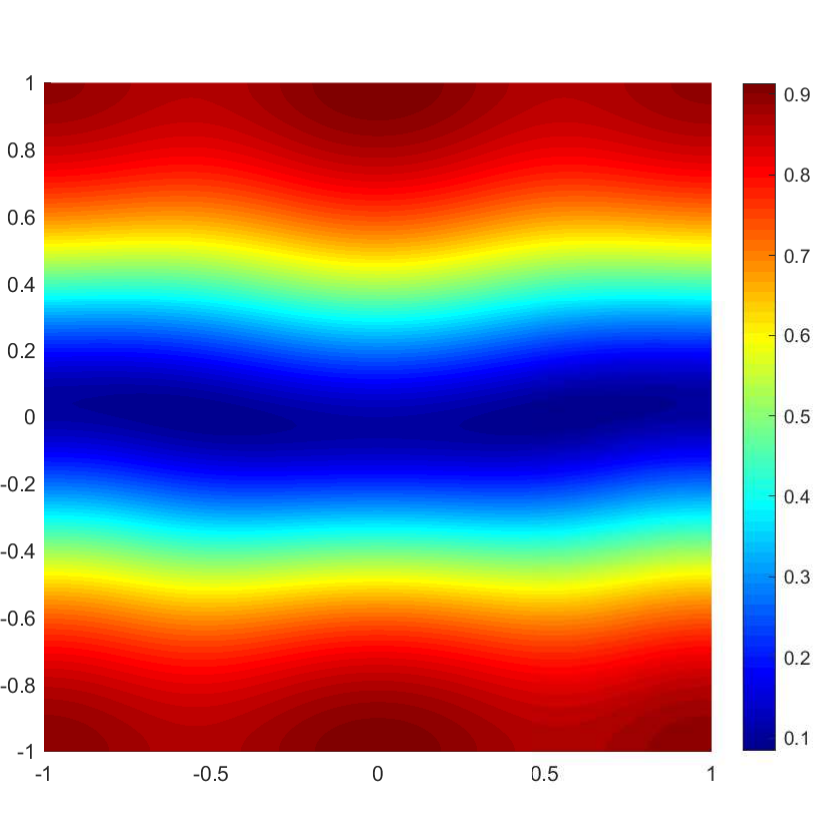}
\includegraphics[scale=0.24]{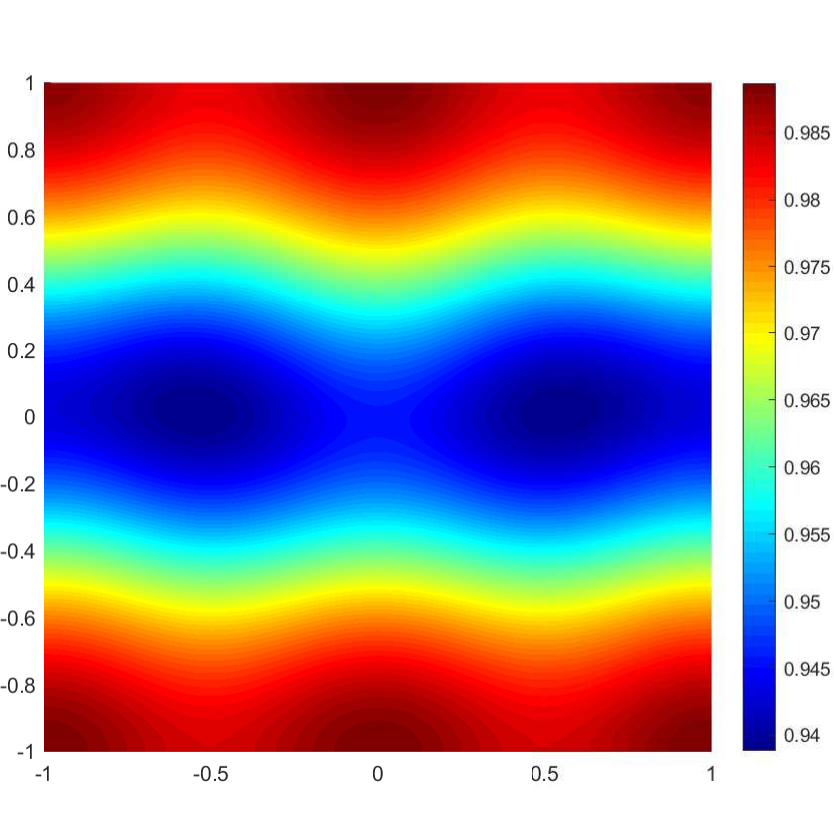}
\includegraphics[scale=0.24]{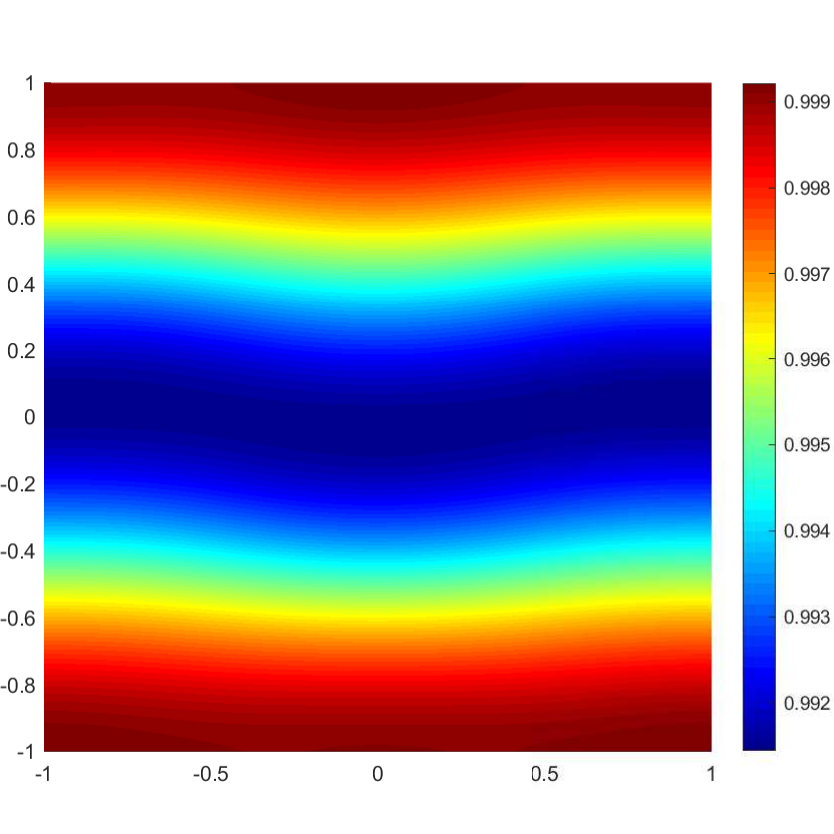}
\caption{Example \ref{62}, images of the director field $\dd$ (first row) and $\lvert\dd\rvert$ (second row) at $t=0.01, 0.1, 0.2, 0.4$ computed by PCSAV-ECT scheme.}
\label{fig:NS-smooth_dir}
\end{center}
\end{figure}
\begin{figure}[tp]
\begin{center}
\includegraphics[scale=0.24]{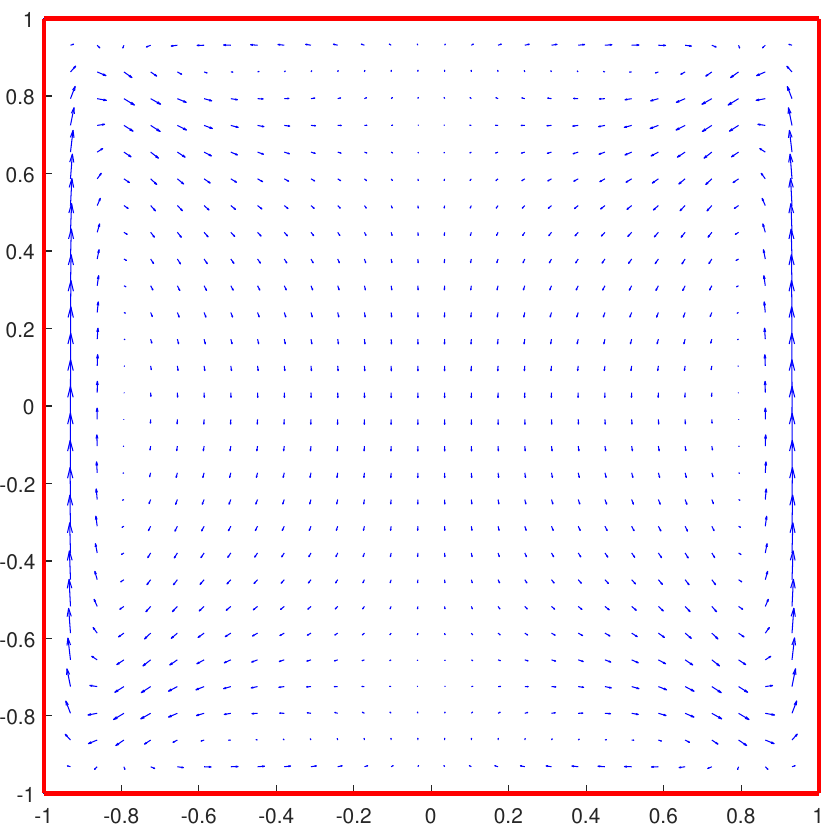}
\includegraphics[scale=0.24]{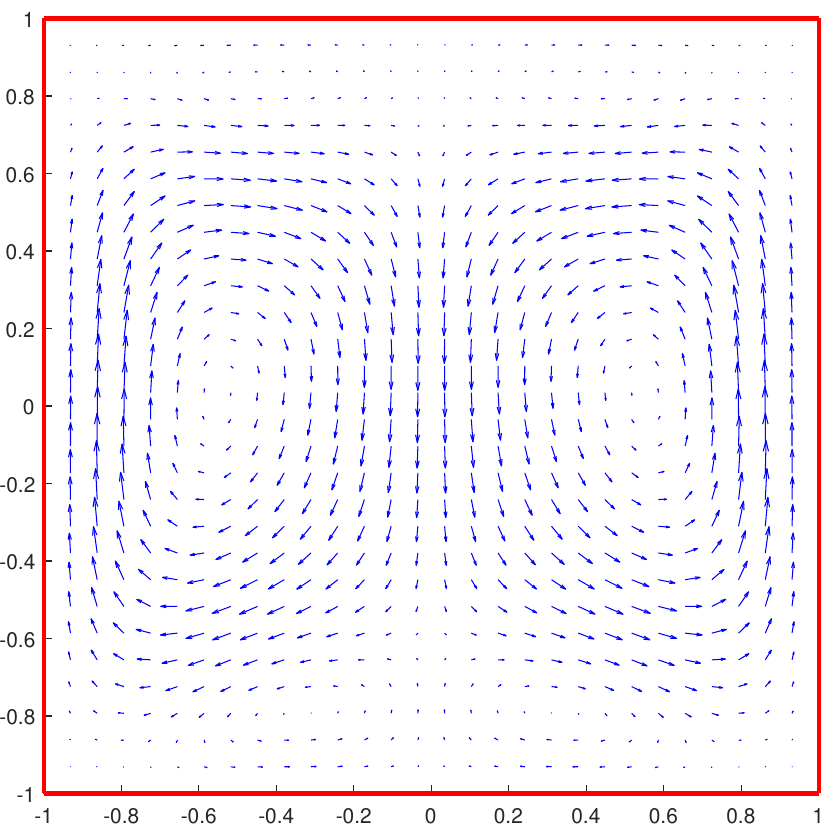}
\includegraphics[scale=0.24]{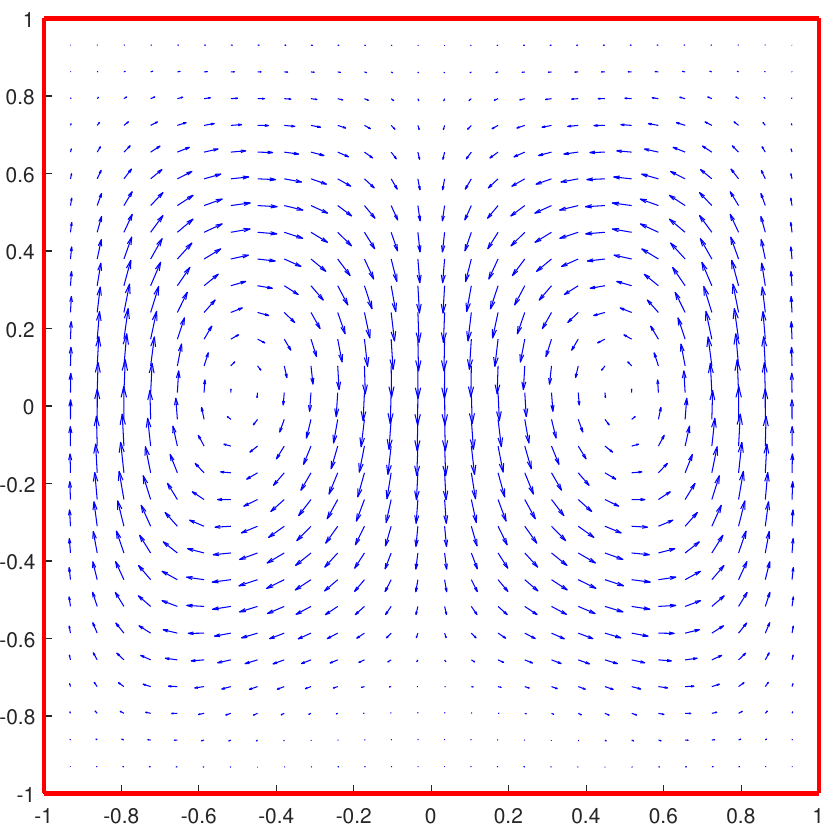}
\includegraphics[scale=0.24]{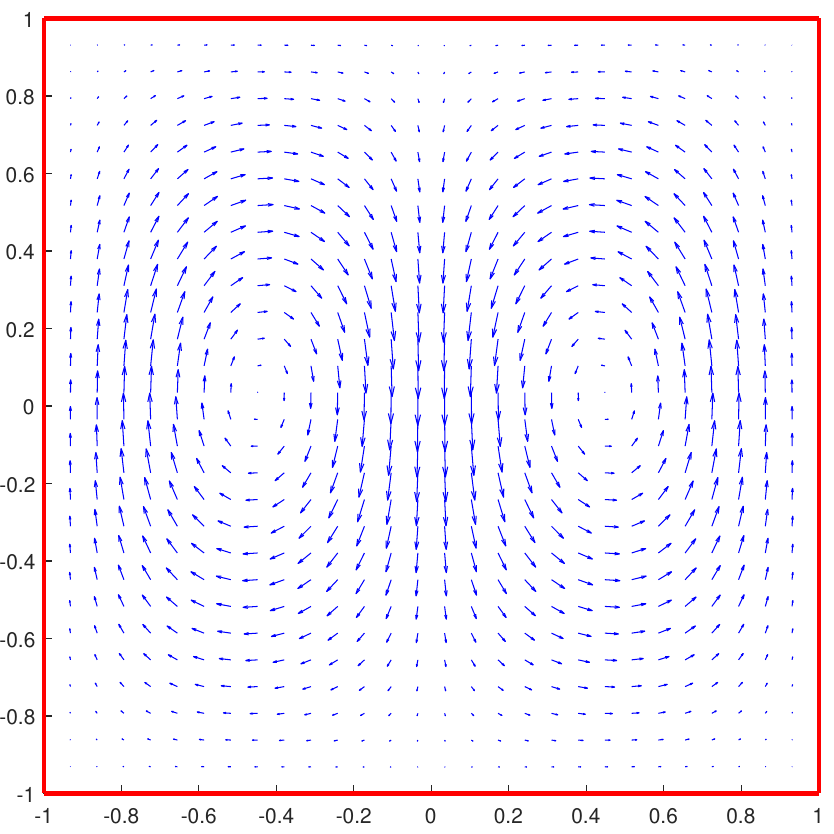}

\includegraphics[scale=0.24]{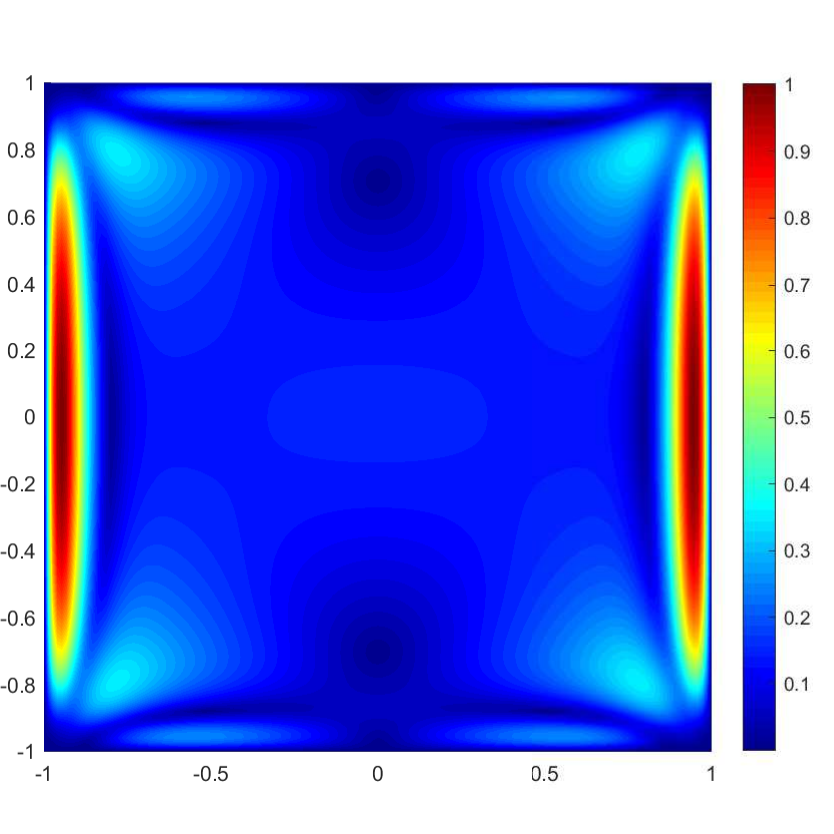}
\includegraphics[scale=0.24]{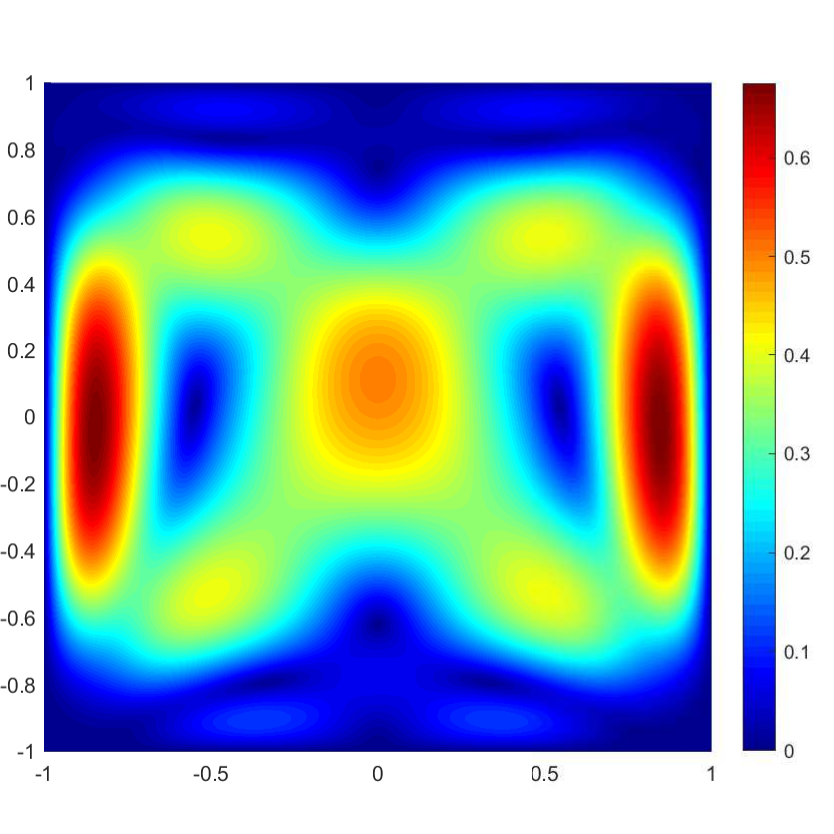}
\includegraphics[scale=0.24]{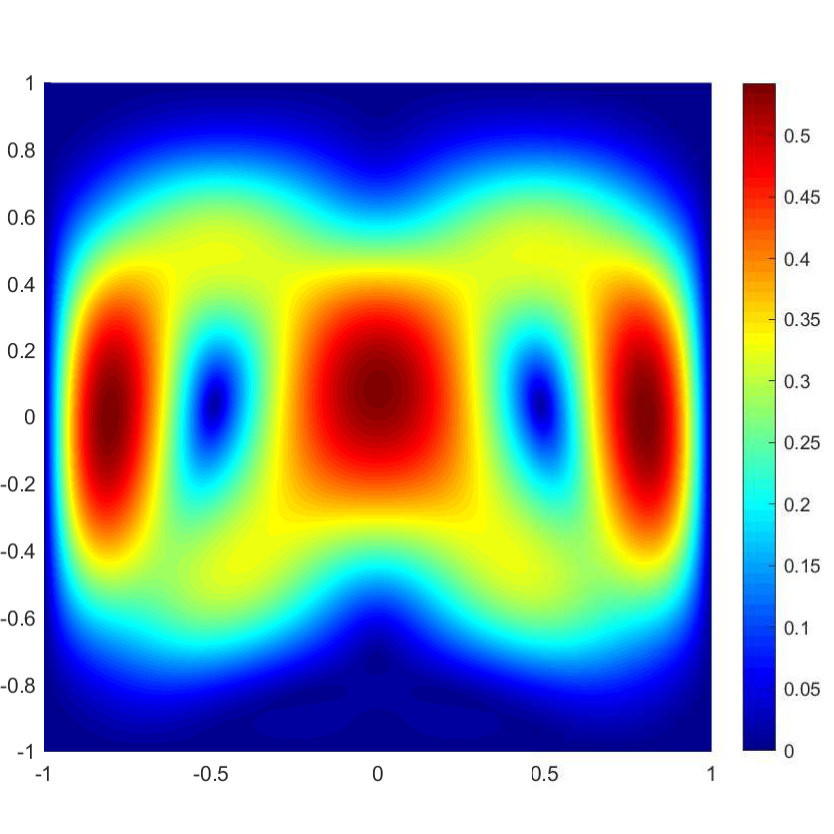}
\includegraphics[scale=0.24]{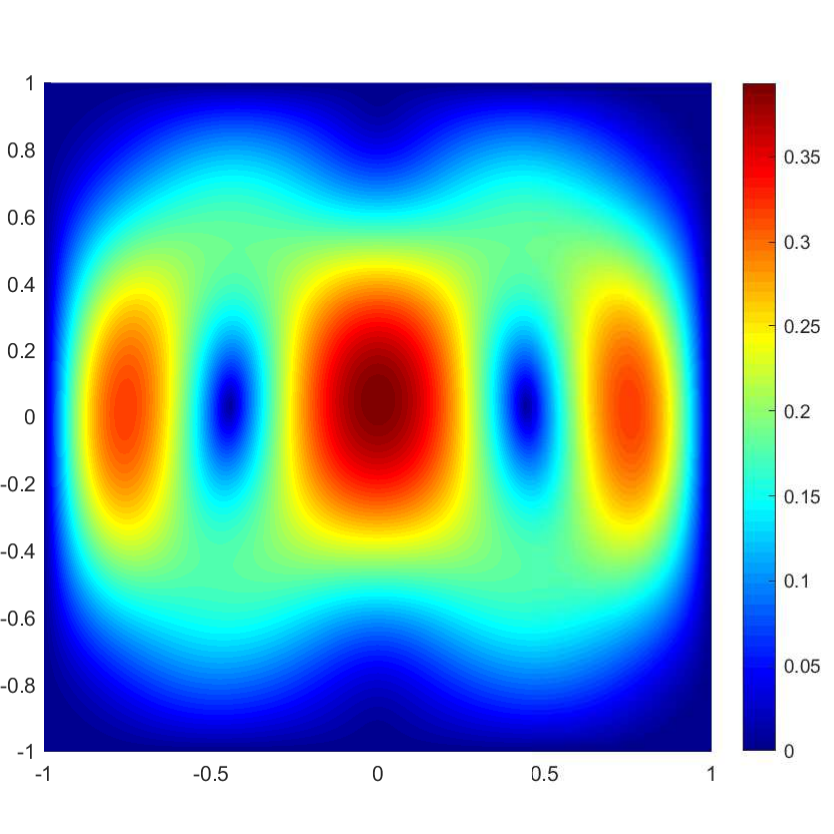}
\caption{Example \ref{62}, images of the velocity $\uu$ (first row) and $\lvert\uu\rvert$ (second row) at $t=0.01, 0.1, 0.2, 0.4$ computed by PCSAV-ECT scheme.}
\label{fig:NS-smooth_u}
\end{center}
\end{figure}

\begin{figure}[tp]
\begin{center}
\subfigure[]
{\label{fig:PCSAVKin_energy}\includegraphics[scale=0.45]{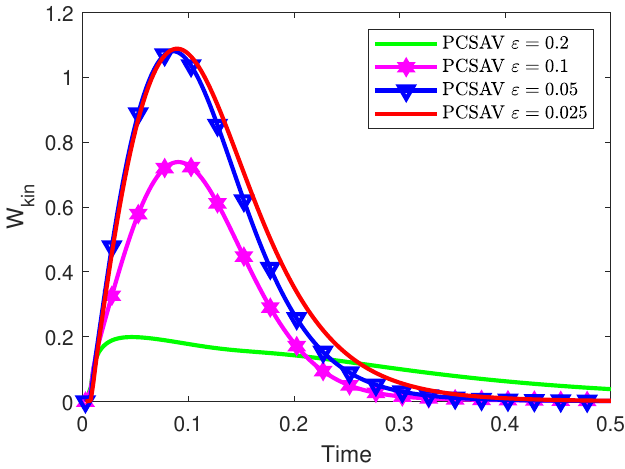}}	
\subfigure[]
{\label{fig:PCSAVela_energy}\includegraphics[scale=0.45]{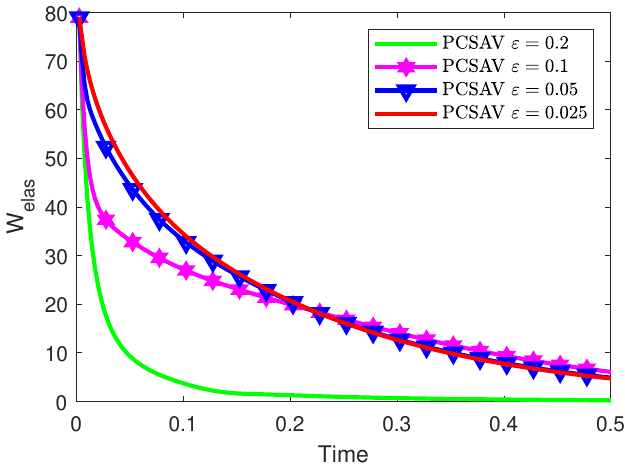}}	
\subfigure[]
{\label{fig:PCSAVpen_energy}\includegraphics[scale=0.45]{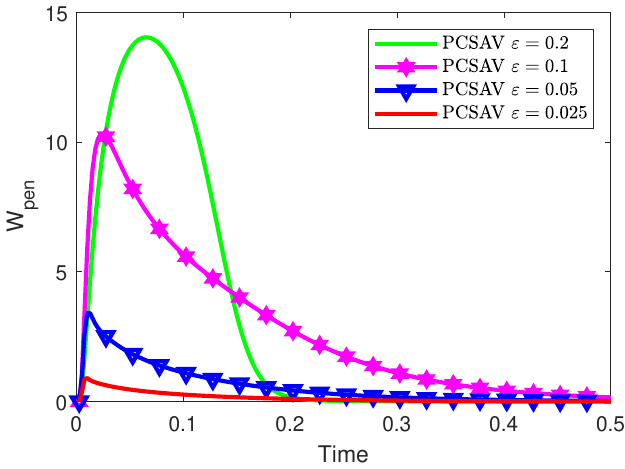}}	\\
\subfigure[]
{\label{fig:PCSAVenergy}\includegraphics[scale=0.45]{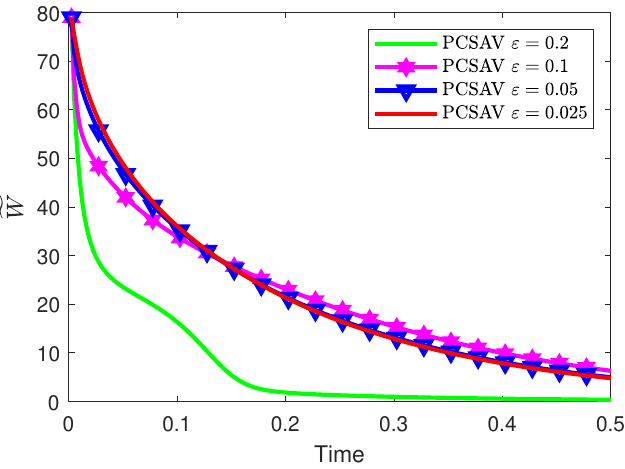}}
\subfigure[]
{\label{fig:PCSAVenergyori}\includegraphics[scale=0.45]{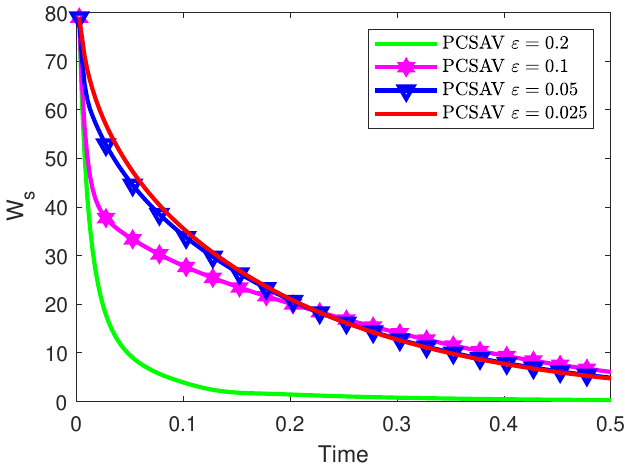}}	
\caption{Example \ref{62}, kinetic energy, elastic energy, penalty energy and modified energy computed by the PCSAV scheme.}
\label{fig:pcsav_energy}
\end{center}
\end{figure}

\begin{figure}[tp]
\begin{center}
\subfigure[]
{\label{fig:NSPCSAVKin_energy}\includegraphics[scale=0.45]{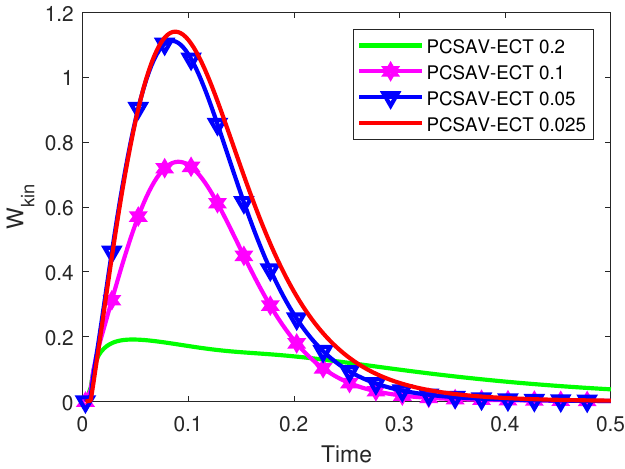}}	
\subfigure[]
{\label{fig:NSPCSAVela_energy}\includegraphics[scale=0.45]{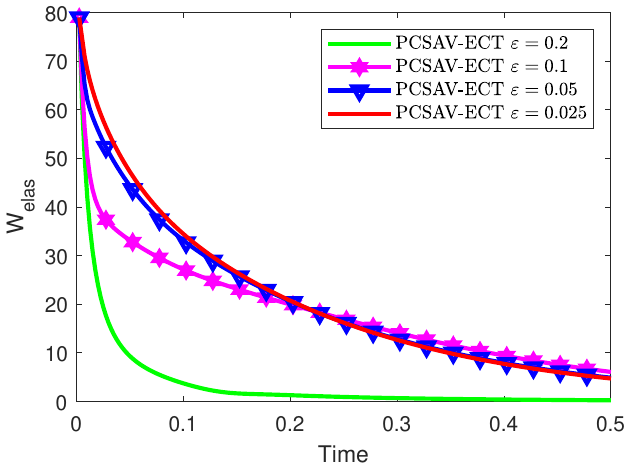}}	
\subfigure[]
{\label{fig:NSPCSAVpen_energy}\includegraphics[scale=0.45]{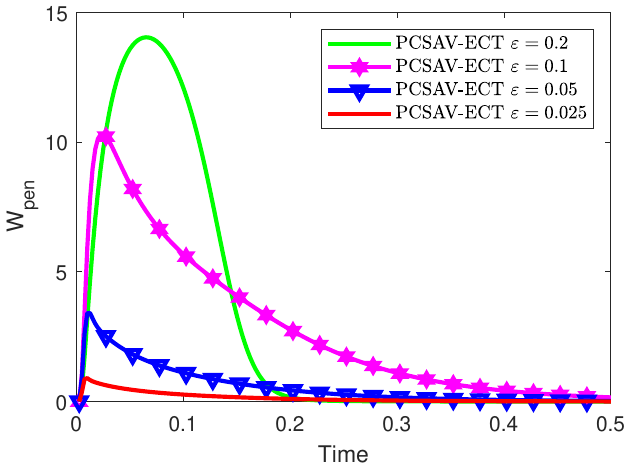}}\\
\subfigure[]
{\label{fig:NSPCSAVenergy}\includegraphics[scale=0.45]{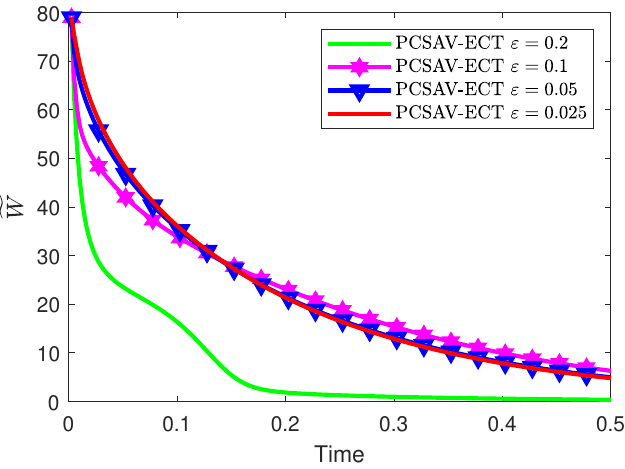}}
\subfigure[]
{\label{fig:NSPCSAVenergyori}\includegraphics[scale=0.45]{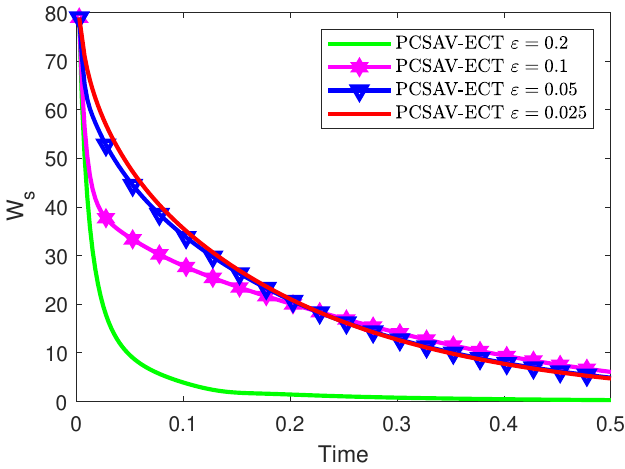}}	
\caption{Example \ref{62}, kinetic energy, elastic energy, penalty energy and modified energy computed by the PCSAV-ECT scheme.}
\label{fig:nspcsav_energy}
\end{center}
\end{figure}

\begin{figure}[tp]
\begin{center}
\subfigure[]
{\includegraphics[scale=0.6]{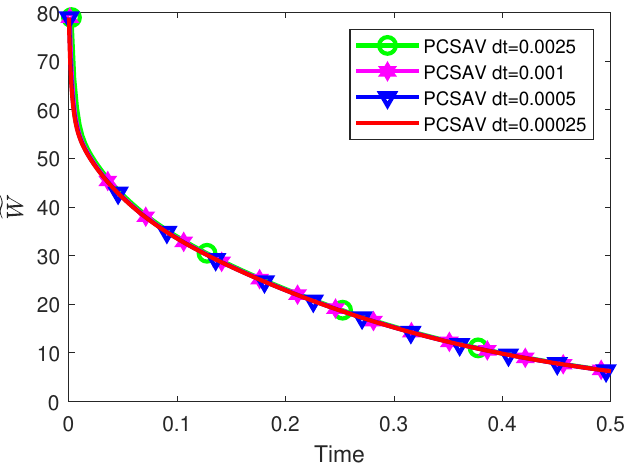}}	
\subfigure[]
{\includegraphics[scale=0.6]{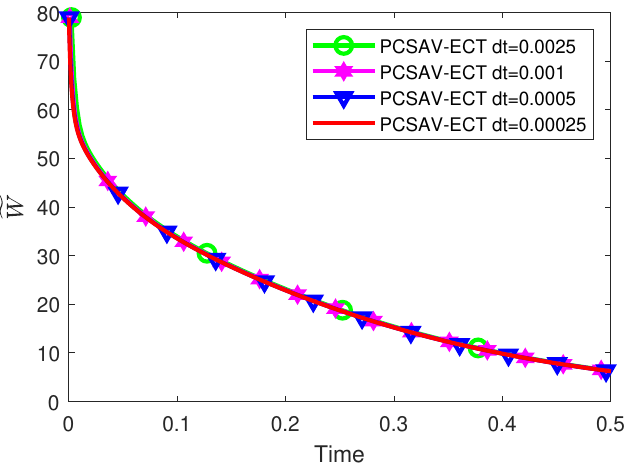}}
\caption{Example \ref{62}, modified energy computed by the PCSAV and PCSAV-ECT scheme.}
\label{fig:diffdt_energy}
\end{center}
\end{figure}

\begin{example}\label{63}
In this example, we apply the PCSAV Scheme \ref{2stPCSAV} and PCSAV-ECT Scheme \ref{2stPCSAV-ect} for numerically solving the modified Ericksen-Leslie model \eqref{eq:PEL-1}-\eqref{eq:PEL-3} with $\Omega=(-1,1)\times(-1,1)$, $\nu=1, \lambda=0.01, \gamma=1, \varepsilon=0.05$, and the initial condition \cite{23zheng}:
\begin{numcases}{}
\dd^0=\tilde{\dd}/\sqrt{\lvert\tilde{\dd}\rvert^2+0.05^2},\quad\text{where}\;\tilde{\dd}=(x^2+y^2-0.5^2,y),\nonumber\\
\uu^0=\left(0,0\right)^t,\nonumber\\
p^0=0.\nonumber
\end{numcases}
\end{example}

This numerical example simulates the evolution of the liquid crystal flows with singularities, which is an interesting physical phenomenon in liquid crystal modeling.
We use the PCSAV-ECT scheme to solve the modified Ericksen-Leslie model with a time step size of $\Delta t=0.0005$. 
In Fig. \ref{fig:2stsimpled} and Fig. \ref{fig:2stsimpled_dnorm}, we show the evolution of two defects at times $t=$0.01, 0.2, 0.32, 0.4 and 0.5.
The snapshots of the velocity field $\uu$ and its magnitudes at times $t=$0.01, 0.2, 0.32, 0.4 and 0.5 are shown in Fig. \ref{fig:2stsimpleu} and Fig. \ref{fig:2stsimpleu_norm}.
At time $t=0.01$, there are two defects at $(\pm\frac{1}{2},0)$, which cause the velocity field to form four symmetric vortices.
As time passes, the two singularities are gradually carried back to the origin and finally annihilate each other simultaneously.
We note that the evolution of the singularities observed in this example is consistent with the results of previous study \cite{23zheng}. This agreement further supports the correctness and reliability of our proposed schemes.

\begin{figure}[tp]
\begin{center}
\subfigure[]{\label{fig:2stsubfigd1}\includegraphics[scale=0.26]{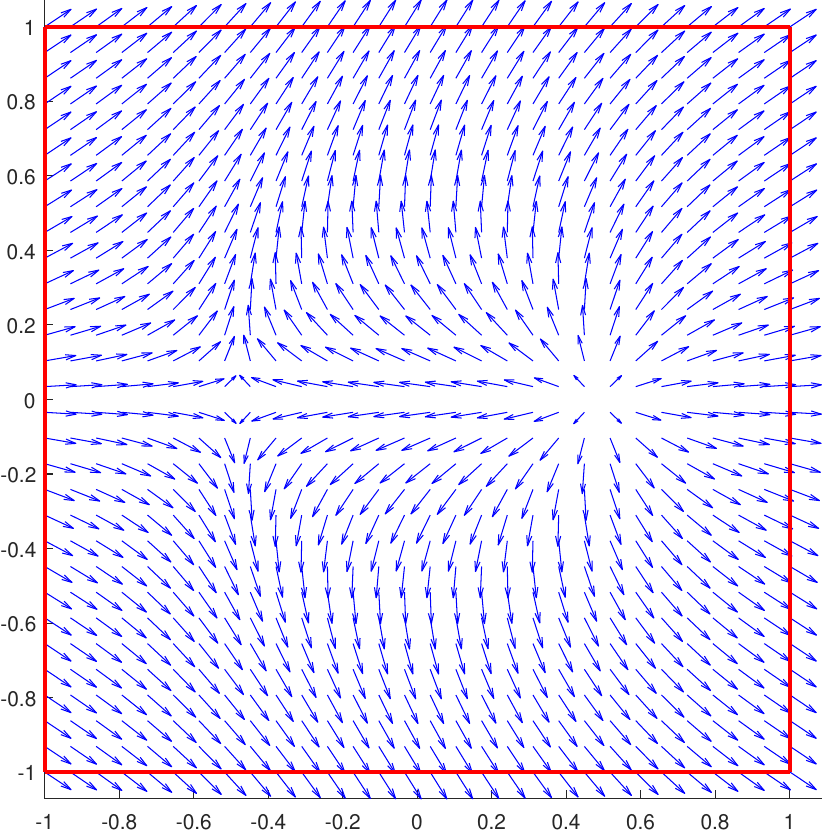}}	
\subfigure[]{\label{fig:2stsubfigd2}\includegraphics[scale=0.26]{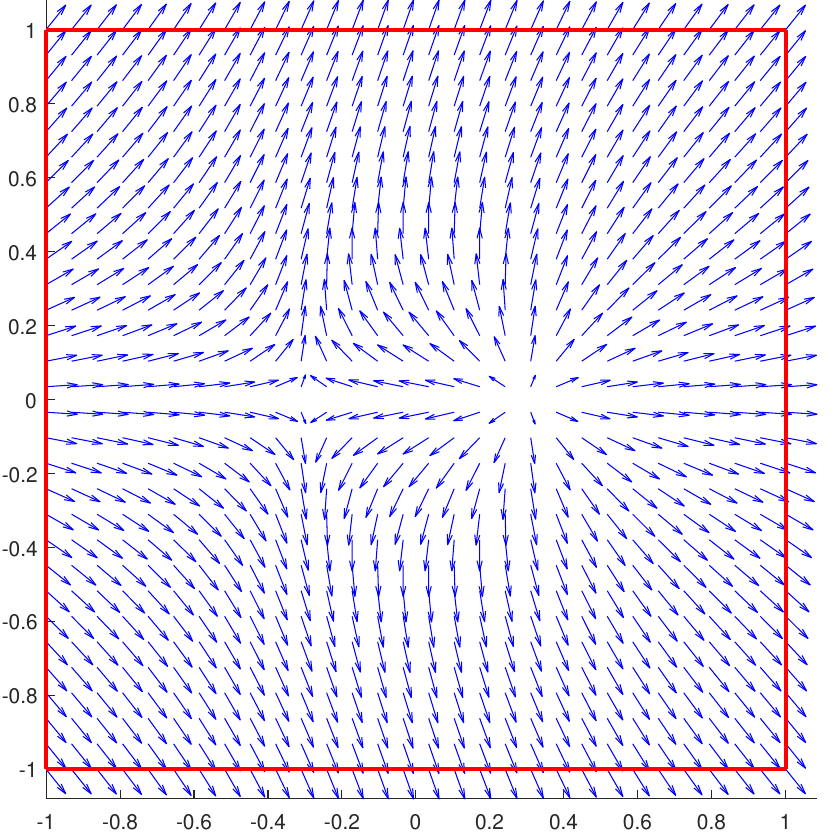}}	
\subfigure[]{\label{fig:2stsubfigd3}\includegraphics[scale=0.26]{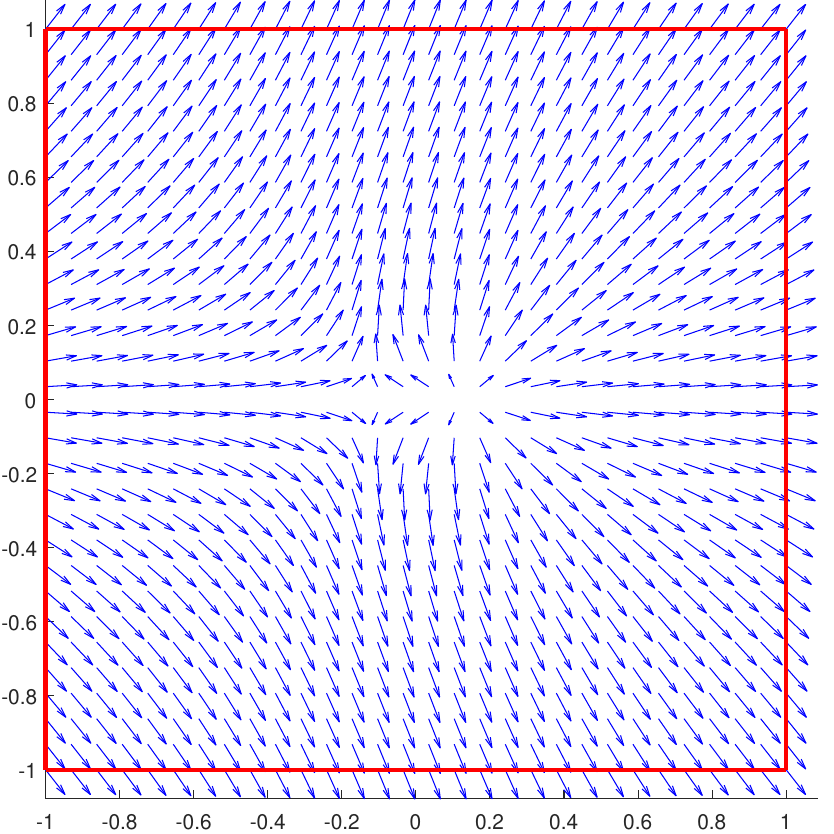}}

\subfigure[]{\label{fig:2stsubfigd4}\includegraphics[scale=0.26]{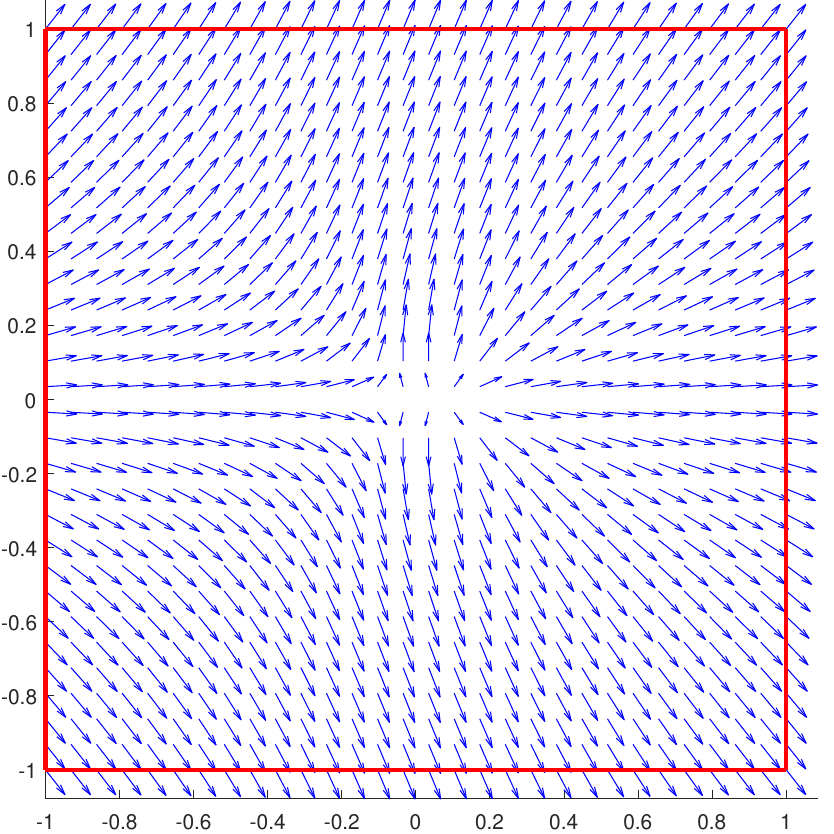}}
\subfigure[]{\label{fig:2stsubfigd5}\includegraphics[scale=0.26]{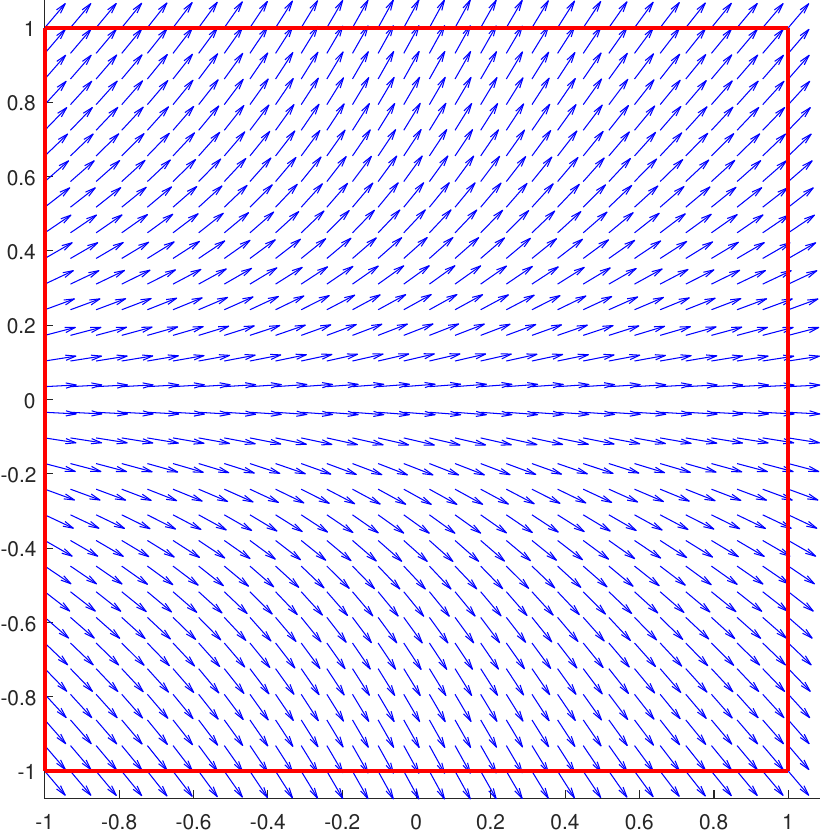}}
\subfigure[]{\label{fig:2stsubfigd6}\includegraphics[scale=0.26]{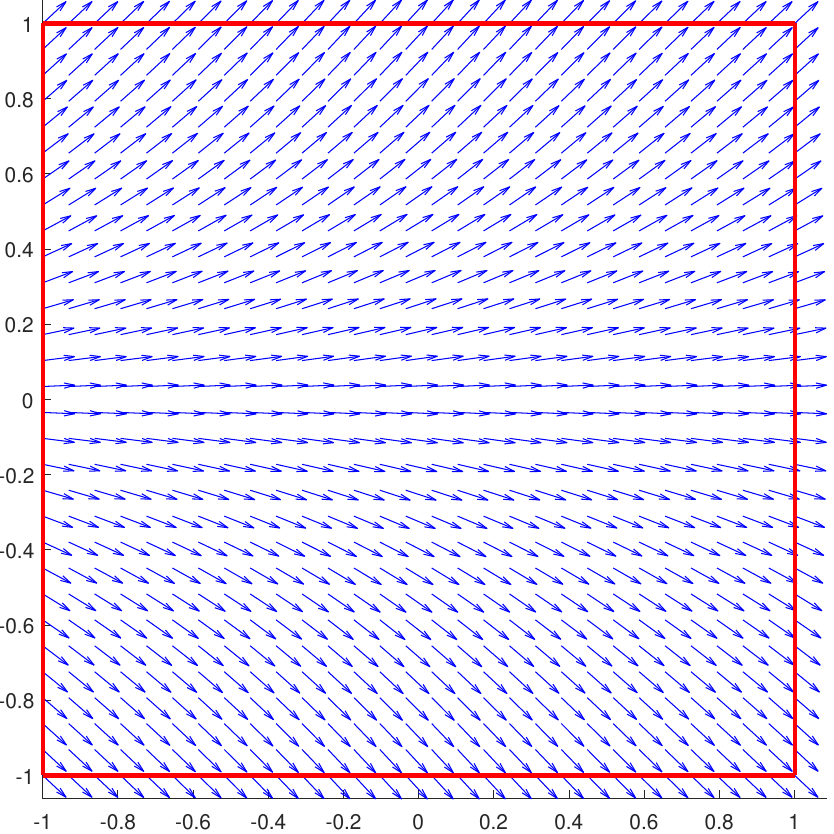}}
\caption{Example \ref{63}, snapshots of the director field $\dd$ by the PCSAV-ECT scheme at $t=$0.01, 0.2, 0.32, 0.33, 0.4, 0.5.}
\label{fig:2stsimpled}
\end{center}
\end{figure}

\begin{figure}[tp]
\begin{center}
\subfigure[]{\label{fig:2stsubfigdnorm1}\includegraphics[scale=0.26]{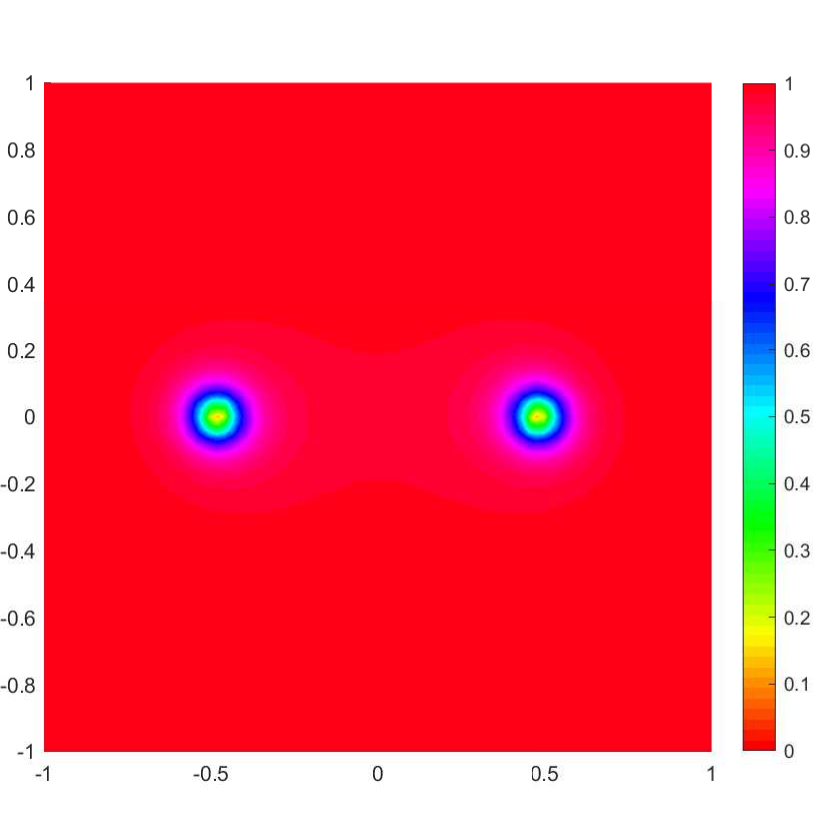}}
\subfigure[]{\label{fig:2stsubfigdnorm2}\includegraphics[scale=0.26]{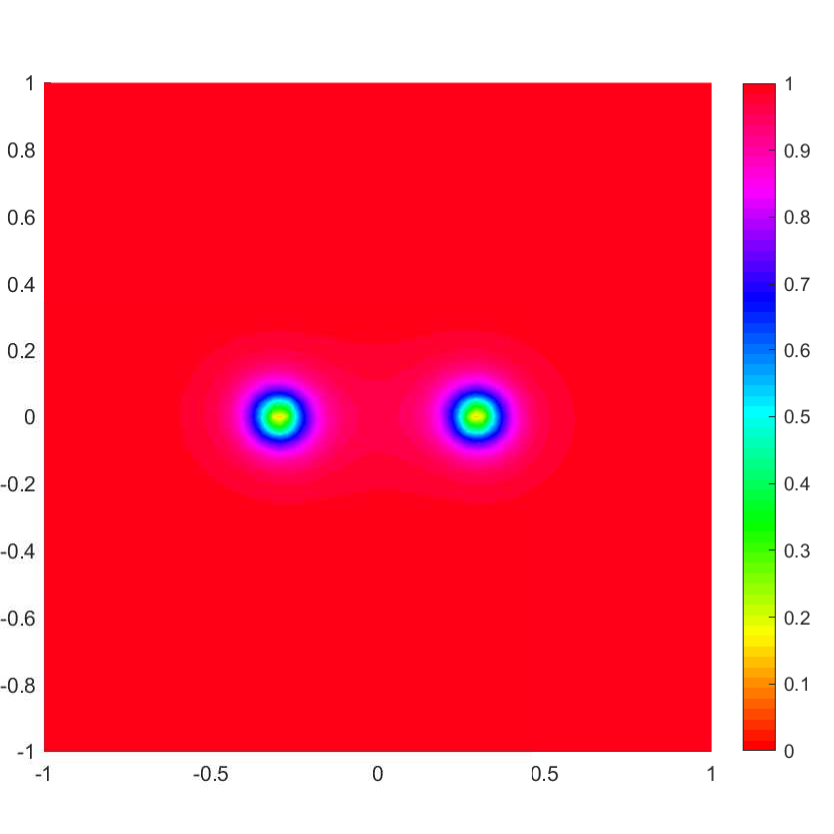}}	
\subfigure[]{\label{fig:2stsubfigdnorm3}\includegraphics[scale=0.26]{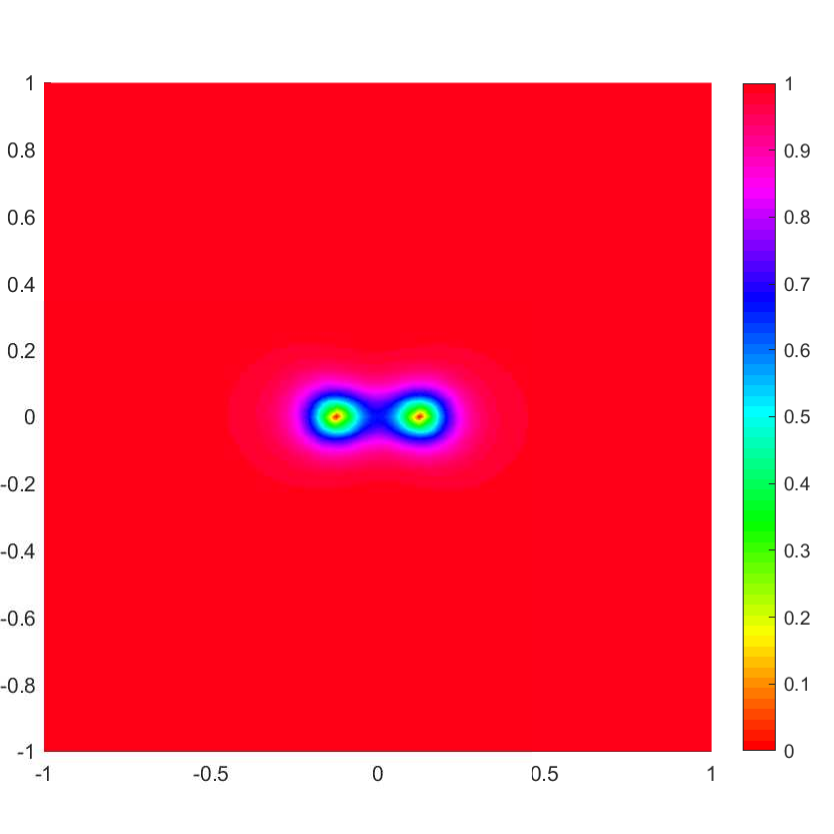}}

\subfigure[]{\label{fig:2stsubfigdnorm4}\includegraphics[scale=0.26]{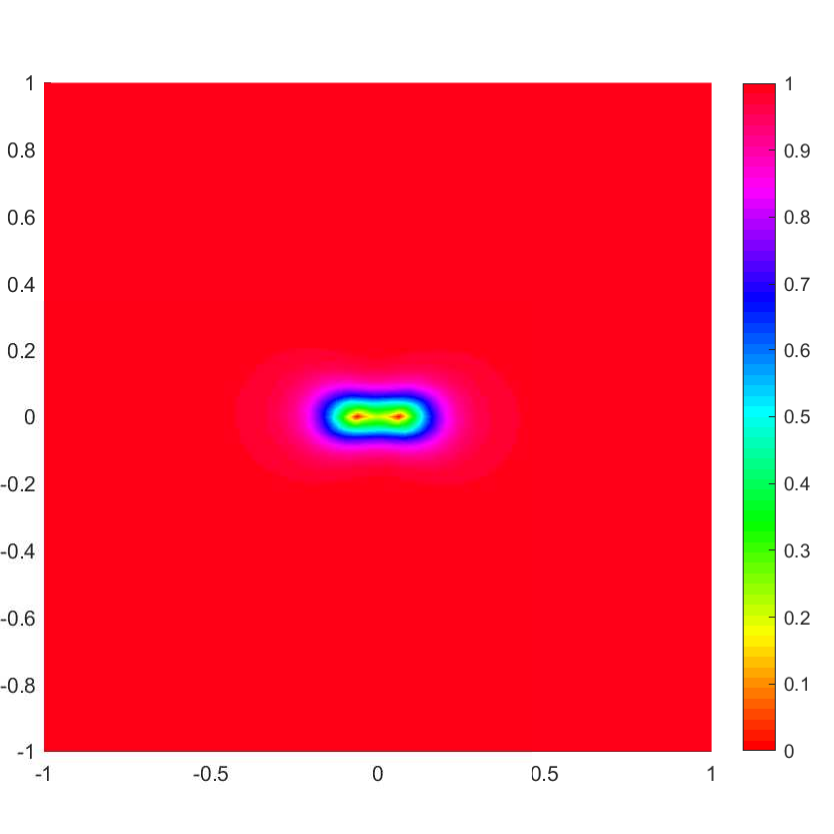}}
\subfigure[]{\label{fig:2stsubfigdnorm5}\includegraphics[scale=0.26]{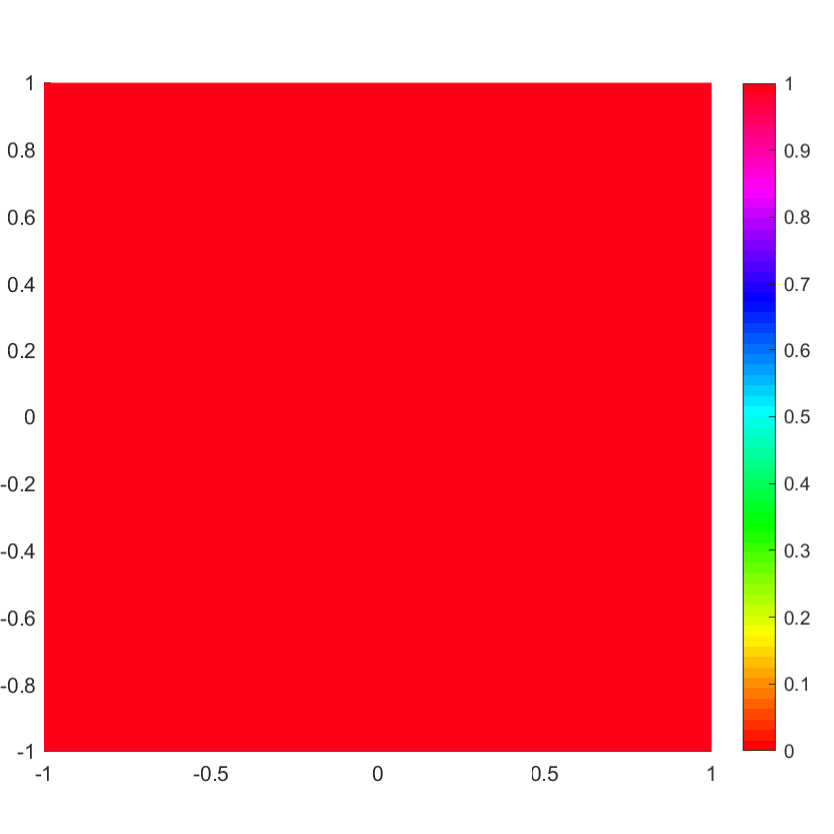}}
\subfigure[]{\label{fig:2stsubfigdnorm6}\includegraphics[scale=0.26]{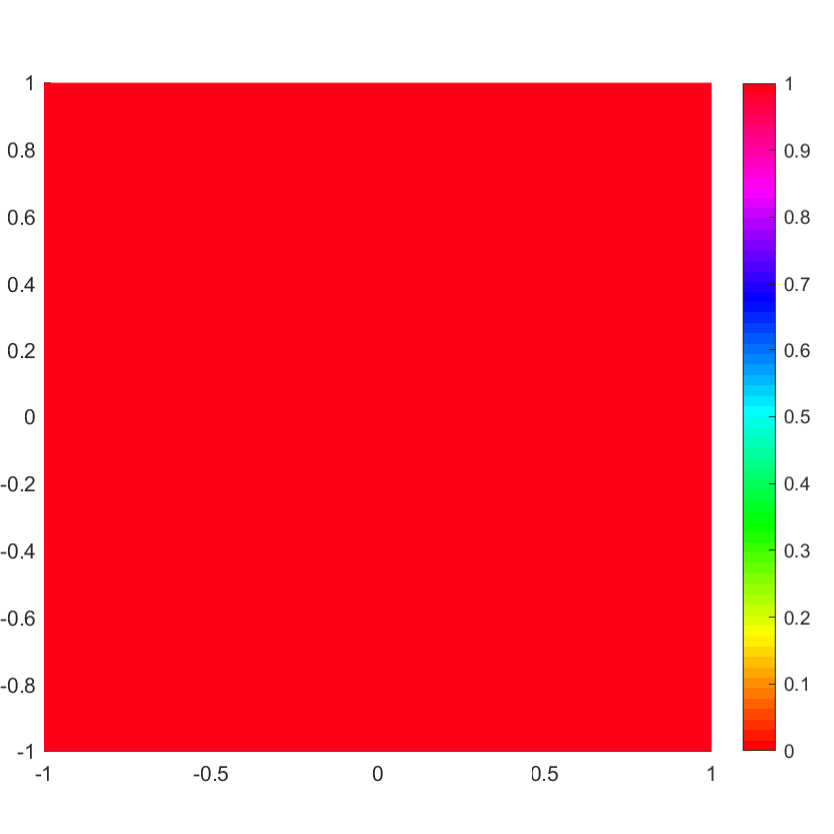}}	
\caption{Example \ref{63}, snapshots of $\lvert\dd\rvert$ by the PCSAV-ECT scheme at $t=$0.01, 0.2, 0.32, 0.33, 0.4, 0.5.}
\label{fig:2stsimpled_dnorm}
\end{center}
\end{figure}

\begin{figure}[tp]
\begin{center}	
\subfigure[]{\label{fig:2stsubfigu1}\includegraphics[scale=0.26]{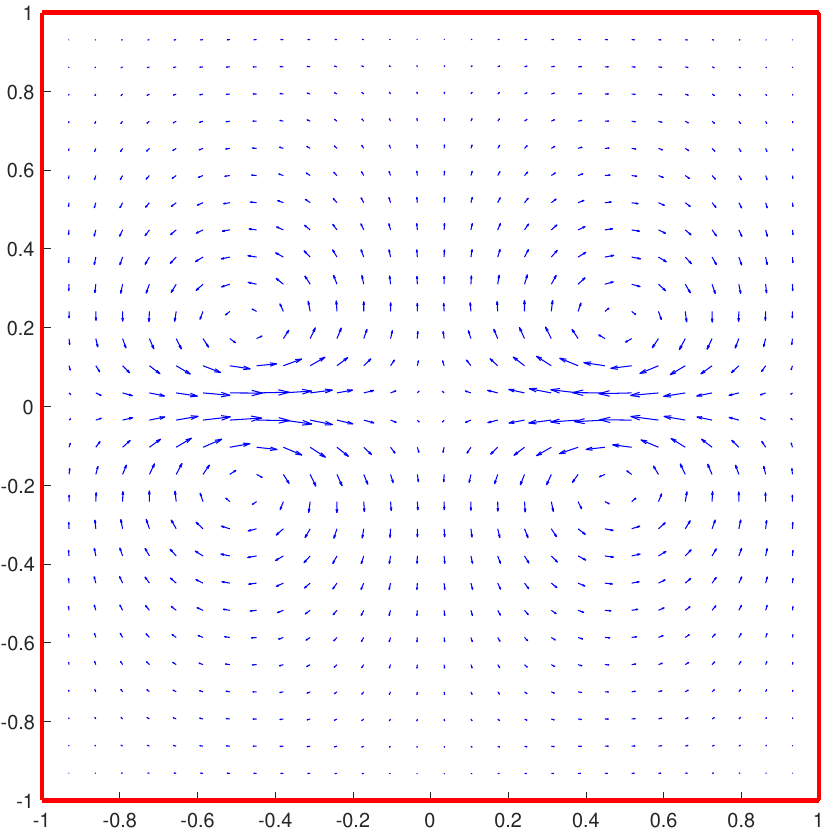}}	
\subfigure[]{\label{fig:2stsubfigu2}\includegraphics[scale=0.26]{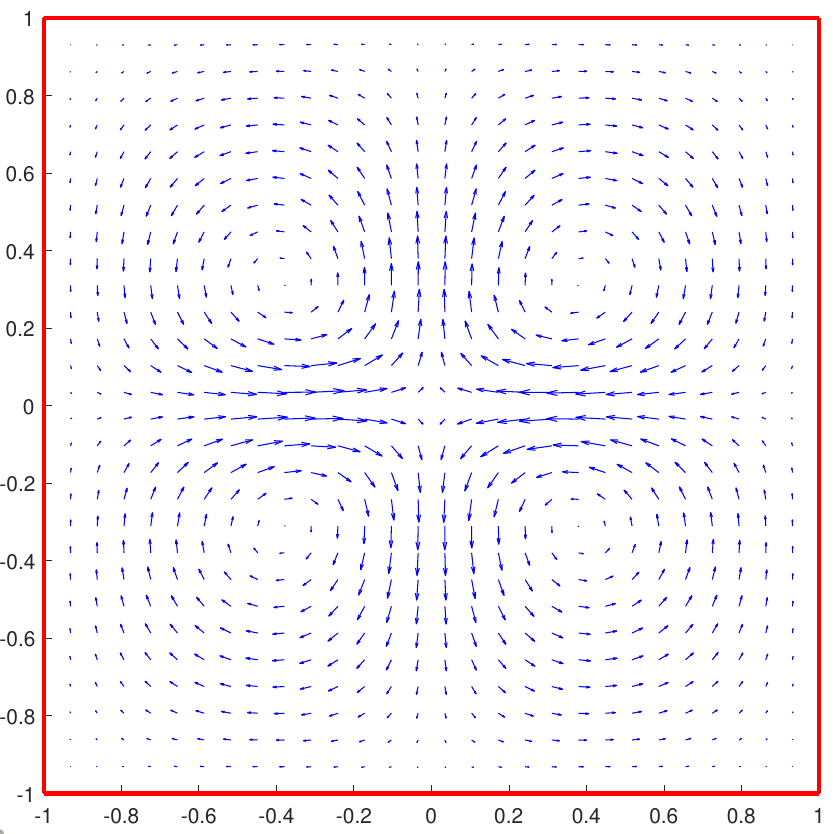}}	
\subfigure[]{\label{fig:2stsubfigu3}\includegraphics[scale=0.26]{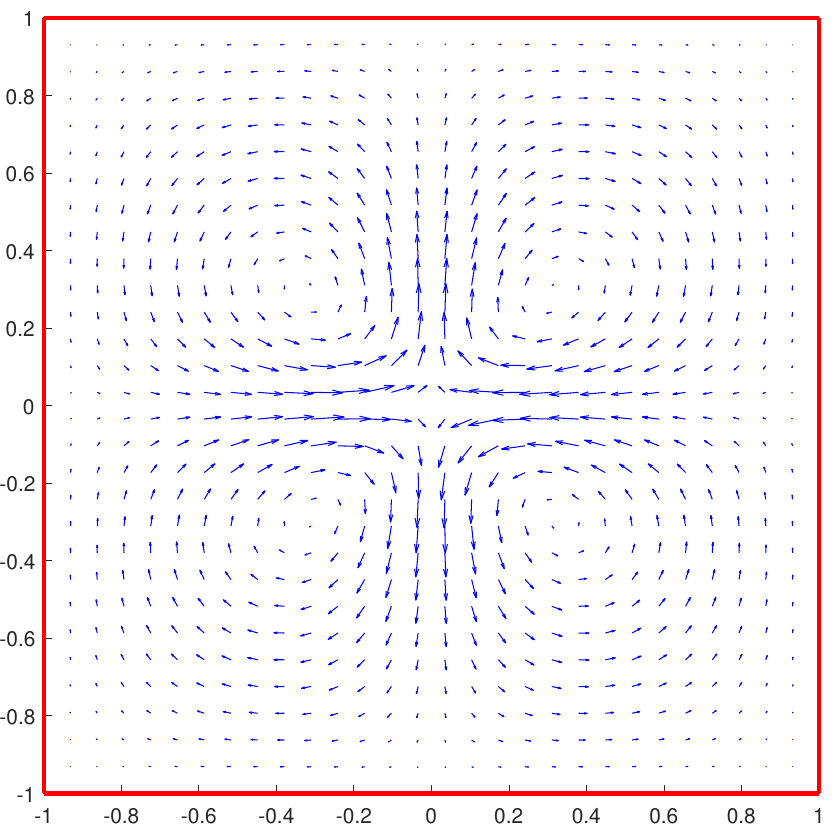}}	

\subfigure[]{\label{fig:2stsubfigu4}\includegraphics[scale=0.26]{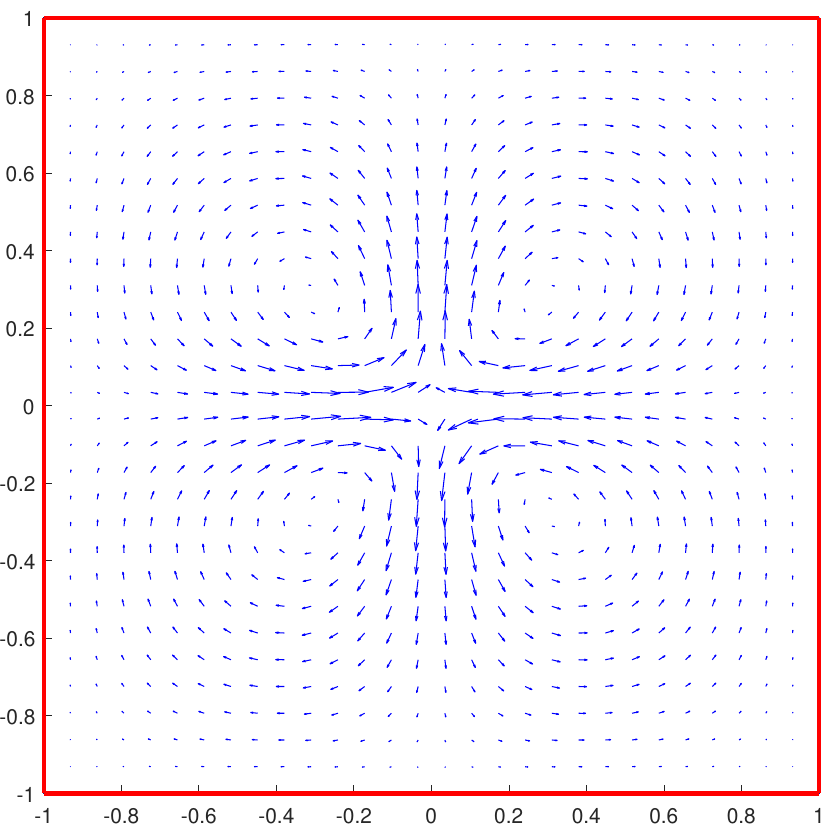}}
\subfigure[]{\label{fig:2stsubfigu5}\includegraphics[scale=0.26]{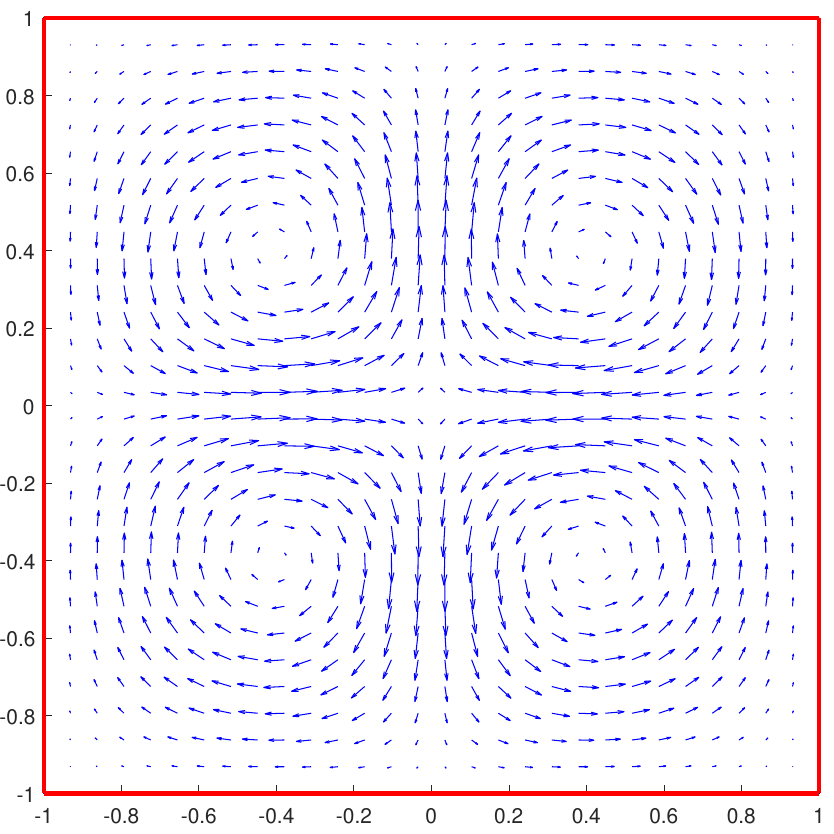}}
\subfigure[]{\label{fig:2stsubfigu6}\includegraphics[scale=0.26]{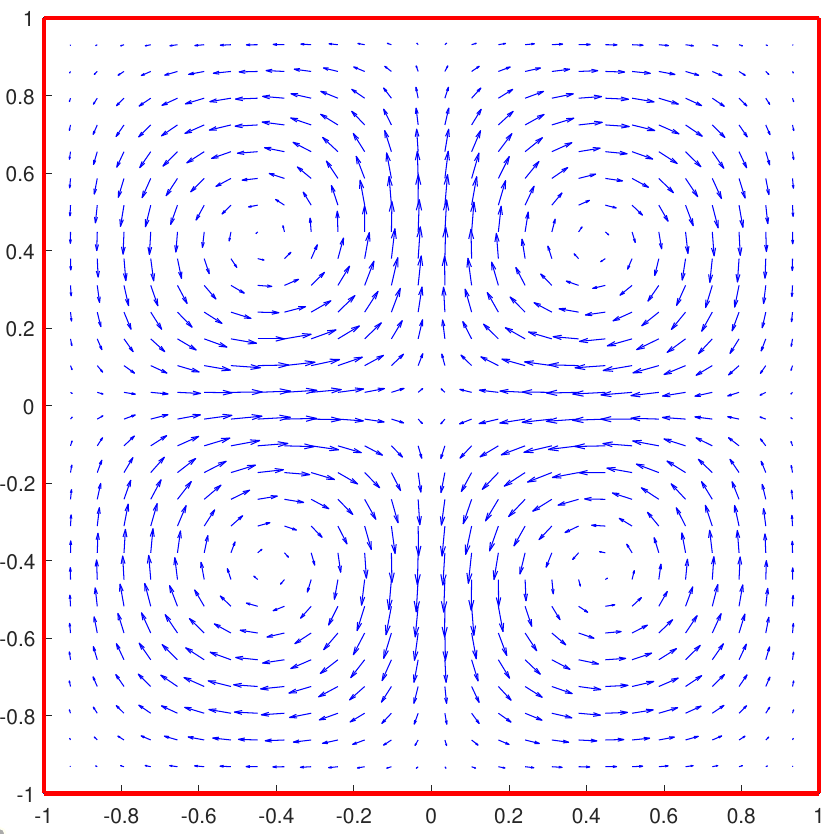}}	
\caption{Example \ref{63}, snapshots of the velocity $\uu$ by the PCSAV-ECT scheme at $t=$0.01, 0.2, 0.32, 0.33, 0.4, 0.5.}
\label{fig:2stsimpleu}
\end{center}
\end{figure}

\begin{figure}[tp]
\begin{center}	
\subfigure[]{\label{fig:2stsubfigunorm1}\includegraphics[scale=0.26]{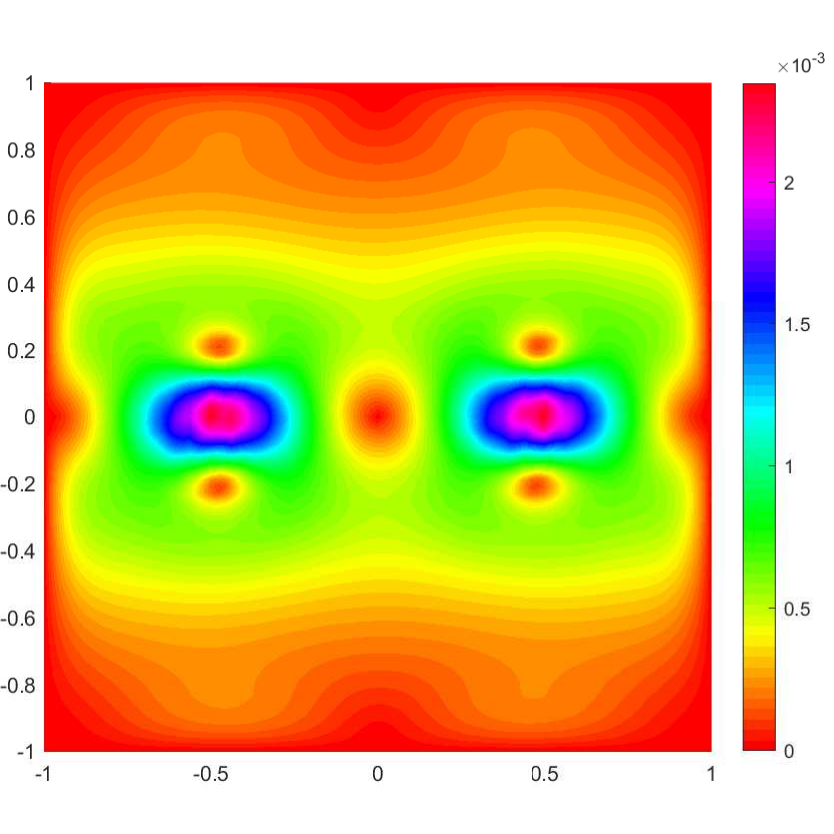}}	
\subfigure[]{\label{fig:2stsubfigunorm2}\includegraphics[scale=0.26]{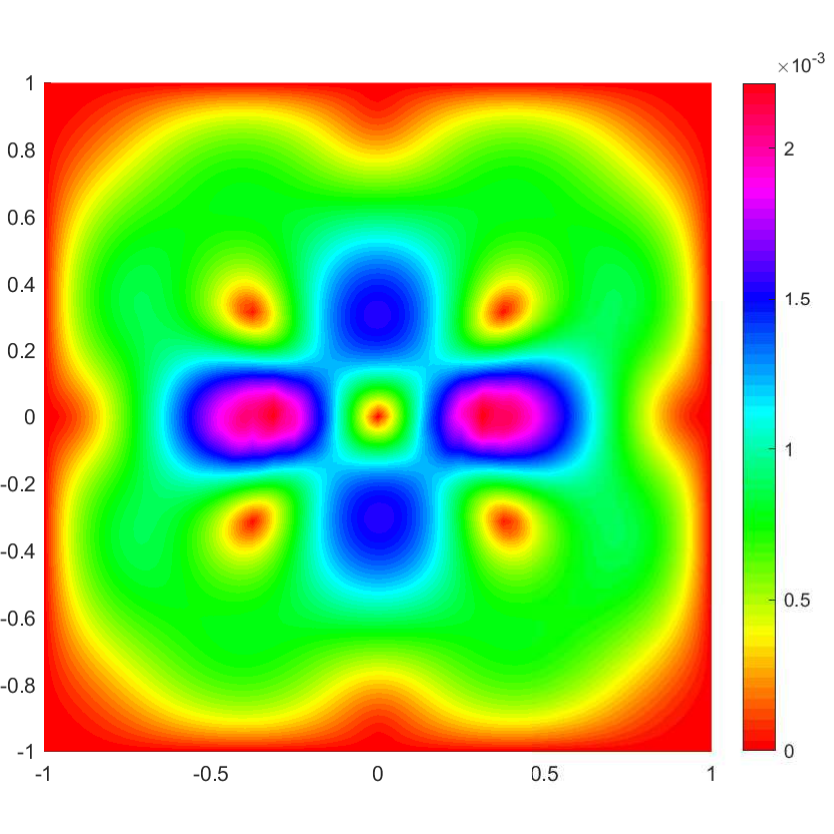}}	
\subfigure[]{\label{fig:2stsubfigunorm3}\includegraphics[scale=0.26]{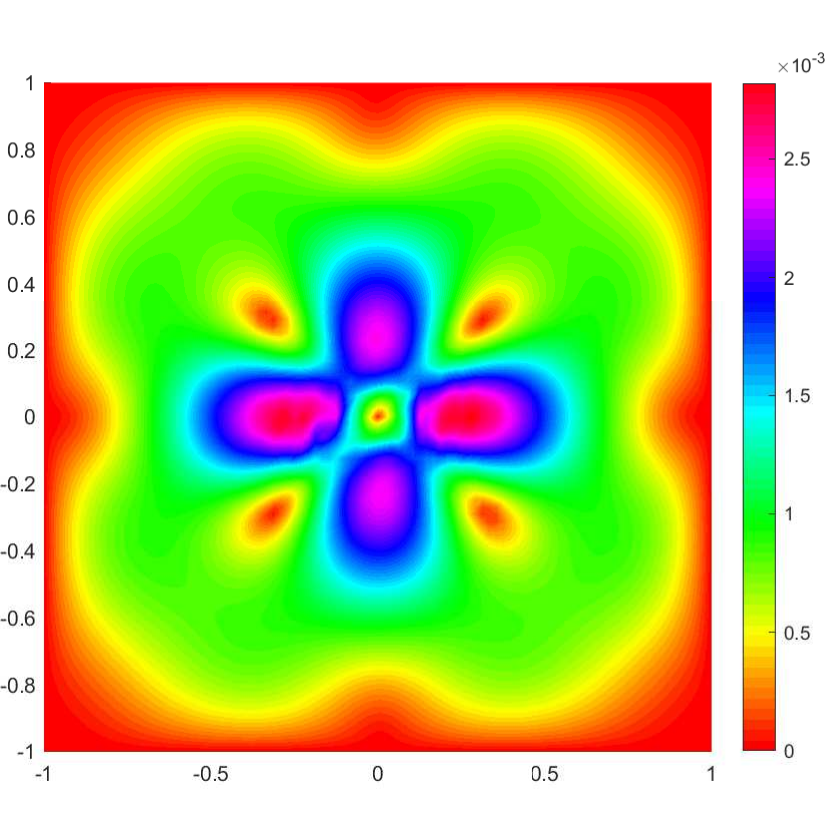}}	

\subfigure[]{\label{fig:2stsubfigunorm4}\includegraphics[scale=0.26]{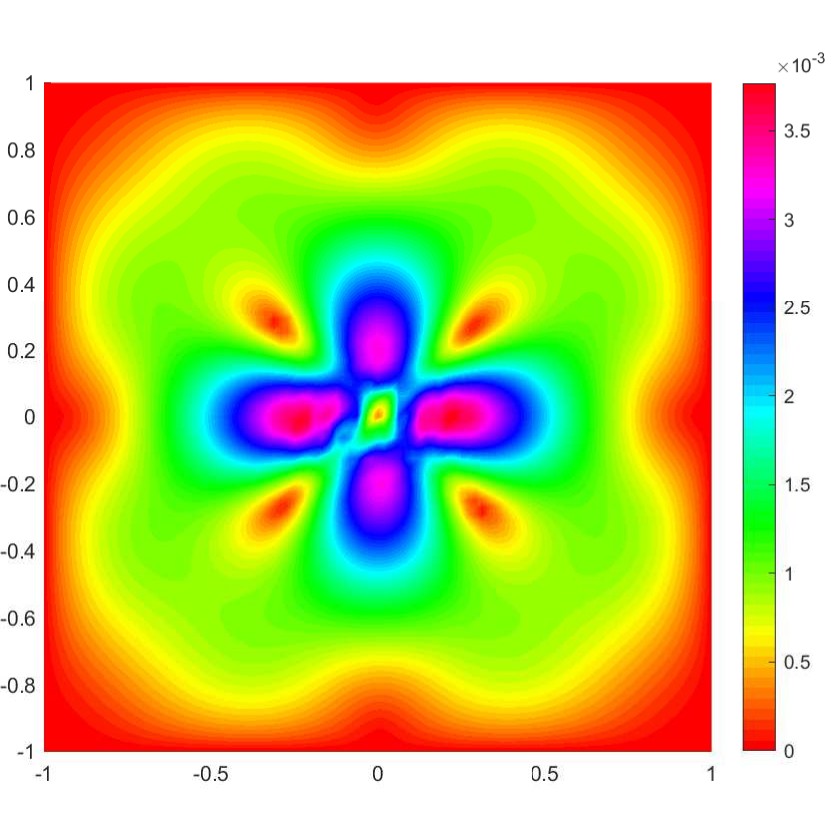}}
\subfigure[]{\label{fig:2stsubfigunorm5}\includegraphics[scale=0.26]{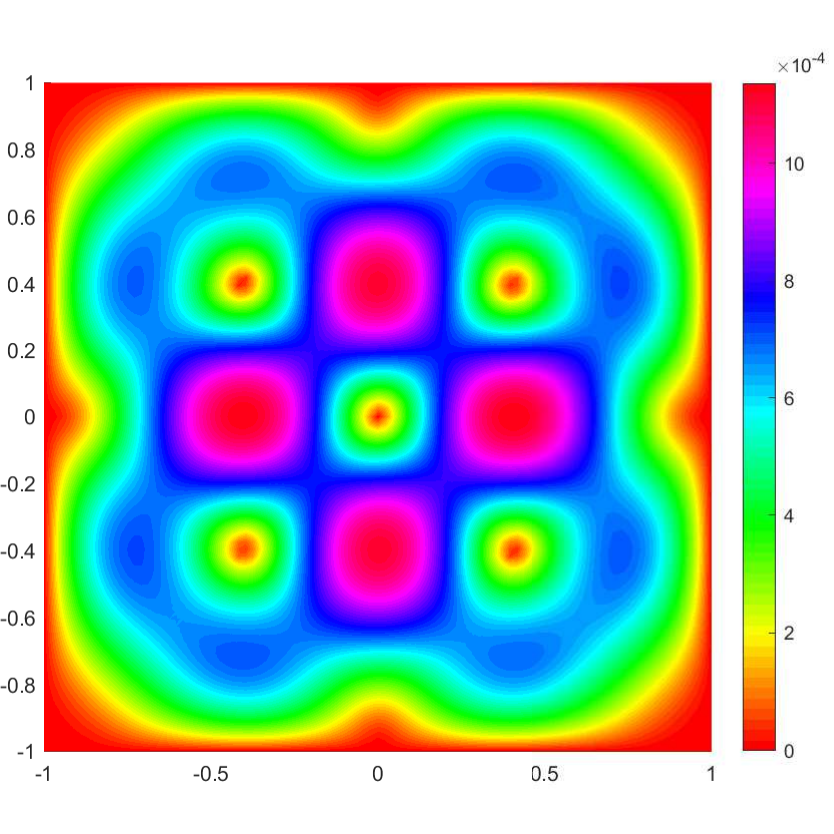}}
\subfigure[]{\label{fig:2stsubfigunorm6}\includegraphics[scale=0.26]{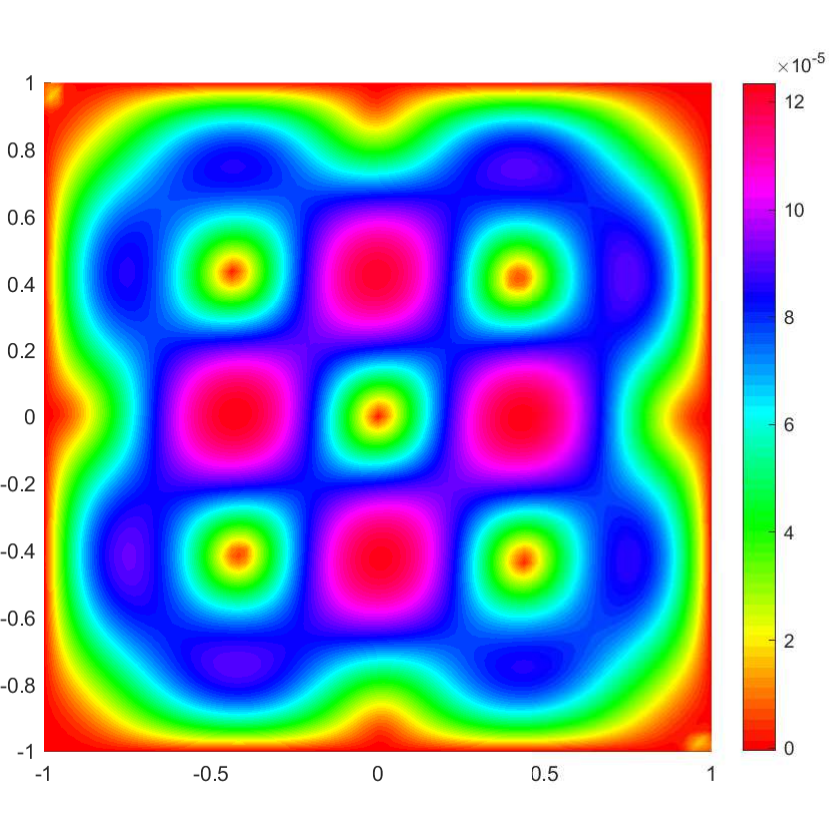}}	
\caption{Example \ref{63}, snapshots of $\lvert\uu\rvert$ by the PCSAV-ECT scheme at $t=$0.01, 0.2, 0.32, 0.33, 0.4, 0.5.}
\label{fig:2stsimpleu_norm}
\end{center}
\end{figure}

To assess the computational efficiency of the PCSAV-ECT scheme relative to the PCSAV scheme, the CPU times for Examples~\ref{621}, \ref{62}, and \ref{63} are summarized in Table~\ref{tab:cputime}. The results indicate that the PCSAV-ECT scheme consistently achieves lower computational cost, reducing the total runtime to 70.3\% on average, while preserving the accuracy and stability of the PCSAV scheme.

\begin{table}[tp]
\centering
\caption{\noindent CPU time spent on Examples~\ref{621}, Examples~\ref{62} and~\ref{63}, along with the relative percentage (PCT) compared to the PCSAV scheme.}
\label{tab:cputime}
\begin{tabular}{|c||c|c||c|c||c|c|}
\hline
\multirow{2}{*}{Scheme} 
& \multicolumn{2}{c||}{Example~\ref{621}}
& \multicolumn{2}{c||}{Example~\ref{62}} 
& \multicolumn{2}{c|}{Example~\ref{63}} \\
\cline{2-7}
& CPU time & PCT& CPU time & PCT & CPU time & PCT \\
\hline
PCSAV & 2092.005s & 100.0\% & 6056.58s & 100.0\% & 1132.41s & 100.0\% \\
PCSAV-ECT & 1524.7s & 72.9\% & 4150.19s & 68.5\% & 787.511s & 69.5\% \\
\hline
\end{tabular}
\end{table}

\section*{Conclusion}
In this work, we proposed a linear, unconditionally stable, and fully decoupled numerical scheme for the modified Ericksen-Leslie model of nematic liquid crystals. 
The scheme integrates the SAV technique to handle the nonlinear convection terms, a Lagrange multiplier method to linearize the nonlinear term, and a rotational pressure correction approach to decouple the incompressibility condition. 
The scheme updates each variable by solving a single linear system at each time step, thereby significantly improving computational efficiency while preserving the discrete energy dissipation law.

We apply the second-order backward differentiation formula (BDF2) for time discretization and rigorously establish the unconditional energy stability of the proposed scheme. 
Several numerical experiments are carried out to verify its accuracy, stability, and computational efficiency. 
The simulation results further demonstrate the capability of the scheme in capturing complex defect dynamics and the long-time behavior of the director field.
The proposed strategy provides a flexible and robust framework for simulating nematic liquid crystal flows and can be extended to more complex models, such as the full Ericksen-Leslie system or systems coupled with external electric or magnetic fields.

Although the proposed scheme achieves unconditional energy stability, linearity, and full decoupling, it does not strictly enforce the unit-length constraint $|\mathbf{d}| = 1$. 
Developing numerical schemes that rigorously preserve this constraint at the discrete level will be the subject of future research.

\section*{Acknowledgments}
Yi's research was partially supported by NSFC Project (12431014), 
Project of Scientific Research Fund of the Hunan Provincial Science and Technology Department (2024ZL5017), Postgraduate Scientific Research Innovation Project of Hunan Province (CX20240056), and Program for Science and Technology Innovative Research Team in Higher Educational Institutions of Hunan Province of China.

\end{document}